\documentclass[11pt]{article}
\usepackage{amssymb,amsmath,amsfonts,amsthm,mathtools,color}

\usepackage{mathrsfs}
\usepackage{dsfont}
 \usepackage{bm} 

\usepackage{color} 

\usepackage{url} 

\usepackage[overload]{empheq}

\usepackage{comment} 
\usepackage[title,titletoc,header]{appendix}
\usepackage{graphicx}
\usepackage{subcaption}
\usepackage{float} 
\usepackage{dsfont} 

\usepackage{paralist}

\usepackage{cases}

\usepackage{tikz}
\usetikzlibrary{arrows,positioning,shapes.geometric}

\usepackage{algorithm}
\usepackage{algpseudocode}

\usepackage{multirow}

\graphicspath{ {./figures/} }

\usepackage{indentfirst}
\usepackage{multicol}
\usepackage{booktabs}
\usepackage{url}
\usepackage[outdir=./]{epstopdf} 

\usepackage[shortlabels]{enumitem}

\usepackage{hyperref}
\hypersetup{
    colorlinks=true, 
    linktoc=all,     
    linkcolor=blue,  
}
\numberwithin{equation}{section}

\newtheorem{theorem}{Theorem}[section]
\newtheorem{lemma}[theorem]{Lemma}
\newtheorem{proposition}[theorem]{Proposition}
\newtheorem{assumption}{Assumption}

\theoremstyle{definition}

\theoremstyle{remark}
\newtheorem{remark}{Remark}[section]

\def \Th{\Theta}

\def\eps{\varepsilon}

\def\E{{\bf E}}

 \def\tv{\mathrm{TV}}

\def\cA{\mathcal{A}}

\def\cP{\mathcal{P}}

\def\cX{\mathcal{X}}

\def\d{{\mathrm{d}}}
\def\mv{{\mathrm{MV}}}
\def\st{{\mathrm{ST}}}
\def\var{{\mathrm{Var}}}

\def\E{{\mathbb{E}}}

\def\N{{\mathbb{N}}}
\def\P{\mathbb{P}}

\def\R{{\mathbb R}}

\newcommand{\citep}{\cite}
\newcommand{\citet}{\cite}

\DeclareMathOperator*{\logit}{logit}
\DeclareMathOperator*{\softmax}{softmax}

\newcommand{\lc}
{\mathrel{\raise2pt\hbox{${\mathop<\limits_{\raise1pt\hbox
{\mbox{$\sim$}}}}$}}}

\newcommand{\gc}
{\mathrel{\raise2pt\hbox{${\mathop>\limits_{\raise1pt\hbox{\mbox{$\sim$}}}}$}}}

\newcommand{\ec}
{\mathrel{\raise2pt\hbox{${\mathop=\limits_{\raise1pt\hbox{\mbox{$\sim$}}}}$}}}

\def\bb{\begin{equation}} \def\ee{\end{equation}}
\def\bbn{\begin{equation*}} \def\een{\end{equation*}}

\def\beqn{\begin{eqnarray}}  \def\eqn{\end{eqnarray}}

\def\beqnx{\begin{eqnarray*}} \def\eqnx{\end{eqnarray*}}

\def\bn{\begin{enumerate}} \def\en{\end{enumerate}}

\def\bd{\begin{description}} \def\ed{\end{description}}



\begin{document}

\title{Model-free policy gradient for discrete-time mean-field control}

\author{
Matthieu Meunier\thanks{Mathematical Institute, University of Oxford, Oxford OX2 6GG, UK ({\tt matthieu.meunier@maths.ox.ac.uk}). This author gratefully acknowledges the support of the Paul Lévy initiative, which enabled a four-month visit to École Polytechnique  during which part of this work was completed. This author is also supported by the EPSRC Centre for Doctoral
Training in Mathematics of Random Systems: Analysis, Modelling and Simulation (EP/S023925/1).
} 
\and
Huyên Pham\thanks{CMAP, École Polytechnique, Palaiseau 91120, France ({\tt huyen.pham@polytechnique.edu}). 
This author is supported by the Chair ``Financial Risks'', by FiME (Laboratory of Finance and Energy Markets), and the EDF–CACIB Chair ``Finance and Sustainable Development''.
}
\and
Christoph Reisinger\thanks{Mathematical Institute, University of Oxford, Oxford OX2 6GG, UK ({\tt  
christoph.reisinger@maths.ox.ac.uk}).
This author is supported by the EPSRC grant EP/Y028872/1, Mathematical Foundations of Intelligence: An ``Erlangen Programme'' for AI.
} 
}

\date{}
\maketitle

\begin{abstract}
We study model-free policy learning for discrete-time mean-field control (MFC) problems
with finite state space and compact action space.
In contrast to the extensive literature on 
value-based methods for MFC, policy-based approaches remain largely unexplored due to the intrinsic dependence of transition kernels and rewards on the evolving population state distribution, which
prevents the direct use of likelihood-ratio estimators of policy gradients from classical single-agent reinforcement learning.
We introduce a novel perturbation scheme on the state-distribution flow and prove that the gradient of the resulting perturbed value function converges to the true policy gradient as the perturbation magnitude vanishes.
This construction yields  a fully model-free estimator based solely on simulated trajectories and an auxiliary estimate of the sensitivity of the state distribution. 
Building on this framework, we develop MF-REINFORCE, a model-free policy gradient algorithm for MFC, and establish  explicit quantitative bounds on its bias and mean-squared error. 
Numerical expe\-riments on representative mean-field control tasks demonstrate the effectiveness of the proposed approach.  \end{abstract}

\medskip


%

\vspace{5mm}  


%
\medskip

\section{Introduction}\label{sec:intro}

Designing algorithms which optimally control large populations of cooperative agents has recently become an important problem due to the breadth and depth of applications. In the face of the complexity of modelling individual interactions between each pair of agents, mean-field control provides a rigorous model particularly suited for problems with many agents interacting in a symmetrical way: instead of considering all possible interactions between every pair of players, one lets the number of agents go to infinity and tries to solve a limiting problem for which it suffices to study the interaction between a typical agent and the probability distribution of its state. 
An important and practical research direction for such decision-making problems consists in model-free reinforcement learning (RL) algorithms.
RL algorithms are based on two elementary approaches: \textit{value-based} learning algorithms, such as $Q$-learning (\citet{watkins1992q}), where one tries to learn a map which allows one to compare the optimal expected reward for different actions and subsequently learn a policy from this map, and \textit{policy-based} algorithms, also known as policy gradient methods, such as REINFORCE (\citet{williams1992simple}), which directly try to optimize the value function with respect to the parameters of a policy. These two approaches are of great importance and are combined in modern deep reinforcement learning algorithms.
The study of such approaches tailored to mean-field control problems has recently received more attention, with many works developing value-based algorithms.
However, there is a notable lack of study of model-free policy-based methods for this setting in the literature.

This paper fills this gap by introducing the first policy gradient algorithm for mean-field control problems in discrete time. As in the classical REINFORCE algorithm for single-agent problems, our algorithm is entirely model-free and can be performed by simply observing trajectories of the system.

\subsection{Outline of Main Results}

In this section, we provide an informal description of the problem at hand and present our main contributions.
We consider a mean-field control problem in discrete time without common noise and time horizon $T \in \N$. We assume that the state space $\cX$ is finite and that the action space $\cA$ is compact and Polish.
Given an admissible control $\alpha = (\alpha_t)_{t = 0}^{T-1}$, the $\cX$-valued state follows the dynamics 
\[
X_0 = \xi, \quad X_{t + 1} \sim P( \cdot \mid X_t, \alpha_t, \mathbb P_{X_t}),
\]
where $\mathbb P_{X_t} \in \mathcal P (\mathcal X)$ denotes the law of $X_t$. We denote by $P(\cdot \mid x,a,\mu)$ the underlying probability kernel. Consider a running reward $r: \mathcal X \times \mathcal A \times \mathcal P (\mathcal X) \rightarrow \R$ and a terminal reward $g: \mathcal X \times \mathcal P (\mathcal X) \rightarrow \R$, which are both assumed to be measurable and bounded.
The goal is to maximize the following objective, for a given initial condition $\xi$ $\sim$ $\mu_0$, 
\[
\mathbb E_{\xi \sim \mu_0} \left[\sum_{t = 0}^{T-1} r(t,X_t, \alpha_t, \mathbb P_{X_t}) + g(X_T, \mathbb P_{X_T})\right],
\]
over a set of admissible controls. For now, we consider randomized controls of feedback form, meaning that the set of admissible controls is parametrized by 
\[
\Pi \coloneq \{ \pi: \{0,\ldots,T-1\} \times \cX \times \mathcal P (\mathcal X) \rightarrow \mathcal P (\mathcal A) \text{ measurable} \}.
\]
Hence a control $\alpha$ is admissible if there exists $\pi \in \Pi$ such that $\alpha_t \sim \pi (\cdot \mid t,X_t, \mathbb P_{X_t})$ for all $t \in \{0, \ldots, T-1\}$.
Therefore, we seek to maximize
\begin{equation}
\label{eq:value-function}
V^\pi (\mu_0) \coloneq \mathbb E^{\pi}_{\xi \sim \mu_0} \left[ \sum_{t = 0}^{T - 1} r(X_t, \alpha_t, \mathbb P_{X_t}) + g(X_T, \mathbb P _{X_T})\right],
\end{equation}
where $\mathbb E^{\pi}$ refers to $\alpha = (\alpha_t)$ being generated by $\pi$.

Our goal is to develop a model-free policy gradient algorithm to solve this mean-field control problem. Consider a set of parametrized policies $\{\pi_{\theta}: \theta \in \Theta \}$ with $\Th \subseteq \R^D$ given, e.g., by a neural network with trainable parameters. 
Then we want to find an estimator for $\nabla_{\theta} V^{\pi_{\theta}}(\mu_0)$ that can easily be computed in a reinforcement learning setting, e.g., by generating trajectories of $X$ with a simulator.
In continuous time, this problem has been studied by \citet{frikha2025actor}, where the It\^o formula along flows of measures is used to write a policy gradient of the value function that can be estimated in a model-free setting.
Our main contributions are as follows:
\begin{itemize}
    \item We rigorously derive a policy gradient formula for the value function in Proposition \ref{proposition:exact-policy-gradient-formula}. This result highlights the computational challenges stemming from the mean-field dependence of the transition kernel and the reward functions. In particular, contrary to the single-agent problem, it is not possible to directly obtain a model-free estimator of the gradient using a likelihood ratio trick.
    \item We introduce a perturbation scheme, along with a perturbed value function, and show in Theorem \ref{theorem:convergence-perturbed-gradient} that the gradient of the perturbed value function converges to the true value function as the size of the perturbation goes to $0$. In particular, the gradient of the perturbation is amenable to model-free estimation, but requires an auxiliary estimation procedure of the gradient (of the logits) of the state distribution.
    \item We present a model-free algorithm to estimate the gradient of the value function, which we call MF-REINFORCE, based on the formula for the gradient of the perturbed value function in Algorithm \ref{alg:reinforce}. We provide a detailed analysis and quantitative bounds on the bias and mean squared error of the newly introduced estimator in Theorems \ref{theorem:main-bias-result} and \ref{theorem:main-mse-result} respectively, based on the size of the perturbation and number of trajectories sampled.
    \item We present numerical experiments of MF-REINFORCE on a few problems in Section \ref{section:numerical-experiments}.
\end{itemize}

\subsection{Related Works}

\paragraph{Policy Gradient Methods in Multi-Agent Problems}

Previous works tackle the challenge of learning optimal policies in problems with many agents, both in the mean-field and finite-player settings.
The work of \citet{carmona2019linear} investigates policy gradient methods for linear-quadratic MFC in discrete time.
The recent work of \citet{soner2025learning} studies a (model-based) policy learning algorithm for MFC by simulating a corresponding $N$-particle system.
Actor--critic algorithms for asymptotic mean-field games and control problems have been introduced and analyzed by \citet{angiuli2023deep} and \citet{fouque2025convergence}. 
Surprisingly enough, continuous-time problems have received more attention (see, e.g., \citet{fouque2020deep}), due to the explicit policy gradient formula one can obtain via Itô's formula along flows of measures (\citet{frikha2025actor}), where policy gradient algorithms have been developed using moment neural networks (\citet{pham2025actor}), and for linear-quadratic problems (\citet{wang2021global, frikha2024full}).
From the perspective of multi-agent reinforcement learning (MARL), \citet{yang2018mean} use mean-field symmetry to design policy gradient methods, \citet{mondal2022approximation} propose a natural policy gradient method for heterogeneous problems based on mean-field control.
There are many works studying policy gradient algorithms in the context of games, that is when agents compete without a central planner, both in the finite player (see, e.g., \citet{hambly2023policy, giannou2022convergence}) and mean-field (see, e.g., \citet{subramanian2019reinforcement, fu2020actor, guo2023general, liang2024actor}) settings.

\paragraph{Value-Based Methods in the Mean-Field Setting}
Many previous works have studied value-based methods, such as $Q$-learning, for MFC. This line of work stems from establishing the dynamic programming principle for the MFC problem (\citet{lauriere2016dynamic, pham2017dynamic, gu2023dynamic}), or an equivalent mean-field Markov decision process (MFMDP) (\citet{motte2022mean,bauerle2023mean, carmona2023model}). Given that equivalent problem, which typically consists of a Markov decision process on augmented state and action spaces, one can define an augmented or integrated $Q$-function and perform $Q$-learning (\citet{angiuli2022unified, gu2023dynamic, carmona2023model}), either by discretizing the augmented spaces or using function approximation. This line of work has proven useful to cooperative multi-agent reinforcement learning (\citet{gu2021mean, gu2025mean}). Value-based methods have also been explored in continuous-time mean-field control settings (\citet{wei2025continuous}).
Our approach is different from these previous works since we do not rely on an equivalent MFMDP, and we do not directly learn the optimal value function.

\paragraph{Perturbation Methods in Reinforcement Learning}
One of the most popular perturbation methods in reinforcement learning consists in adding entropy regularization to the cost functional (see, e.g., \citet{ziebart2008maximum,fox2016tamingnoisereinforcementlearning,schulman2017proximalpolicyoptimizationalgorithms,haarnoja2018sac,geist19atheory}), which has proven effective in mean-field settings (\citet{cui2021approximately, guo2022entropy, anahtarci2023q}).
In our work, we take a different approach by directly perturbing the state probability distribution. More generally, perturbation methods are a key concept in model-free / zeroth-order optimization, and such methods have previously been used in a RL context.
Early works such as simultaneous perturbation stochastic approximation (SPSA) (\citet{spall1992multivariate}) propose to optimize a loss function by stochastic perturbation of the parameters, an idea further developed in parameter-exploring policy gradient (\citet{sehnke2010parameter}) and in the context of black-box policy search (\citet{salimans2017evolution}). These approaches typically perturb the parameters of the policy, rather than the underlying probability distributions.
Distributional perturbations have also appeared in the analysis of robustness in RL (\citet{liu2022distributionally, wang2024samplecomplexityvariancereduceddistributionally}). However, to the best of our knowledge, no existing work investigates perturbations of the state-distribution flow itself, nor uses such perturbations to derive a model-free, policy-gradient estimator tailored to mean-field control.

\subsection{Notations and Paper Structure}

For a compact Polish space $E$, the space of probability distributions $\cP (E)$ is equipped with the total variation distance denoted by $d_\tv(\cdot, \cdot)$.
For a vector $x = (x_1, \ldots, x_D) \in \R^D$, we write $\| x \|_{\infty} = \sup_{i =1, \ldots, D } |x_i|$.
For a continuous bounded function $f: \mathcal E \rightarrow \R^D$, where $\mathcal E$ is a topological space, we write $\|f\|_\infty = \sup_{e \in \mathcal E} \| f(e) \|_\infty$.
For a matrix $A \in \R^{m \times n}$, we denote by $\| A \|$ its operator norm induced by the $1$-norm on $\R^m$ and $\R^n$, i.e., $\| A \| = \max_{1 \leq j \leq n } \sum_{i = 1}^m |A_{ij}|$.
In Section \ref{section:pg-mfrl}, we state the precise working assumptions and present the perturbation technique, we derive the gradient of the perturbed value function and state a convergence result towards the gradient of the true value function.
In Section \ref{section:pg-algorithm}, we present the algorithmic details to implement an estimator of the gradient of the perturbed value function and we give quantitative bounds on the bias and mean-squared error (MSE) of the resulting estimator. 
Section \ref{section:numerical-experiments} is dedicated to numerical experiments. The proofs of the results stated in Sections \ref{section:pg-mfrl} and \ref{section:pg-algorithm} are given in Section \ref{section:proofs}.

\section{Policy Gradient Approximation via Perturbation for Mean-Field Reinforcement Learning}
\label{section:pg-mfrl}

In this section, we first present the logit representation of discrete measures as an important tool for dealing with the measure dependence in the subsequent computations. 
Then, we state the main working assumption and present the exact policy gradient formula.
Finally, we show convergence of a perturbed, computable gradient to the true gradient.

Throughout this section and the rest of the paper, we consider a discrete-time mean-field control problem with state space $\cX$, action space $\cA$, running reward $r(x,a,\mu)$, finite horizon $T \in \N$, terminal reward $g(x, \mu)$ and transition kernel $P(\d x^\prime \mid x, a, \mu)$. For a given initial condition $\xi$ such that $\mathcal L (\xi) = \mu$, the goal is to maximize the value function $V^\pi(\mu)$ defined in \eqref{eq:value-function} over $\pi \in \{ \pi_\theta: \theta \in \Theta\}$ where $\Theta \subseteq \R^d$ is the set of admissible parameters.

\subsection{Parametrization of State Distributions by Logits}
\label{subsection:logits}

Defining a perturbation scheme on $\cP(\cX)$ is a priori non-trivial since $\cP(\cX)$ is not a vector space.
Therefore, we choose to parametrize probability distributions by the vector of log-probabilities, also known as logits. Then, one can perturb the logits linearly and map the perturbed logits back to a probability distribution with a softmax operation.
Assume that  $| \cX | = d < \infty$. Then we can view $\mathcal P (\cX)$ as
\[
\{ \mu \in \R^d: \quad \mu_i \geq 0\:\: \forall i, \quad \sum_{i = 1}^d \mu_i = 1\}.
\]
For $l \in \R^{d}$, we can define the following probability vector 
$$
\softmax (l) \coloneq \left( \frac{e^{l_1}}{\sum_{i=1}^{d} e^{l_i}}, \ldots, \frac{e^{l_{d}}}{\sum_{i=1}^{d} e^{l_i}}  \right) \in \R^{d}.
$$
 The inverse transformation, for a probability vector $\mu \in \mathcal{P}(\cX)$ such that $\mu_i > 0$ for all $i$, is
$$
\logit(\mu) = \left(\log \mu_1, \ldots, \log \mu_{d} \right).
$$
Hence we can parametrize any distribution on $\cX$ with non-zero mass everywhere by a vector of logits in $\R^{d}$.
 Let $ \cP(\cX)^* \coloneq \{ \mu \in \mathcal P (\cX): \mu(x) > 0, \: \forall x \in \cX \}$. 

Since we are interested in differentiating the value function, we need to make the notion of derivative in the measure argument precise. We make the following set of standing assumptions throughout the paper.

\begin{assumption}
\label{assumption:finite-state-space}
The state space $\cX = \{x^{(1)}, \ldots, x^{(d)} \}$ is finite and the action space $\cA$ is compact and Polish. 
Moreover, there exist functions $\tilde r: \cX \times \cA \times \R^d \to \R$, $\tilde g: \cX \times \R^d \to \R$, $\tilde P: \cX \times \cX \times \cA \times \R^d \to \R$, $p: \Theta \times \cA \times \{0, \ldots, T-1\} \times \cX \times \cP (\cX) \to \R$, $\tilde p: \Theta \times \cA \times \{0, \ldots, T-1\} \times \cX \times \R^d \to \R$ such that
    \begin{itemize}
        \item for any $x,x^\prime \in \cX, a \in \cA, l \in \R^d$, we have 
        $$
        \begin{aligned}
        &r(x,a,\softmax (l)) = \tilde r(x,a,l), \quad 
        g(x,\softmax (l)) = \tilde g(x,l), \\
        & P(x^\prime \mid x,a, \softmax(l)) = \tilde P(x^\prime \mid x, a, l).
        \end{aligned}
        $$ 
        \item $(x,a,l) \mapsto \tilde r(x,a,l)$ is bounded in $[-M_0, M_0]$ where $0 \leq M_0 < \infty$, continuous in $a \in \cA$ and continuously differentiable in $l \in \R^d$, 
        similarly $(x,l) \mapsto \tilde g(x,l)$ is bounded in $[-M_0,M_0]$ and continuously differentiable in $l \in \R^d$;
        \item the extended transition probability $(x^\prime, x, a, l) \mapsto \tilde P (x^\prime \mid x, a, l)$ is continuous in $a \in \cA$ and continuously differentiable in $l \in \R^d$;
        \item the parametrized policy $\pi_\theta(\d a \mid t,x, \mu)$ admits a density $a$ $\mapsto$ $p(\theta, a,t,x,\mu)$ with respect to a finite reference measure $\nu_{\cA}$, and we denote $| \cA | = \nu_\cA (\cA)$. We have $p(\theta,a,t,x,\softmax l) = \tilde p (\theta, a, t,x, l)$, and $\tilde p$ is differentiable in $\theta$, continuous in $a \in \cA$ and continuously differentiable in $l \in \R^d$.
         Moreover, $\nabla_\theta \log p$ is bounded, i.e., there exists $C > 0$ such that 
         $$
         \sup_{(\theta, a, x, \mu) \in \Theta \times \cA \times \cX \times \cP (\cX)^*}\| \nabla_\theta \log p (\theta, a, t,x, \mu) \|_\infty \leq C, \qquad t \in \{0,\ldots,T-1\}. 
         $$
        \item the transition probability $P$ and the policies $\{\pi_\theta: \theta \in \Theta\}$ are such that, for any $\mu \in \cP (\cX)^*$ and any $\theta \in \Theta$, we have, for any $t$ $\in$ $\{0,\ldots,T-1\}$, $x^\prime \in \cX$
        $$
         \sum_{i = 1}^d\int_\cA  P(x^\prime \mid x^{(i)}, a, \mu) \pi_\theta (\d a \mid t,x^{(i)}, \mu) \mu (x^{(i)}) > 0, 
        $$
        ensuring that if the initial condition $\mu_0$ has a valid logits representation, then under any admissible policy $\pi_\theta$ the flow of measures $(\mu_t)_{t=0}^T$ has a valid logits representation.
    \end{itemize}
\end{assumption}

\subsection{Policy Gradient Formula: from Single-Agent to Mean-Field}

Let $V(\mu_0, \theta) = V^{\pi_{\theta}}(\mu_0)$. Under Assumption \ref{assumption:finite-state-space}, we have $\pi_{\theta}(\d a \mid t,x,\mu) = p(\theta, a,t,x,\mu) \nu_{\cA}(\d a)$, $\mu_t^{\theta}(\d x) \coloneq  \mathbb P(X_t^\theta = x) \nu_{\cX}(\d x)$, and $l_t^\theta = \logit (\mu_t^\theta)$, where $\nu_{\cX}$ is the counting measure on $\cX$. 
Consider the joint probability of the whole trajectory $\left((X_t^{\theta}, \alpha_t^{\theta})_{t = 0}^{T - 1}, X_{T}^{\theta} )\right)$ under $\pi_{\theta}$, given by 
\[
P_{\theta}(\d x_0, \d a_0, \ldots, \d a_{T - 1}, \d x_T) \coloneq \mu_0(\d x_0) \prod_{\tau = 0}^{T - 1} \pi_{\theta} (\d a_\tau \mid \tau,x_{\tau}, \mu_\tau^{\theta}) P(\d x_{\tau + 1} \mid x_{\tau}, a_{\tau}, \mu_{\tau}^\theta).
\]
We want to emphasize the direct dependence of the law of $X_t^\theta$, denoted by $\mu_t^\theta$, on the parameters $\theta$. In practice, this means that differentiating the value function with respect to $\theta$ induces differentiating the reward functions $r,g$ and transition kernel $P$ with respect to the measure argument.
In the remainder of the paper, as hinted in Section \ref{subsection:logits}, the derivative with respect to the measure argument will be expressed in terms of the derivative with respect to the logits representation of the measure.
We denote by $l_t^\theta = \logit (\mu_t^\theta)$ the logits representation of the law of $X_t^\theta$, and we write $\nabla_\theta l_t^\theta \in \R^{d \times D}$ for the Jacobian matrix of $l_t^\theta$ with respect to $\theta$.

In particular, on top of the standard REINFORCE-like term (\citet{williams1992simple}), we obtain two additional terms corresponding to the measure derivative of the rewards, and the measure derivative of the transition kernel and policy (due to the mean-field nature of the dynamics).
\begin{proposition}
\label{proposition:exact-policy-gradient-formula}
    Let Assumption \ref{assumption:finite-state-space} hold. For any $\mu_0 \in \mathcal P(\cX)^*$ and $\theta \in \Th$, we have
    \begin{equation}
        \label{eq:exact-policy-gradient-formula}
        \nabla_{\theta} V(\mu_0, \theta) =  \mathrm{RF}(\theta) + \mathrm{MD}(\theta)+\mathrm{MFD}(\theta),
    \end{equation}
    where
    \begin{equation}
        \label{eq:decomposition-exact-policy-gradient}
\begin{aligned}
  \mathrm{RF}(\theta) & \coloneq  \int  \left( \sum_{t^\prime = 0}^{T - 1} \nabla_{\theta} \log p(\theta, a_{t^\prime},t^\prime,x_{t^\prime},\mu_{t^\prime}^\theta) \right)
    \left(\sum_{t = 0}^{T - 1} r(x_t, a_t, \mu_t^\theta)  + g(x_T, \mu_T^\theta)\right) \d P_{\theta},\\
\mathrm{MD}(\theta) &\coloneq \int \left(\sum_{t = 0}^{T - 1} \nabla_{l} \tilde r(x_t, a_t, l_t^\theta) \nabla_{\theta}l_t^\theta  + \nabla_{l} \tilde g (x_T, l_T^\theta) \nabla_{\theta}l_T^\theta \right) \d P_{\theta}, \\
\mathrm{MFD}(\theta) &\coloneq \int \left( \sum_{t^\prime = 0}^{T - 1} \nabla_{l} \log \tilde p(\theta, a_{t^\prime},t^\prime,x_{t^\prime},l_{t^\prime}^\theta) \nabla_{\theta} l_{t^\prime}^\theta + \nabla_{l} \log \tilde P(x_{t^\prime + 1} \mid x_{t^\prime},a_{t^{\prime}}, l_{t^\prime}^\theta) \nabla_{\theta} l_{t^\prime}^\theta \right) \\
&\qquad \cdot \left(\sum_{t = 0}^{T - 1} r(x_t, a_t, \mu_t^\theta)  + g(x_T, \mu_T^\theta)\right) \d P_{\theta},
\end{aligned}
\end{equation}
where $\mathrm{MD}(\theta), \mathrm{MFD}(\theta)$ stand for \textbf{M}easure \textbf{D}erivative and \textbf{M}ean-\textbf{F}ield \textbf{D}erivative respectively. 
\end{proposition}

\begin{remark}
\label{remark:single-agent-remark}
    In the classical single-agent RL setting, the policy gradient formula only contains the term $\mathrm{RF}(\theta)$, 
    which can be estimated in a model-free setting by generating trajectories of the state process. 
    However, in the mean-field setting, we have two additional terms, $\mathrm{MD}(\theta)$ and $\mathrm{MFD}(\theta)$, that involve the derivatives of the rewards, transition kernel and policy with respect to the measure argument. 
    These terms cannot be estimated directly in a model-free setting since they require knowledge of the model coefficients. 
\end{remark}

\subsection{Perturbed Process and Value Function}
\label{subsection:perturbed-process}

We first give an informal presentation of the perturbation argument. For $\pi \in \Pi, \eps > 0,$ and a sequence of perturbations $\lambda$ $=$ $(\lambda_t)_{t = 0}^{T }$, we consider processes $X^\pi_t, Y^{\pi, \lambda}_t$ such that
\begin{equation}
\label{eq:perturbed-process}
\begin{aligned}
X^{\pi}_0 &= \xi, \quad \alpha^{\pi}_t \sim \pi(\cdot \mid t, X^{\pi}_t, \mathbb P_{X_t^{\pi}}),  \quad X^{\pi}_{t + 1} \sim P( \cdot \mid X^{\pi}_t, \alpha^{\pi}_t, \mathbb P_{X_t^\pi}), \\
Y^{\pi,\lambda}_0 &= \xi,\quad  \alpha_{t}^{\pi, \lambda} \sim \pi (\cdot \mid t, Y_t^{\pi, \lambda}, \softmax (l_t^\pi + \eps \lambda_t)), \\ Y_{t + 1}^{\pi, \lambda} &\sim P(\cdot \mid Y_t^{\pi, \lambda}, \alpha_t^{\pi, \lambda}, \softmax (l_t^\pi + \eps \lambda_t)),
\end{aligned}
\end{equation}
where $l_t^\pi \coloneq \logit(\mathbb P_{X_t^\pi})$.
We consider an i.i.d.\ sequence of random perturbations $\Lambda = (\Lambda_t)_{t = 0}^{T}$ of the state distribution at each time step, independent of the idiosyncratic noise, and distributed according to $\nu$. 
As $\mathcal P (\cX)$ is not a vector space, one needs to be careful about the definition of $(\Lambda_t)$. Thanks to the logits parametrization introduced in Section \ref{subsection:logits}, this perturbation can easily be applied in the space of logits. We write $\tilde \pi ( \d a \mid t,x, l) = \pi (\d a \mid t,x, \softmax (l))$. Then, for a fixed $\pi$, $Y^{\pi, \Lambda}$ is a process for which the distributional dependences are randomized by $\Lambda$.
We consider the following approximation of the value function for a policy $\pi$
\begin{equation}\label{eq:def-perturbed-value-function}
    V_{\eps}^{\pi}(\mu_0) \coloneq \mathbb E^{\pi}_{\xi \sim \mu_0} \left[ \sum_{t = 0}^{T - 1} r(Y_t^{\pi, \Lambda}, \alpha^{\pi,\Lambda}_t, \softmax(l_t^\pi + \eps \Lambda_t)) + g(Y^{\pi, \Lambda}_T, \softmax (l_T^\pi + \eps \Lambda_T))\right].
\end{equation}
This perturbed value function can be explicitly written as an integral
\[
\begin{aligned}
    V_{\eps}^{\pi}(\mu_0) &= \int \left(\sum_{t = 0}^{T - 1} \tilde r(y_t, a_t,l^\pi_t + \eps \lambda_t) + \tilde g(y_T, l_T^\pi + \eps \lambda_T) \right) \mu_0(\d y_0)  \\
    & \qquad \cdot \prod_{\tau = 0}^{T - 1} \nu (\d \lambda_\tau)\tilde \pi (\d a_\tau \mid \tau,y_\tau, l^\pi_\tau + \eps \lambda_\tau) 
    \tilde P (\d y_{\tau + 1} \mid y_\tau, a_\tau, l^\pi_\tau + \eps \lambda_\tau).
\end{aligned}
\]
From this expression, the goal is to perform a change of variables in order to remove the dependence on $\theta$ in the measure term of the rewards, policy and transition kernel. The gradient of this perturbed value function will not include the derivatives of the coefficients in the measure argument, which will lead to a model-free approximate policy gradient formula.

 The parametrization of a probability distribution by its logits gives a way of continuously perturbing any probability distribution $\mu \in \cP (\cX)^*$: 
 first, compute $l_{\mu} = \logit (\mu)$, then sample an $\R^{d}$-valued random variable $\Lambda$, and finally define the perturbed distribution as $\softmax (l_\mu + \eps \Lambda) \in \cP (\cX)^*$. This is the main motivation for introducing this parametrization.

We now present a lemma that motivates the choice of this perturbation-via-logits scheme.

\begin{lemma}
    \label{lemma:tv-bound}
    Let $\mu \in \cP (\cX)^*$ and $\eps > 0$. Denote by $\mu_\eps$ the perturbed distribution defined by 
    \[
    \mu_\eps = \softmax (\logit (\mu) + \eps \Lambda), \: \Lambda \sim \mathcal N (0, I_d),
    \]
    hence $\mu_\eps$ is a random variable valued in $\cP (\cX)^*$. Then, we have
    \begin{equation}
        \label{eq:tv-bound}
        \mathbb E[d_\tv(\mu, \mu_\eps)] \leq \frac{\eps}{2}.
    \end{equation}
\end{lemma}

In particular, Lemma \ref{lemma:tv-bound} shows that the perturbation scheme leads to a controllable perturbation of the state distribution in total variation distance, of order $\eps$.

\subsection{Derivation and Convergence of Perturbed Gradient}

We now present a policy gradient formula for the perturbed value function $V_\eps(\mu, \theta) = V_\eps^{\pi_\theta}(\mu)$ in \eqref{eq:def-perturbed-value-function}. 
It is the foundation of our model-free estimator of $\nabla_\theta V(\mu, \theta)$.

\begin{theorem}
    \label{theorem:perturbed-gradient}
    Let Assumption \ref{assumption:finite-state-space} hold.
    Let $\Lambda_t,t = 0, \ldots, T$ be i.i.d. $\mathcal N (0, I_d)$ random variables. 
    Assume that for any $\pi_\theta$, $(X_t^\theta)_{t = 0}^T$ is an $\cX$-valued process independent of $(\Lambda_t)_{t=0}^{T}$ and
    $(Y_t^{\theta, \eps})_{t=0}^T$ is an $\cX$-valued process such that, conditioned on $\Lambda = \lambda = (\lambda_0, \ldots, \lambda_{T}) \in \R^{dT}$,
    $X^\theta$ and $Y^{\theta, \eps}$ are distributed according to \eqref{eq:perturbed-process} with $\pi = \pi_\theta$ 
    and initial condition $\xi$ such that $\mu_0 = \mathcal L (\xi) \in \mathcal P (\cX)^*$.
    Let $(\alpha_t^\theta)_{t=0}^{T-1}$ and $(\alpha_t^{\theta, \eps})_{t=0}^{T-1}$ be the corresponding controls.
    Then, for $\theta \in \Theta$, the perturbed value function $V_\eps (\mu_0, \theta) = V_\eps^{\pi_\theta}(\mu_0)$ is differentiable in $\theta$ with 
    \begin{equation}
        \label{eq:policy-gradient-formula}
\nabla_{\theta} V_\eps(\mu_0, \theta) = \mathbb E \Bigg[ \sum_{t = 0}^{T} \left(\eps^{-1} \Lambda_t \nabla_{\theta} l_t^\theta +  1_{t < T} \nabla_\theta \log \tilde p(\theta, \alpha_t^{\theta,\eps}, t,Y_t^{\theta,\eps}, l_t^\theta + \eps \Lambda_t) \right)G_t^\eps\Bigg],
    \end{equation}
    where $G_t^\eps = \sum_{s=t}^{T-1} \tilde r(Y^{\theta,\eps}_s, \alpha_s^{\theta,\eps}, l_s^\theta + \eps \Lambda_s) + \tilde g(Y_T^{\theta, \eps}, l_T^\theta + \eps \Lambda_T)$ and $l_t^\theta = \logit (\mathbb P_{X_t^\theta})$.
\end{theorem}

\begin{remark}
    Notice that, contrary to the exact policy gradient formula \eqref{eq:exact-policy-gradient-formula}, the formula \eqref{eq:policy-gradient-formula} only involves derivatives with respect to the trainable parameters $\theta$. However, estimating this quantity numerically is not straightforward due to the gradient of logits $\nabla_\theta l_t^\theta$. To overcome this challenge, one can use the same perturbation method applied to a problem with the same dynamics but with a different reward, as detailed in Section \ref{subsection:model-free-estimator}, 
\end{remark}

\begin{theorem}
    \label{theorem:convergence-perturbed-gradient}
    Assume the same conditions as in Theorem \ref{theorem:perturbed-gradient}. Furthermore, assume that $l \mapsto \tilde P (x^\prime \mid x, a, l) p(\theta, a, t,x, l)$ is twice differentiable, with bounded second derivative uniformly in $t,x, x^\prime, a, \theta$.
    Then, as $\eps \to 0$, we have $\nabla_\theta V_\eps(\mu_0, \theta) \to \nabla_\theta V(\mu_0, \theta)$.
\end{theorem}

\begin{remark}
    We want to highlight that both the gradient representation of Theorem \ref{theorem:perturbed-gradient} and the convergence result Theorem \ref{theorem:convergence-perturbed-gradient} 
    hold independently of the way the processes $X^\theta$ and $Y^{\theta, \Lambda}$ are generated, as long as the assumptions of the theorems are satisfied.
     In particular, we do not make any coupling assumption between $X_t^\theta$ and $Y_t^{\theta, \Lambda}$.
     In practice, using coupled samples might help reduce variance in the gradient estimates when subtracting a baseline depending on $X_t^\theta$ to the cumulated rewards in \eqref{eq:policy-gradient-formula},
     as commonly done in single-agent REINFORCE implementations (\citet{williams1992simple, Sutton1998}).
\end{remark}

\section{Policy Gradient Algorithm}
\label{section:pg-algorithm}

In this section, we introduce a policy gradient estimator for mean-field control problems based on \eqref{eq:policy-gradient-formula}, and analyze its bias and error.

\subsection{Model-Free Estimator}
\label{subsection:model-free-estimator}

For the policy gradient formula \eqref{eq:policy-gradient-formula} to be estimated in a model-free way, one needs to approximate $\nabla_{\theta} l_t^\theta$.
Recall that, under the probability--logit correspondence, we have
\[
\nabla_{\theta} l_t^\theta = \nabla_{\theta}\left( \log \mathbb P(X_t^\theta = x^{(1)}), \cdots, \log \mathbb P(X_t^\theta = x^{(d)})\right).
\]
Notice that 
\[
\nabla_{\theta} \log \mathbb P(X_t^\theta = x^{(i)}) = \frac{\nabla_\theta \mathbb P (X_t^\theta = x^{(i)})}{\mathbb P (X_t^\theta = x^{(i)})}.
\]
Using the same perturbation method as in Section \ref{subsection:perturbed-process}, we can approximate $\nabla_{\theta} \mathbb P(X_t^\theta= x^{(i)})$ by
\[
\begin{aligned}
    \nabla_\theta \mathbb P (X_t^\theta= x^{(i)}) &= \nabla_{\theta} \mathbb E \left[ 1_{X_t^\theta = x^{(i)}} \right] \\
    &\approx  \mathbb E \left[ 1_{Y_t^{\theta,\eps} = x^{(i)}} \sum_{s = 0}^{T} \eps^{-1} \Lambda_s \nabla_\theta l_s^\theta +  1_{Y_t^{\theta,\eps} = x^{(i)}} \sum_{s = 0}^{T-1} \nabla_{\theta} \log \tilde p(\theta, \alpha_s^{\theta, \eps},s,  Y_s^{\theta, \eps},l_s^\theta + \eps \Lambda_s)\right] \\
    &= \mathbb E \left[ 1_{Y_t^{\theta,\eps} = x^{(i)}} \sum_{s = 0}^{t-1} \eps^{-1} \Lambda_s  \nabla_\theta l_s^\theta  +  1_{Y_t^{\theta,\eps} = x^{(i)}} \sum_{s = 0}^{t-1} \nabla_{\theta} \log \tilde p(\theta, \alpha_s^{\theta, \eps},s,  Y_s^{\theta, \eps},l_s^\theta + \eps \Lambda_s)\right].
\end{aligned}
\]
Notice that we used the causality and log-likelihood trick: we are able to remove the terms $s \geq t$ due to the `reward' at time $t$ not depending on the actions at time $s \geq t$. Indeed, for $s \geq t$, by denoting $\mathcal F_s$ the filtration generated by $(Y_{\leq s}^{\theta, \eps}, \Lambda_{\leq s}, \alpha_{ < s}^{\theta, \eps})$, we have
\[
\begin{aligned}
&\mathbb E \left[ 1_{Y_t^{\theta,\eps} = x^{(i)}} \left( \eps^{-1}\Lambda_s  \nabla_\theta l_s^\theta + \nabla_{\theta} \log \tilde p(\theta, \alpha_s^{\theta, \eps},s,  Y_s^{\theta, \eps},l_s^\theta + \eps \Lambda_s) \right)\right] \\ 
& \qquad= \mathbb E \left[ 1_{Y_t^{\theta,\eps} = x^{(i)}} \right] \mathbb E \left[\eps^{-1}\Lambda_s  \nabla_\theta l_s^\theta\right] + \mathbb E \left[ 1_{Y_t^{\theta,\eps} = x^{(i)}} \nabla_{\theta} \log \tilde p(\theta, \alpha_s^{\theta, \eps},s,  Y_s^{\theta, \eps},l_s^\theta + \eps \Lambda_s)  \right] \\ 
&\qquad = \mathbb E \left[ 1_{Y_t^{\theta,\eps} = x^{(i)}} \right] \cdot 0 + \mathbb E \left[ \mathbb E \left[ 1_{Y_t^{\theta,\eps} = x^{(i)}} \nabla_{\theta} \log \tilde p(\theta, \alpha_s^{\theta, \eps},s, Y_s^{\theta, \eps}, l_s^\theta + \epsilon \Lambda_s) \mid \mathcal{F}_s \right] \right] \\
& \qquad= 0 + \mathbb E \left[ 1_{Y_t^{\theta,\eps} = x^{(i)}} \int \nabla_{\theta} \tilde p(\theta, a, s, Y_s^{\theta, \eps}, l_s^\theta + \epsilon \Lambda_s) \, da \right] \\
&\qquad = \mathbb E \left[ 1_{Y_t^{\theta,\eps} = x^{(i)}} \nabla_{\theta} (1) \right] = 0. 
\end{aligned}
\]
Hence, assuming we have access to $\mathbb P(X_t^\theta =x^{(i)})$ for every $i =1,\ldots,d$ (either via direct observation or by sampling many particles), we get $T \times d \times D $ linear equations with unknowns $\frac{\partial }{\partial \theta_k} \mathbb P (X_t^\theta = x^{(i)})$, for $t = 0,\ldots, T- 1,i = 1,\ldots,d, k = 1,\ldots, D$, where the coefficients can be estimated by Monte Carlo sampling.
This system is triangular in the time variable, and it is easy to see that $\nabla_\theta \mathbb P(X_0^\theta = x^{(i)}) = 0$ for all $i$, hence the system can be solved by forward substitution.
Still, the resulting estimator is biased due to the approximation in the perturbation method, but the bias can be made arbitrarily small by taking $\eps$ small enough. 
However, taking $\eps$ too small may lead to high variance in the estimator. 
The bias and variance of this estimator are studied in Sections \ref{subsection:convergence-of-estimator} and \ref{subsection:error} respectively. 
In practice, the linear system might be easier to solve by subtracting a baseline from the first term
\[
\mathbb E \left[ 1_{Y_t^{\theta,\eps} = x^{(i)}} \sum_{s = 0}^{t-1} \eps^{-1}  \Lambda_s \nabla_\theta l_s^\theta \right] = \mathbb E \left[ (1_{Y_t^{\theta,\eps} = x^{(i)}} - b(X^\theta_t) )\sum_{s = 0}^{t-1} \eps^{-1}  \Lambda_s \nabla_\theta \mathbb P_{X_s^\theta} \right].
\]
We leave the study of variance reduction techniques for this estimator to future work.

\subsection{Algorithm}

The perturbed gradient given by \eqref{eq:policy-gradient-formula} is the basis of our policy-gradient algorithm, which we call Mean-Field REINFORCE (abbreviated MF-REINFORCE).
The pseudocode for Mean-Field REINFORCE is given in Algorithm \ref{alg:reinforce}, where we use Algorithm \ref{alg:state-distribution-gradients} to estimate the state distribution gradients $\nabla_\theta l_t^\theta, t = 0, \ldots, T$. In both of these algorithms, we assume that, on top of being able to generate individual transitions, we have access to a population simulator, i.e. a simulator that outputs $\mathbb P_{X_{t+1}}$ given $\mathbb P_{X_t}$ and a policy $\pi$. In practice, one can also estimate $\mathbb P_{X_t}$ by sampling many independent trajectories.

\begin{algorithm}[H]
\caption{Mean-Field REINFORCE for Finite-Horizon MFMDP}
\label{alg:reinforce}
\begin{algorithmic}[1]
\State \textbf{Inputs:} number of training steps $N_\text{train}$; horizon $T$; number of trajectories $N$; number of trajectories for logits gradient estimation $n$; differentiable policy $\pi_\theta(\d a\mid t,x,\mu)$; reward $r(x,a,\mu)$; terminal reward $g(x,\mu)$; transition kernel $P(\cdot\mid x,a,\mu)$; initial state law $\mu_0$.
\State Initialize parameters $\theta$.
\For{$i = 1, \ldots, N_\text{train}$}
\For{$k=1,\ldots, N$}
    \State Sample $X_0,Y_0 \sim \mu_0$ independently.
    \State Set $ l_0 = \mathrm{logit}(\mu_0)$.
    \State Sample $\Lambda_t \sim \mathcal N(0,I_d)$ independently for $t = 0, \ldots, T$.
    \For{$t=0,\ldots, T-1$} \Comment{Sample trajectories for mean-field and perturbed processes}
        \State Sample $\alpha_t \sim \pi_\theta(\cdot \mid t,X_t,\mathbb P_{X_t})$.
        \State Sample $\alpha^\eps_t \sim \tilde \pi_\theta(\cdot \mid t,Y_t, l_t + \eps \Lambda_t)$.
        \State Observe one-step rewards $R_t \gets r(X_t, \alpha_t,\mathbb P_{X_t}), R_t^\eps \gets \tilde r(Y_t, \alpha_t^\eps,  l_t + \eps \Lambda_t)$.
        \State Sample next states $X_{t+1} \sim P(\cdot \mid X_t, \alpha_t, \mathbb P_{X_t}), Y_{t + 1} \sim \tilde P(\cdot \mid Y_t, \alpha_t^\eps,  l_t + \eps \Lambda_t)$.
        \State Observe new state distribution (as logits) $l_{t + 1} = \logit (\mathbb P_{X_{t + 1}})$ given $\mathbb P_{X_t}, \pi_\theta$.
    \EndFor
    \State Set terminal rewards $G_T \gets g(X_T,\mathbb P_{X_T}), G_T^\eps \gets \tilde g(Y_T, l_T + \eps \Lambda_T)$.
    \For{$t=T-1,\ldots, 0$} \Comment{Backward pass to form returns}
        \State $G_t \gets R_t + G_{t+1}, G^\eps_{t} \gets R_t^\eps + G^\eps_{t + 1}.$ 
    \EndFor
    \State Estimate state distribution gradients $\left(\widehat{\nabla_\theta l_t}\right)_{t=0}^{T}$ with Algorithm \ref{alg:state-distribution-gradients} using $n$ trajectories.
    \State $\widehat{\nabla V}(\theta)^{(k)} \gets \displaystyle \sum_{t=0}^{T} \left(\eps^{-1} \Lambda_t \widehat{\nabla_\theta l_t} + \mathbf 1 _{t < T}\nabla_\theta \log \tilde \pi_\theta(\alpha_t^\eps \mid t,Y_t, l_t + \eps \Lambda_t) \right)G_t^\eps$. \Comment{PG estimator}
\EndFor
\State $\widehat{\nabla V}(\theta) \gets \frac{1}{N} \sum_{k=1}^N \widehat{\nabla V}(\theta)^{(k)}$.
\State{Perform gradient ascent step on $\theta$ using $\widehat{\nabla V}(\theta)$.}
\EndFor
\State \textbf{Output:} $\theta$.
\end{algorithmic}
\end{algorithm}

\begin{algorithm}[H]
\caption{State Distribution Gradient Estimation}
\label{alg:state-distribution-gradients}
\begin{algorithmic}[2]
\State \textbf{Inputs:} horizon $T$; number of trajectories $n$; differentiable policy $\pi_\theta(a\mid t,x,\mu)$; transition kernel $P(\cdot\mid x,a,\mu)$; initial state law $\mu_0$.
\State Set $\widehat{\nabla_\theta l_0} \gets \mathbf{0}_{d \times D}$.
\For{$t=1,\ldots, T$} 
\For{$k=1,\ldots, n$}
    \State Sample $X_0^k,Y_0^k \sim \mu_0$ independently.
    \State Set $l_0 = \mathrm{logit}(\mu_0)$.
    \State Sample $\Lambda_t^k \sim \mathcal N(0,I_d)$ independently for $t = 0, \ldots, T-1$.
    \For{$s=0,\ldots, t-1$} \Comment{Sample trajectories}
        \State Sample $\alpha_s^k \sim \pi_\theta(\cdot \mid s,X_s^k,\mathbb P_{X_s})$.
        \State Sample $\alpha^{\eps,k}_s \sim \tilde \pi_\theta(\cdot \mid s,Y_s^k, l_s + \eps \Lambda_s)$.
        \State Sample next states $X_{s+1}^k \sim P(\cdot \mid X_s^k, \alpha_s^k, \mathbb P_{X_s}), Y_{s + 1}^k \sim \tilde P(\cdot \mid Y_s^k, \alpha_s^{\eps,k}, l_s + \eps \Lambda_s^k)$.
        \State Observe new state distribution (as logits) $l_{s + 1} = \logit (\mathbb P_{X_{s + 1}})$ given $\mathbb P_{X_t}, \pi_\theta$.
    \EndFor
\EndFor

    \For{$i = 1, \ldots, d$}
        \State $\widehat{\nabla_{\theta} \mu^{(i)}} \gets \displaystyle \frac{1}{n}\sum_{k = 1}^n 1_{Y^k_t = x^{(i)}} 
        \left( \sum_{s = 0}^{t-1}  \eps^{-1} \Lambda_s^k \widehat{\nabla_\theta l_s}  + \nabla_\theta \log \tilde \pi_\theta(\alpha_s^{\eps,k} \mid s,Y_s^k, l_s + \eps \Lambda^k_s) \right)  $.
    \EndFor
    \State $\widehat{\nabla_\theta l_t} \gets \displaystyle \left(\frac{\widehat{\nabla_{\theta} \mu^{(i)}}}{\mathbb P (X_t = x^{(i)})}\right)_{i = 1}^{d} \in \R^{d \times D}$.
\EndFor
\State \textbf{Output:} $\left(\widehat{\nabla_\theta l_t}\right)_{t=0}^{T}$.
\end{algorithmic}
\end{algorithm}

\begin{remark}
    Notice that for every time step $t=1, \ldots, T$ we resample new trajectories to estimate $\nabla_\theta l_t^\theta$ in Algorithm \ref{alg:state-distribution-gradients},
     in order to reduce the bias induced by the approximation of $\nabla_\theta l_s^\theta$ for $s \leq t-1$. 
    Similarly, we use independent trajectories for estimating $\nabla_\theta l_t^\theta$ and $\nabla_\theta V_{\eps}(\mu, \theta)$.
\end{remark}

\subsection{Bias of Estimator}
\label{subsection:convergence-of-estimator}

We now study the bias of the estimator given by the Mean-Field REINFORCE algorithm.
The first step is to analyze the bias in the estimation of $\nabla_\theta l_t^\theta$ given by Algorithm \ref{alg:state-distribution-gradients}.
As $\nabla_\theta l_t^\theta \in \R^{d \times D}$, we bound the bias in the following matrix norm $\| M \| = \sup_{j = 1, \ldots, D} \sum_{i=1}^d M_{i,j}$ for $M \in \R^{d \times D}$. In order to provide quantitative bounds, we make the following set of assumptions.

\begin{assumption}
    \label{assumption:grad-log-policy}
    The parametrization of the policy $\pi_\theta$ is such that $p$ and $\nabla_\theta \log p$ are globally Lipschitz in the measure argument, and so is the transition kernel $P$. Formally, there exist $L_p, L^\prime_p,L_P > 0$ such that for all  $\theta$, $t$ $\in$ $\{0,\ldots,T-1\}$, $x,x^\prime \in \cX, a \in \cA$ and $\mu, \mu' \in \mathcal P(\cX)$, 
    \begin{equation}
        \begin{aligned}
            &\left| p(\theta, a, t,x, \mu) - p(\theta, a, t,x, \mu') \right| \leq L_p d_\tv(\mu, \mu'), \\
            &\left\| \nabla_\theta \log p(\theta, a, t,x, \mu) - \nabla_\theta \log p(\theta, a, t,x, \mu') \right\|_\infty \leq L^\prime_p d_\tv(\mu, \mu'), \\
            &\sup_{x^\prime \in \cX}\left| P(x^\prime \mid x, a, \mu) - P(x^\prime \mid x, a, \mu') \right| \leq L_P d_\tv(\mu, \mu').
        \end{aligned}
    \end{equation}
    Moreover, $\tilde P$ and $\tilde p$ are twice differentiable in $l \in \R^d$, and there exist $B_1, B_2 \geq 0$ such that 
    for all $\theta \in \Theta, x,x' \in \cX, a \in \cA, t \in \{0, \ldots, T-1\}$ and $l \in \R^d$,
    \[
    \begin{aligned}
        &\left|\frac{\partial }{\partial l_i} \left( \tilde P (x^\prime \mid x, a, l)\tilde p(\theta, a,t, x, l) \right)  \right|\leq B_1, \: \forall i, \\
        &\left|\frac{\partial ^2}{\partial l_i \partial l_j} \left( \tilde P (x^\prime \mid x, a, l)\tilde p(\theta, a,t, x, l) \right) \right| \leq B_2, \: \forall i,j.
    \end{aligned}
    \] 
\end{assumption}

We first prove that the gradient of logits estimator is bounded, starting with a lemma on the true gradient of logits.
\begin{lemma}
    \label{lemma:bound-gradient-logits}
    Assume the same conditions as in Theorem \ref{theorem:perturbed-gradient}, and let Assumption \ref{assumption:grad-log-policy} hold. Then for any $t = 0, \ldots, T$, we have
    \begin{equation}
        \label{eq:bound-gradient-logits}
        \left\| \nabla_\theta l_t^\theta \right\| \leq \begin{cases}
            \frac{3 C (\beta^t - 1)}{2(\beta - 1)}, & \text{if } B_1 \neq 0, \\
            t \left\| \nabla_\theta \tilde p \right\| _\infty | \cA |, & \text{if } B_1 = 0,
        \end{cases}
    \end{equation}
    where 
    \begin{equation}
        \label{eq:definition-beta}
        \beta \coloneq \frac{1}{2} \left( 1 +B_1| \cA | + \sqrt{(1 + B_1| \cA |)^2 + 4 B_1| \cA | (d - 1)} \right).
    \end{equation}
\end{lemma}
Before presenting the results on the estimator, we introduce the sampling and perturbation setup of MF-REINFORCE.
\paragraph{Sampling and Perturbation Setup}
Let $\theta \in \Theta$ and $\eps \in (0,1)$.
Let $(\Lambda_t)_{t = 0}^T$ be i.i.d.\ $\mathcal N(0,I_d)$ random variables.
Let $(X_t^\theta)_{t = 0}^T$ and $(Y_t^{\theta,\eps})_{t = 0}^T$ be $\cX$-valued processes such that,
conditionally on $\Lambda = (\Lambda_0,\Lambda_1,\ldots,\Lambda_{T})$,
they are distributed according to the perturbed dynamics \eqref{eq:perturbed-process}
with policy $\pi=\pi_\theta$ and initial condition $\xi$ satisfying
$\mu_0=\mathcal L(\xi)\in\mathcal P(\cX)^*$.
Let $(\alpha_t^\theta)_{t=0}^{T-1}$ and $(\alpha_t^{\theta,\eps})_{t=0}^{T-1}$ denote the associated controls of $X^\theta$ and $Y^{\theta, \eps}$ respectively.

\begin{proposition}
    \label{proposition:bias-state-distribution-gradient}
    Let Assumptions \ref{assumption:finite-state-space} and \ref{assumption:grad-log-policy} hold. 
    Let $n \in \N$ and $\eps \in (0,1)$. 
    Under the sampling and perturbation setup above, let
$\{(X^{\theta,t,k},Y^{\theta,\eps,t,k},\Lambda^{t,k}, \alpha^{\theta, t, k}, \alpha^{\theta, \eps,t ,k})\}_{k=1}^n$
be i.i.d.\ copies truncated at time $t$, for all $t = 1, \ldots, T$.
    Moreover, assume that the trajectories $ ( X^{\theta,t,k}_{\leq t}, Y^{\theta, \eps,t,k}_{\leq t}, \Lambda^{t,k}_{\leq t},)$ for $k=1,\ldots,n$ and $t = 1,\ldots,T$ are independent.
    Define $\widehat{\nabla_\theta l^\theta_0} = 0 \in \R^{d \times D}$ and, for $t = 1, \ldots, T$,
    \begin{equation}
        \label{eq:logit-gradient-estimator}
    \begin{aligned}
        \widehat{\nabla_{\theta} \mu_t^{(i)}} &= \frac{1}{n}\sum_{k = 1}^n 1_{Y^{\theta, \eps,t,k}_t = x^{(i)}}
        \left( \sum_{s = 0}^{t-1}  \eps^{-1} \Lambda_s^{t,k} \widehat{\nabla_\theta l^\theta_s} + \nabla_\theta \log \tilde \pi_\theta \left(\alpha_s^{\theta, \eps,t,k} \mid s, Y_s^{\theta, \eps,t,k}, l_s^\theta + \eps \Lambda_s^{t,k}\right) \right), \\
        \widehat{\nabla_\theta l^\theta_t} &= \left(\frac{\widehat{\nabla_{\theta} \mu_t^{(i)}}}{\mathbb P (X_t^\theta = x^{(i)})}\right)_{i = 1}^{d} \in \R^{d \times D}.
    \end{aligned}
\end{equation}
    Let $\theta \in \Theta$. If $B_1 > 0$, 
    for $t = 0, \ldots, T$, we have 
    \begin{equation}
        \label{eq:bias-logit-gradient-estimator}
        E_t(\theta, \eps) \coloneq \left\| \mathbb E \left[ \widehat{\nabla_\theta l_t^\theta} \right] - \nabla_\theta l_t^\theta \right\|
         \leq 
         \mathbf \eps \left((K_{1, \eps} + K_{2, \eps}) \beta_\eps^t - K_{1, \eps} \beta^t - K_{2, \eps} \right),
    \end{equation}
    where
    $\beta_\eps, K_{1, \eps},K_{2, \eps}$ are constants defined in \eqref{eq:constants-bias-logits} depending on $\eps$ and the problem parameters, bounded as $\eps \to 0$.
     Otherwise, $\left\| \mathbb E \left[ \widehat{\nabla_\theta l_t^\theta} \right] - \nabla_\theta l_t^\theta \right\| \leq \mathbf s^\theta \frac{L_p^\prime}{2} \eps t$ where $\mathbf s^\theta \coloneq \sup_{t \geq 0} \sum_{i=1}^d \mathbb P (X_t^\theta = x^{(i)})^{-1}$.
\end{proposition}

\begin{remark}
\label{remark:bias-logit-estimation}
    As $\eps \to 0$, the dominating term for the bias is of order $\mathcal O (\eps)$. However, according to \eqref{eq:bias-logit-gradient-estimator}, the bias deteriorates exponentially with the time step $t$ considered when $B_1 > 0$. In practice, one might be able to counter this `curse of time' by discounting future rewards. In the case when $B_1 = 0$, the time dependence is linear. 
    This illustrates the challenge of designing policy gradient algorithms for mean-field dynamics, as $B_1=0$ corresponds to having no mean-field dependence in the state transitions.
\end{remark}

In order to properly bound the bias of the policy gradient estimator, we need to introduce the following regularity assumption on the reward.

\begin{assumption}
    \label{assumption:lipschitz-reward}
    The reward functions $r$ and $g$ are globally Lipschitz in the measure argument, i.e.,
    there exist constants $L_r, L_g > 0$ such that, for all $x \in \cX, a \in \cA, \mu, \mu^\prime \in \cP(\cX)$,
    \[
    | r(x,a,\mu) - r(x,a,\mu^\prime) | \leq L_r d_\tv (\mu, \mu^\prime), \quad | g(x,\mu) - g(x,\mu^\prime) | \leq L_g d_\tv (\mu, \mu^\prime).
    \]
    Moreover, $\tilde r$ and $\tilde g$ are twice differentiable in $l \in \R^d$, and there exist $M_1, M_2 > 0$ such that 
    for all $x \in \cX, a \in \cA$ and $l \in \R^d$,
    \[
    \begin{aligned}
        &\left| \frac{\partial}{\partial l_i} \tilde r(x,a,l) \right| \leq M_1, \: \left| \frac{\partial}{\partial l_i} \tilde g(x,l) \right| \leq M_1, \quad  \forall i, \\
        &\left| \frac{\partial ^2}{\partial l_i \partial l_j} \tilde r(x,a,l) \right| \leq M_2,\: \left|\frac{\partial ^2}{\partial l_i \partial l_j} \tilde g(x,l) \right| \leq M_2,   \quad  \forall i,j.
    \end{aligned}
    \]
\end{assumption}

We are now ready to state the main theorem of this section, which gives an upper bound on the bias
of the policy gradient estimator.

\begin{theorem}
\label{theorem:main-bias-result}
Let Assumptions \ref{assumption:finite-state-space}, \ref{assumption:grad-log-policy} and \ref{assumption:lipschitz-reward} hold.
Let $\eps \in (0,1)$ and $\theta \in \Theta$.
Let $\widehat{\nabla_\theta l_t^\theta}$ be the gradient of logits estimator defined in \eqref{eq:logit-gradient-estimator}.
Let $X^\theta, Y^{\theta, \eps}, \Lambda, \alpha^{\theta}, \alpha^{\theta, \eps}$ be given by the sampling and perturbation scheme above.
Assume that $(X_{\leq T}^\theta,Y_{\leq T}^{\theta, \eps},\Lambda_{\leq T}, \alpha^{\theta}_{\leq T - 1}, \alpha^{\theta, \eps}_{\leq T - 1})$ are independent of $(\widehat{\nabla_\theta l_{\leq T}^\theta})$.
Let 
\begin{equation}
    \label{eq:policy-gradient-estimator}
    \widehat{\nabla V}(\theta) = \sum_{t=0}^{T} \left(\eps^{-1} \Lambda_t \widehat{\nabla_\theta l^\theta_t} + \mathbf 1 _{t < T}\nabla_\theta \log \tilde p (\theta, \alpha_t^{\theta, \eps}, t,Y_t^\theta, l_t^\theta + \eps \Lambda_t) \right)G_t^{\theta, \eps}, 
\end{equation}
where $G_t^{\theta, \eps} = \sum_{s=t}^{T-1} \tilde r(Y_s^{\theta, \eps}, \alpha_s^{\theta, \eps}, l^\theta_s + \eps \Lambda_s) + \tilde g(Y_T^{\theta, \eps}, l_T^\theta + \eps \Lambda_T)$.
Then, 
\begin{equation}
    \label{eq:theorem-bound-bias-pg-estimator}
    \left\| \mathbb E \left[\widehat{\nabla V}(\theta) \right] - \nabla_\theta V (\mu_0, \theta)\right\|_\infty \leq \eps \left(k_{1,\eps} \beta_\eps^T - k_{2,\eps} \beta^T + k_{3, \eps} + \eps k_{4,\eps}\right),
\end{equation}
where $k_{1, \eps}, k_{2, \eps}, k_{3, \eps}, k_{4, \eps}$ are constants defined in \eqref{eq:constant-bias-pg-expression} that depend on $\eps$ and the problem parameters, which are bounded as $\eps \to 0$.
\end{theorem}

\begin{remark}
    The bias in the gradient of logits \eqref{eq:bias-logit-gradient-estimator} contributes transparently to the bias in the overall policy gradient estimator. As $\eps \to 0$, the dominating term in \eqref{eq:theorem-bound-bias-pg-estimator} is of order $\mathcal O (\eps)$, but we still observe a `curse of time'.
\end{remark}

\subsection{Error of the Estimator}\label{subsection:error}
In this section, we provide an upper bound on the per-coordinate mean squared error of the policy gradient estimator $\widehat{\nabla V}(\theta)$ defined in Theorem \ref{theorem:main-bias-result}, 
corresponding to the output of Algorithm \ref{alg:reinforce}.
We start by giving an upper bound on the second moment of the estimator $\widehat{\nabla_\theta l_t^\theta}$ defined in \eqref{eq:logit-gradient-estimator}.
\begin{proposition}
\label{proposition:l2-logit-gradient-estimator}
Consider the same setting as in Proposition \ref{proposition:bias-state-distribution-gradient}.
Define $\Sigma_{t}^{\theta, \eps}$ as the $d \times D$ matrix with entries 
\[
\left(\Sigma_{t}^{\theta,\eps} \right)_{i,j} \coloneq \var \left[ \left( \widehat{\nabla_\theta l_t^\theta} \right)_{i,j} \right], \quad i = 1,\ldots, d, j=1,\ldots, D.
\]
Then, for all $t = 0, \ldots, T$, $i = 1, \ldots, d$ and $j = 1, \ldots, D$,
\begin{equation}
\label{eq:variance-logit-gradient-bound}
\left\| \Sigma_{t}^{\theta, \eps} \right\| \leq  \frac{\gamma_{1, \eps}}{n \eps^2}  \left(1 + \frac{\gamma_{2} T}{\eps^2}\right)^t,
\end{equation}
where $\gamma_{1, \eps}, \gamma_2$ are constants defined in \eqref{eq:constants-variance-logit-gradient-estimator} that depend on $\eps$ and the problem parameters. 
\end{proposition}

\begin{remark}
\label{remark:variance-logit-gradient-estimator}
    The $n^{-1}$ factor is expected. As $\eps \to 0$, $\gamma_{1,\eps}$ is of order $\eps^{-2}$ (see the definition of $\gamma_{1, \eps}$ in \eqref{eq:constants-variance-logit-gradient-estimator}). Hence the overall dependence in $\eps$ of the bound \eqref{eq:variance-logit-gradient-bound} is at worst $\eps^{-2T - 4}$.  Here again, we observe a curse of time.
    \end{remark}

\begin{theorem}
\label{theorem:main-mse-result}
Let Assumptions \ref{assumption:finite-state-space}, \ref{assumption:grad-log-policy} and \ref{assumption:lipschitz-reward} hold.
Let $\eps \in (0,1)$, $\theta \in \Theta$, $n \in \mathbb N$ and $N \in \mathbb N$.
Let $\widehat{\nabla V}(\theta)^{(1)}, \ldots, \widehat{\nabla V}(\theta)^{(N)}$ be $N$ independent copies of the policy gradient estimator defined in \eqref{eq:policy-gradient-estimator},
where each estimator $ \widehat{\nabla V}(\theta)^{(k)}$ is computed using an independent version of the estimators $\left(\widehat{\nabla_\theta l_t^\theta} \right)_{t=0}^T$ defined in \eqref{eq:logit-gradient-estimator}.
Let $\widehat{\nabla V}(\theta) = N^{-1} \sum_{k=1}^N \widehat{\nabla V}(\theta)^{(k)}$ be the averaged estimator.
Then, for all $\theta \in \Theta$,
\begin{equation}
\label{eq:mse-bound}
    \mathrm{MSE}(\theta, N, n, \eps) \coloneq \max_{j = 1, \ldots, D}\mathbb E \left[\left| \widehat{\nabla V}(\theta)_j - \nabla_\theta V (\mu, \theta)_j \right|^2 \right] \leq  \mathcal K_{1, \eps} \eps^2 + \frac{\mathcal K_{2, \eps}}{ N \eps^2} + \frac{\mathcal K_{3, \eps}}{n N \eps^4}  \sum_{t = 0}^T  \mathcal{C}_t\eps^{-2 t},
\end{equation}
where
\begin{equation}
    \label{eq:constants-main-mse-result}
    \begin{aligned}
  \mathcal K_{1, \eps} &= \left( k_{1,\eps}  \beta_\eps^T - k_{2,\eps} \beta^T + k_{3,\eps} + k_{\eps} \eps \right)^2, \\ 
  \mathcal K_{2, \eps} &=  2 (T + 1)^3 M_0^2 \sum_{t = 0}^T \left(E_t(\theta, \eps) + \left\| \nabla_\theta l_t^\theta \right\| +  \eps C\right),\\
  \mathcal K_{3,\eps} & =  6 (T + 1)^3 M_0^2 d  \sup_{0 \leq t \leq T} \mathbb E \left[ \left(\sum_{s = 0}^{t-1} \sum_{i^\prime = 1}^d  \left|\widehat{\nabla_\theta l_s^\theta}\right|_{i^\prime,j}  + \eps C\right)^2 \right] \sum_{i=1}^d \mathbb P \left( X_t^\theta = x^{(i)} \right)^{-2}, \\
  \mathcal C_t &= \binom{T + 1}{t + 1}\left(   d T \sup_{0 \leq t \leq T} \sum_{i=1}^d \mathbb P \left( X_t^\theta = x^{(i)} \right)^{-2} \right)^t.
\end{aligned} 
\end{equation}
\end{theorem}

\begin{remark}
    \label{remark:asymptotic-bound-mse}
    Notice that $\mathcal K_{2, \eps}, \mathcal K_{3, \eps}$ are finite due to \eqref{eq:bias-logit-gradient-estimator} and \eqref{eq:variance-logit-gradient-bound} respectively.
    The bound \eqref{eq:mse-bound} provides practical insights into the influence of the hyperparameters $n,N$ and $\eps$.
    First, one chooses $\eps$ small enough for the squared bias term $K_{1, \eps} \eps^2$. 
    Then, one chooses $N$, i.e.\ the number of independent trajectories for the policy gradient estimation, to control $\mathcal K_{2, \eps} / ( N \eps^2 )$. 
    Finally, to further help control the high-order term $\eps^{- 2T - 4}$, one can choose to increase $n$, i.e.\ the number of independent trajectories per time-step in the computation of $\widehat{\nabla_\theta}l_t^\theta$, to be large enough.
    Since the total sample complexity for generating $\widehat{\nabla V}(\theta)$ is $n N T^2$,  \eqref{eq:mse-bound} suggests that one should take $N$ as large as possible and $n=1$ to minimize the MSE for a fixed computational budget. In practice, if time complexity is the main constraint, one can parallelize the simulation of $n$ independent trajectories for the gradient-of-logits estimation $\widehat{\nabla_\theta}l_t^\theta$, which allows to increase $n$ to further reduce the variance without increasing the total computation time.
\end{remark}

\section{Numerical Experiments}
\label{section:numerical-experiments}

\subsection{Two-State Two-Action Toy Problem}

The first experiment is a simple two-state two-action mean-field control problem. 
It was previously considered by \cite{gu2023dynamic} in the context of mean-field $Q$-learning.
The action space is $\cA = \{ \st, \mv \}$ and the state space is $\cX = \{0,1\}$.
The transition kernel is defined as follows: for $x \in \cX, a \in \cA$, $P(x^\prime \mid x, a) = \lambda_x \mathbf 1_{a = \mv}$ if $x^\prime \neq x$ and 
$P(x^\prime \mid x, a) = 1 - \lambda_x \mathbf 1_{a = \mv}$ if $x^\prime = x$, where $\lambda_0, \lambda_1 \in (0,1)$ are fixed parameters.
The running reward function is defined as
\[
r(x,a,\mu) = r(x, \mu) = \mathbf 1_{x = 1} - \mu(1)^2 - \lambda W_1(\mu, B),
\]
where $W_1$ is the 1-Wasserstein distance, $\lambda > 0$ is a fixed scalar parameter and $B$ is a Bernoulli distribution with parameter $p$ such that $1- \lambda_0  \leq p \leq \lambda_1$.
The terminal reward function is $g(x, \mu) = r(x, \mu)$.
Although the original problem in \cite{gu2023dynamic} is an infinite-horizon discounted problem, we consider here a finite-horizon version of with time horizon $T=2$.
In this setting, there exists an optimal stationary policy given by
\begin{equation}
\label{eq:optimal-policy-example31}
\begin{aligned}
    \pi^\star(a \mid 0, \mu) &= \left(1-\frac{1-p}{\lambda_0}\right) \mathbf{1}_{\{a=\st\}}+\frac{1-p}{\lambda_0} \mathbf{1}_{\{a=\mv\}}, \\
   \pi^\star(a \mid 1, \mu) &= \left(1-\frac{p}{\lambda_1}\right) \mathbf{1}_{\{a=\mathrm{ST}\}}+\frac{p}{\lambda_1} \mathbf{1}_{\{a=\mathrm{MV}\}}.
\end{aligned}
\end{equation}
In our experiments, we take the following parameters:
$\lambda_0 = 0.5, \lambda_1 = 0.8, \lambda= 10, p = 0.6$.
This policy guarantees that, no matter the initial distribution $\mu_0$, the distribution at all time steps $t \geq 1$ is $\mu_t^\star = B$.
The training and evaluation are done as follows.
For each training episode we simulate the population starting from a random initial distribution $\mu\in \mathcal P (\cX)^*$ by drawing $\mu(L) \sim \mathcal U ([0.1,0.9])$. 
Every 10 training episodes, we freeze the policy and sample a validation episode for which we compute the population reward $V(\mu_0)$ starting from a fixed initial distribution $\mu_0 = (\mu_0(0), \mu_0(1)) = (0.2,0.8)$.
During training, we compute $N = 200$ trajectories for MF-REINFORCE, with $n = 10$ trajectories for the gradient of logits estimation. 
For this first example, the policy considered is a simple static policy that outputs a $| \cX| \times | \cA |$ matrix corresponding to the probability of each action given each state.
We make this particular choice of policy to check whether MF-REINFORCE is capable of recovering the optimal static policy given in \eqref{eq:optimal-policy-example31}. 
We use the Adam optimizer \cite{kingma2014adam} and train for 5,000 episodes using a learning rate of $10^{-3}$. 
The experiments are run for $\eps$ ranging in the set $\{0.2, 0.5, 1.0, 2.0\}$. 
The validation rewards for the different values of $\eps$ are shown in Figure \ref{fig:example31-validation}. One can see that a larger value of $\eps$ leads to faster increase in the value function at first, but yields worse final policies. 
In particular, we see that choosing $\eps$ too large can lead to catastophic failure.
To evaluate how well the learned policies match the optimal static policy $\pi^\star$ \eqref{eq:optimal-policy-example31}, we compute the average absolute errors in the resulting estimates of $\pi(\mathrm{ST} \mid 0)$ and $\pi(\mathrm{ST} \mid 1)$ over 5 independent training runs for each value of $\eps$ considered. The results are reported in Table \ref{tab:policy-error-example31}. These results match adequately with the curves of Figure \ref{fig:example31-validation}: the smaller $\eps$, the better and closer to optimal the learned policy is.

\begin{figure}[!ht]
    \centering
    \includegraphics[width=0.8\linewidth]{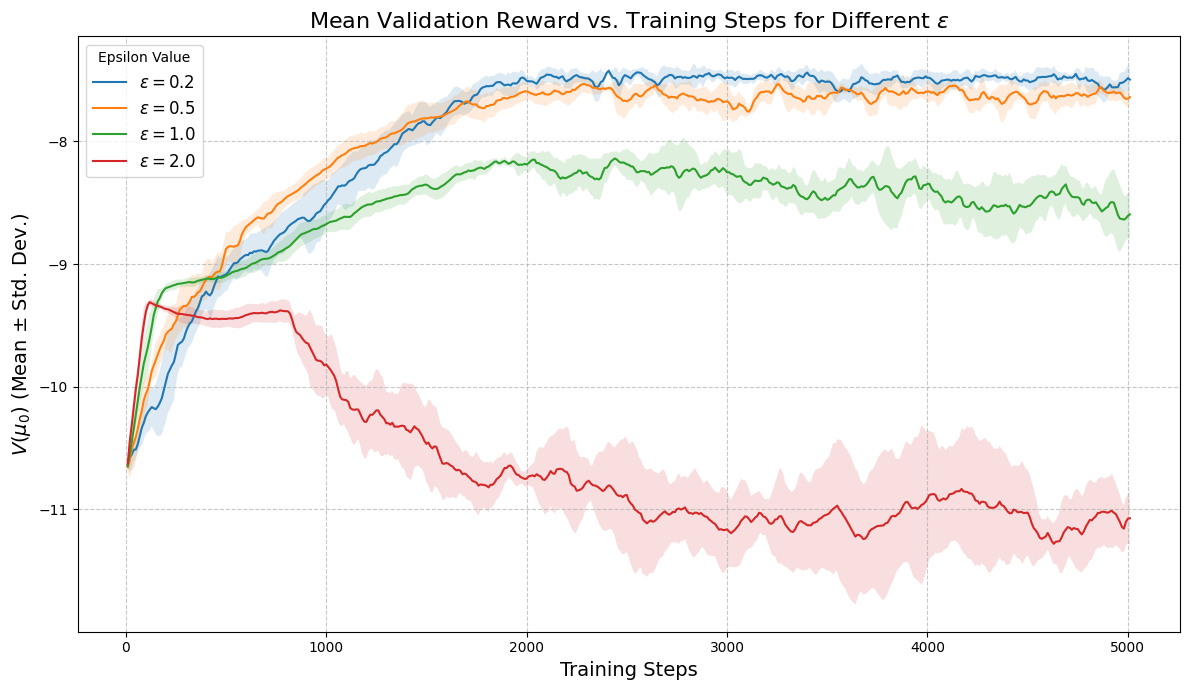}
    \caption{Evolution of validation rewards on two-state two-action problem with MF-REINFORCE for different values of $\eps$. The means and standard deviations are computed over 5 independent training runs.}
    \label{fig:example31-validation}
\end{figure}

\begin{table}[h!]
    \centering
    \caption{Average absolute error of learned policy for two-state two-action problem over 5 independent runs. 
    For reference, we have $\pi^\star( \mathrm{ST} \mid 0) = 0.2$, $\pi^\star( \mathrm{ST} \mid 1) = 0.25$.}
    \label{tab:policy-error-example31}
    \begin{tabular}{|l|c|c|c|c|}
        \hline
        & $\eps = 0.2$ & $\eps = 0.5$ & $\eps = 1.0$ & $\eps = 2.0$\\
        \hline
        Error for $\pi^\star(\mathrm{ST} \mid 0)$ & $0.0103$ & $0.0182$ & $0.0953$ & $0.3109$ \\
        \hline
        Error for $\pi^\star(\mathrm{ST} \mid 1)$ & $0.0051$ & $0.0104$ & $0.0176$ & $0.0538$ \\
        \hline
    \end{tabular}
\end{table}

\subsection{Cybersecurity Example}

We consider now an example inspired by the cybersecurity model first studied in \cite{kolokoltsov2016mean}. This problem has been adapted from a mean-field game to mean-field control (e.g.\ from the perspective of a company or state trying to decide on the best policy to protect a population of computers) in \cite{carmona2023model}, Section 6.1.
In this model, we consider a large population of computers, each of which can be in one of the following four states: defended and infected (DI), defended and susceptible (DS), undefended and infected (UI), undefended and susceptible (US). 
The action space is $\cA = \{0,1\}$, where $0$ means keeping the same level of protection (D or U) for the computer, and $1$ means updating the level of protection (from D to U or vice versa). 
When infected, the computer recovers at rate $q_{\mathrm{rec}}^D$ or $q_{\mathrm{rec}}^U$ depending on the level of protection.
When susceptible, the computer might be infected either directly by a hacker or by the other infected computers, at rates depending on the level of protection of the susceptible computer and the other infected computers.
Originally, this problem was formulated in continuous time, and the infinitesimal generator matrix can be written as follows, given the current distribution of the population $\mu$ and action chosen by the computer $a$,
\[
Q^{\mu, a} = 
\begin{pmatrix}
    \ldots & q_{\mathrm{rec}^D} & \lambda a & 0 \\
    Q_{DS \to DI}^{\mu, a} & \ldots & 0 & \lambda a \\
    \lambda a & 0 & \ldots & q_{\mathrm{rec}}^U \\
    0 & \lambda a & Q_{US \to UI}^{\mu, a} & \ldots
\end{pmatrix},
\]
where
\[
\begin{aligned}
Q_{DS \to DI}^{\mu, a} &= v_H q_{\mathrm{inf}}^D + \beta_{DD} \mu (DI) + \beta_{UD} \mu (UI), \\
Q_{US \to UI}^{\mu, a} &= v_H q_{\mathrm{inf}}^U + \beta_{UU} \mu (UI) + \beta_{DU} \mu (DI), \\
\end{aligned}
\]
and all the instances of $\ldots$ should be replaced by values to make each row sum to $0$.
Each computer incurs a cost whenever it is being defended or infected.
The individual cost per unit of time is given by $f(x) = k_D \mathbf 1_{x \in \{DI,DS\}} + k_I \mathbf 1_{ x \in \{DI, UI\}}$.
We adapt this model to a discrete-time setting by considering a time grid $0, \Delta t, 2 \Delta t, \ldots, N_{\mathrm{epi}} \Delta t$. Between two points of the time grid, the transition matrix is given by $P^{\mu, a}_{\Delta t} = \exp(\Delta t Q^{\mu, a})$, and the running reward is $r_{\Delta t}(x,a,\mu) = r_{\Delta t}(x) = - \Delta t f(x)$.
In our numerical experiments we choose the following values for the parameters:
\[
\begin{aligned}
\beta_{U U}&=0.3, \quad  \beta_{U D}=0.4, \quad \beta_{D U}=0.3, \quad \beta_{D D}=0.4, \\
q_{\mathrm{rec}}^D&=0.5, \quad  q_{\mathrm{rec}}^U=0.4, \quad q_{\mathrm{inf}}^D=0.4, \quad q_{\mathrm{inf}}^U=0.3, \\
v_H&=0.6, \quad \lambda=0.8, \quad k_D=0.3, \quad k_I=0.5, \quad \Delta t = 0.2.
\end{aligned}
\]
Note that these are the same parameter values as in the MFC version considered in \cite{carmona2023model}, Section 6.1, except for the time discretization step. Finally, the rewards are discounted by $\gamma = 0.5$ and the terminal reward is taken to be the same as the running reward. The training and evaluation are done as follows.
As the transition probability is stationary and the terminal reward is the same as the step reward, the length of a training episode is taken to be $T_{\mathrm{train}} = 3$ to mitigate the `curse of time' (see Remark \ref{remark:bias-logit-estimation}), and for each training episode we simulate the population starting from a random initial distribution $\mu\in \mathcal P (\cX)^*$. 
Every 10 training episodes, we freeze the policy and sample a validation episode of length $T_{\mathrm{val}} = 50$, for which we compute the population reward $V(\mu_0)$ starting from a fixed initial distribution $\mu_0 = (1/4,1/4,1/4,1/4)$.
During training, we compute $N = 200$ trajectories for MF-REINFORCE, with $n = 1$ trajectory for the gradient of logits estimation. 
The policy is given by a 2-layer MLP, with 32 hidden units and $\tanh$ activation; it takes as input $t, \mu$ and outputs a $\mathcal |\cX| \times |\cA|$ matrix corresponding to the probability of each action given each state.
We use the Adam optimizer and train for 20,000 episodes using a learning rate of $10^{-3}$. 
The experiments are run for $\eps$ ranging in the set $\{0.2, 0.5, 1.0, 2.0\}$. 
The validation rewards for the different values of $\eps$ are shown in Figure \ref{fig:cybersecurity-validation}.
We clearly see the effect of $\eps$ on the variance of the overall training algorithm: the smaller the perturbation, the more erratic the gradient estimation becomes, which leads to slower convergence.
Surprisingly, taking a smaller value of epsilon does not lead to a better policy. To confirm that the policy obtained is not suboptimal, we compare the state distribution evolution under the learned policy to the results obtained in \cite{carmona2023model} with mean-field $Q$-learning. 
The resulting flow of distributions is shown in Figure \ref{fig:cybersecurity-state}. Clearly, we see that the qualitative behaviour of the population under the policy learned using MF-REINFORCE eventually matches the one learned using mean-field $Q$-learning in \cite{carmona2023model}. 

\begin{figure}[!ht]
    \centering
    \includegraphics[width=0.8\linewidth]{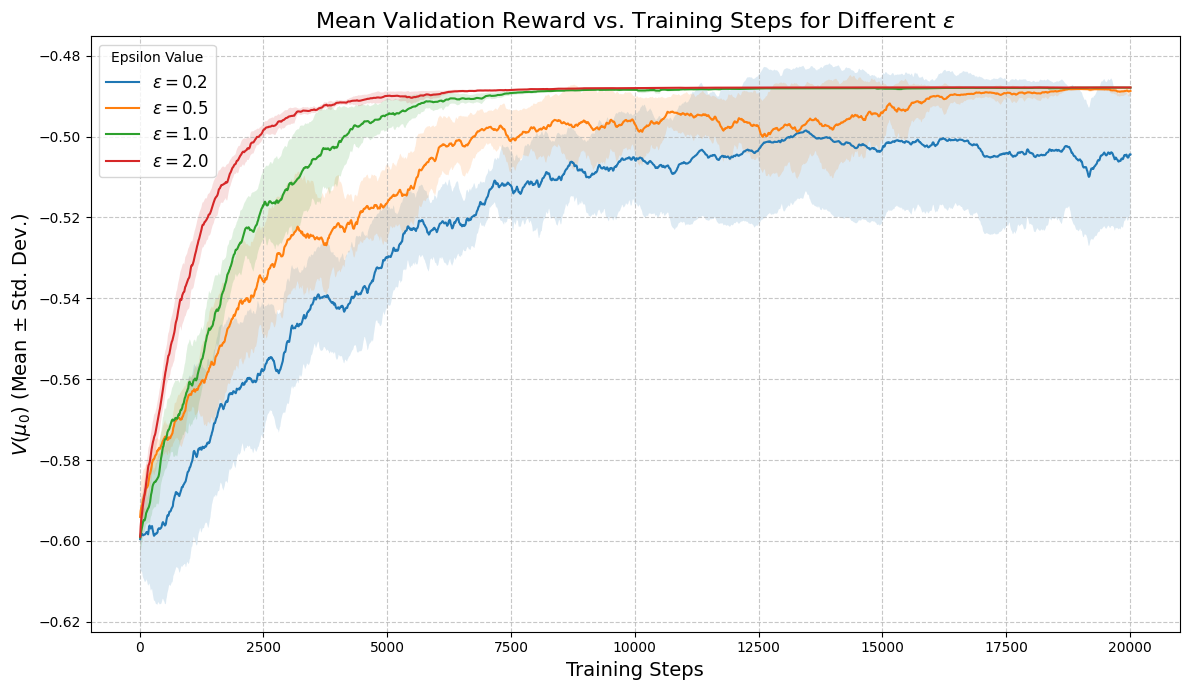}
    \caption{Evolution of validation rewards on cybersecurity problem with MF-REINFORCE for different values of $\eps$. The means and standard deviations are computed over 5 independent training runs.}
    \label{fig:cybersecurity-validation}
\end{figure}

\begin{figure}[!ht]
    \centering
    \includegraphics[width=0.8\linewidth]{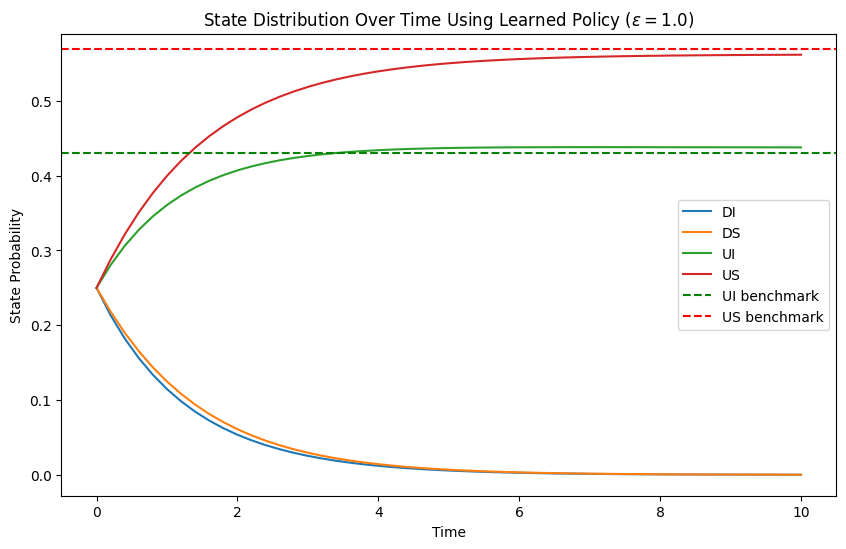}
    \caption{Flow of state distribution on cybersecurity problem using the learned policy trained with MF-REINFORCE ($\eps = 1.0$). The dashed lines correspond to the reported state probabilities at the final time step learned via mean-field $Q$-learning \citep{carmona2023model}.}
    \label{fig:cybersecurity-state}
\end{figure}

\subsection{Distribution Planning Example}

Finally, we consider a problem similar to the distribution planning example of \cite{carmona2023model}, Section 6.2.
The goal is to match a fixed target distribution $\mu_{\mathrm{target}} \in \mathcal P (\cX)$, where $\cX = \mathbb Z / 10 \mathbb Z$. The action space is $\cA = \{-1,0,+1\}$, and the transition probability is deterministic $P(\cdot \mid x, a, \mu) = \delta_{x + a}$. 
At each step, we penalize the deviation of the population from the target distribution and penalize when an agent chooses to move. Specifically, the step reward is $r(x,a,\mu) =- \lambda|a| - \sum_{x \in \cX} |\mu(x) - \mu_{\mathrm{target}}(x)|^2$, where $\lambda > 0$, and the terminal reward is $g(x,\mu) = - \sum_{x \in \cX} |\mu(x) - \mu_{\mathrm{target}}(x)|^2$.
In this example, we take the length of episodes to be $T = 5$ (for both training and validation) and $\lambda=0.01$.
For each training episode, we simulate the population starting from a random initial distribution $\mu\in \mathcal P (\cX)^*$.
Every 10 training episodes, we freeze the policy and sample a validation episode, for which we compute the population reward $V(\mu_0)$ starting from a fixed initial distribution $\mu_0 = \mathcal U(\cX)$.
During training, we compute $N = 500$ trajectories for MF-REINFORCE, with $n = 10$ trajectories for the gradient of logits estimation. 
The policy is given by a 2-layer MLP, with 256 hidden units and $\tanh$ activation.
We use the Adam optimizer and train for 100,000 episodes using a learning rate of $10^{-4}$. 
The experiments are run for $\eps$ ranging in the set $\{0.5, 0.75, 1.0, 2.0\}$. In this example, we see that taking a large value of $\eps$, i.e., $\eps \in \{1,2\}$, leads to more stable training and faster convergence, as shown in Figure \ref{fig:distplanning-validation}. However, the resulting trajectories still differ qualitatively between $\eps=1.0$ and $\eps=2.0$, as seen in Figure \ref{fig:distplanning-final-probs}.

\begin{figure}[!ht]
    \centering
    \includegraphics[width=0.8\linewidth]{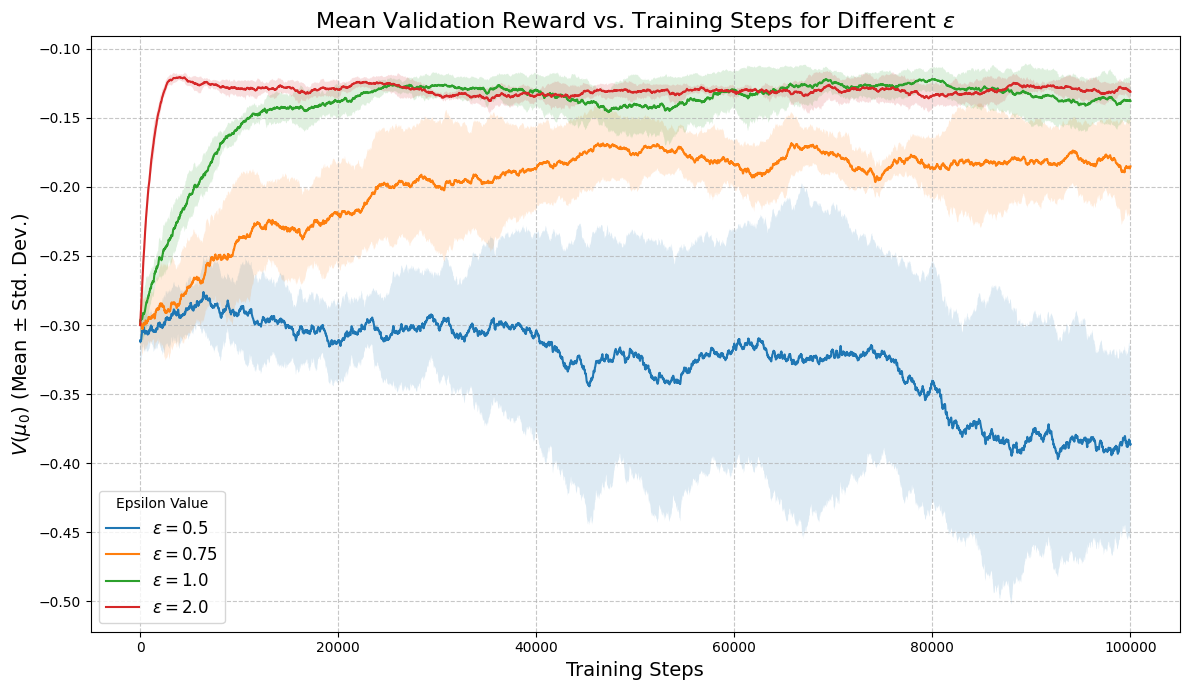}
    \caption{Evolution of validation rewards on distribution planning problem with MF-REINFORCE for different values of $\eps$. The means and standard deviations are computed over 5 independent training runs.}
    \label{fig:distplanning-validation}
\end{figure}

\begin{figure}[htbp]
     \centering
     \begin{subfigure}[b]{0.45\textwidth}
         \centering
         \includegraphics[width=\textwidth]{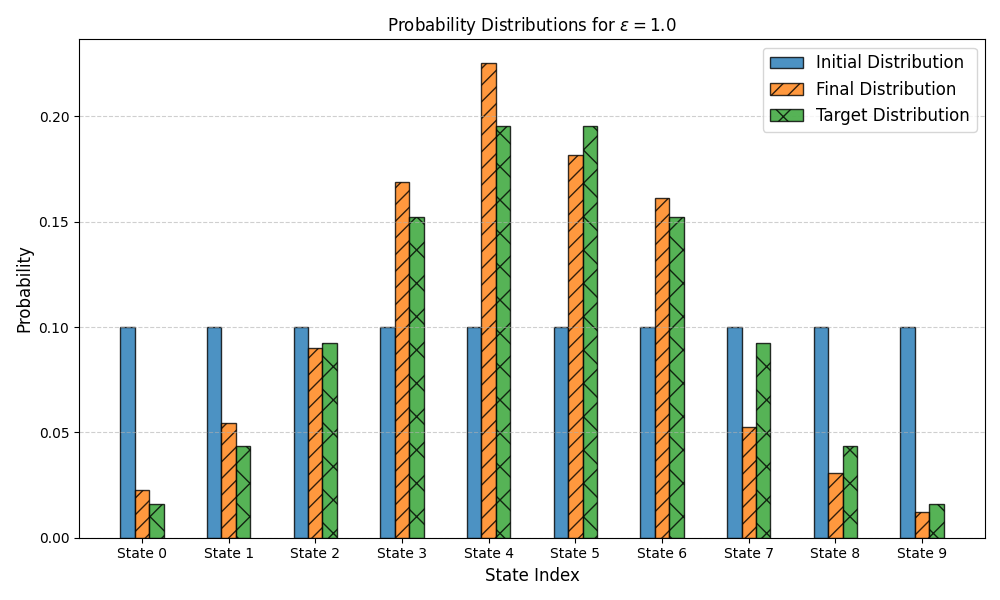}
         \caption{$\eps=1$}
         \label{fig:final-probs_1}
     \end{subfigure}
     \hfill 
     \begin{subfigure}[b]{0.45\textwidth}
         \centering
         \includegraphics[width=\textwidth]{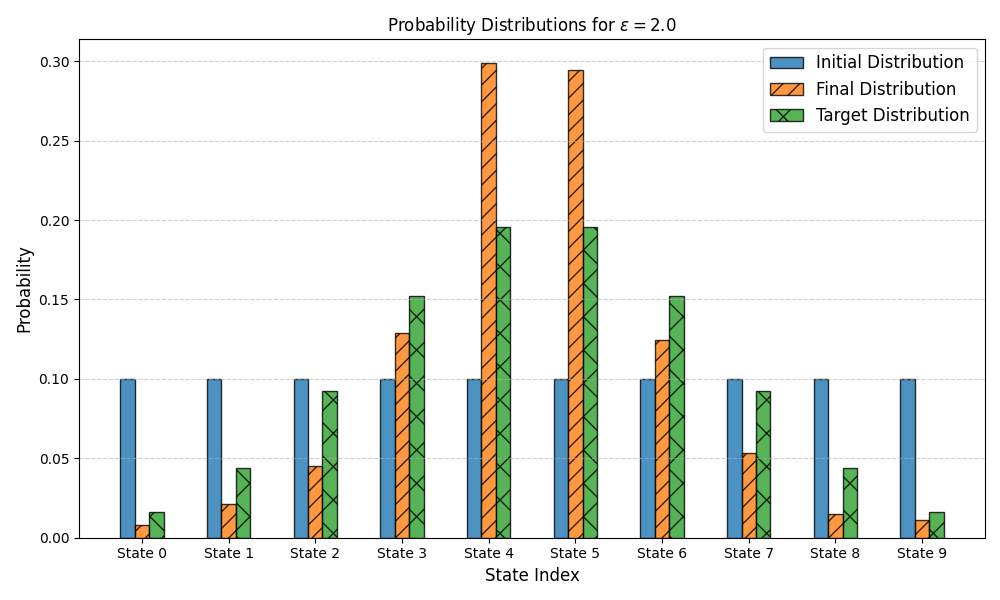}
         \caption{$\eps=2$}
         \label{fig:final-probs_2}
     \end{subfigure}
     
     \caption{Initial, target, and final distributions using MF-REINFORCE. The final distribution is obtained by generating a trajectory of length $T=5$ after training MF-REINFORCE with different values of $\eps$.}
     \label{fig:distplanning-final-probs}
\end{figure}

\section{Proofs of Main Results}
\label{section:proofs}

\subsection{Proofs of Results of Section \ref{section:pg-mfrl}}

\begin{proof}[Proof of Proposition \ref{proposition:exact-policy-gradient-formula}]
By the chain rule and Assumption \ref{assumption:finite-state-space}, one can differentiate under the integral sign to get
\begin{equation}
\label{eq:grad-vf-start}
\begin{aligned}
    \nabla_{\theta}V(\mu_0, \theta) &= \nabla_{\theta} \int \left(\sum_{t = 0}^{T - 1} r(x_t, a_t, \mu_t^\theta)  + g(x_T, \mu_T^\theta)\right) P_{\theta} (\d x_0, da_0, \ldots, \d a_{T - 1}, \d x_T) \\
    &= \int \left(\sum_{t = 0}^{T - 1} \nabla_{l} \tilde r(x_t, a_t, l_t^\theta) \nabla_\theta l_t^\theta  + \nabla_{l} \tilde g(x_T, l_T^\theta)\nabla_{\theta}l_T^\theta \right) P_{\theta} (\d x_0, da_0, \ldots, \d a_{T - 1}, \d x_T) \\
    &\quad + \int \left(\sum_{t = 0}^{T - 1} r(x_t, a_t, \mu_t^\theta)  + g(x_T, \mu_T^\theta)\right) \nabla_{\theta} P_{\theta} (\d x_0, da_0, \ldots, \d a_{T - 1}, \d x_T).
\end{aligned}
\end{equation}
For the second term of the last equation in \eqref{eq:grad-vf-start}, the likelihood ratio trick yields
\[
\begin{aligned}
    &\int \left(\sum_{t = 0}^{T - 1} r(x_t, a_t, \mu_t^\theta)  + g(x_T, \mu_T^\theta)\right) \mu_0(\d x_0) \\
    &\quad \cdot \nabla_{\theta} \left(\prod_{\tau = 0}^{T - 1} p (\theta, a_\tau, \tau,x_{\tau}, \mu_\tau^{\theta}) P(x_{\tau + 1} \mid a_{\tau}, x_{\tau},\mu_{\tau}^{\theta}) 
    \nu_{\cA}(\d a_{\tau}) \nu_{\cX} (\d x_{\tau + 1})  \right) \\
    &\qquad= \int  \left( \sum_{t^\prime = 0}^{T - 1} \nabla_{\theta} \log p(\theta, a_{t^\prime},t^\prime,x_{t^\prime},\mu_{t^\prime}^\theta) + \left(\nabla_{l} \log \tilde p(\theta, a_{t^\prime},t^\prime,x_{t^\prime},l_{t^\prime}^\theta) \tilde P(x_{t^\prime + 1} \mid x_{t^\prime},a_{t^\prime}, l_{t^\prime}^\theta) \right) \nabla_{\theta} l_{t^\prime}^\theta \right) \\ 
    &\qquad \quad  \cdot \left(\sum_{t = 0}^{T - 1} r(x_t, a_t, \mu_t^\theta)  + g(x_T, \mu_T^\theta)\right) P_{\theta}(\d x_0, \d a_0, \ldots, \d a_{T - 1}, \d x_T),
\end{aligned}
\]
which corresponds to \eqref{eq:exact-policy-gradient-formula}.
The term
\begin{equation}
    \label{eq:classical-reinforce-term}
    \mathrm{RF}(\theta) \coloneq  \int  \left( \sum_{t^\prime = 0}^{T - 1} \nabla_{\theta} \log p(\theta, a_{t^\prime},t^\prime,x_{t^\prime},\mu_{t^\prime}^\theta) \right)
    \left(\sum_{t = 0}^{T - 1} r(x_t, a_t, \mu_t^\theta)  + g(x_T, \mu_T^\theta)\right) \d P_{\theta},
\end{equation}
corresponds to the policy gradient formula in the classical REINFORCE algorithm applied to the MDP, assuming the flow of distributions $(\mu_t^\theta)_{t = 0}^T$ is fixed.
\end{proof}

\begin{proof}[Proof of Lemma \ref{lemma:tv-bound}]
    By Pinsker's inequality, we have
    \[
    d_\tv (\mu, \mu_\eps) \leq \sqrt{\frac{1}{2} \mathrm{KL}(\mu \| \mu_\eps)}.
    \]
    We now compute the Kullback-Leibler divergence
    \[
    \begin{aligned}
       \mathrm{KL}(\mu \| \mu_\eps) &= \sum_{i = 1}^d \mu(x^{(i)}) \log \left( \frac{\mu (x^{(i)})}{\mu_{\eps}(x^{(i)})} \right) \\
       &= \sum_{i = 1}^d \mu(x^{(i)}) \log \mu (x^{(i)}) - \sum_{i = 1}^d \mu(x^{(i)}) \log \mu_{\eps}(x^{(i)}) \\
         &= \sum_{i = 1}^d \mu(x^{(i)}) \log \mu (x^{(i)}) - 
         \sum_{i = 1}^d \mu(x^{(i)}) \left( \log \mu(x^{(i)}) + \eps \lambda_i - \log \left( \sum_{j = 1}^d e^{\log \mu(x^{(j)}) + \eps \lambda_j} \right) \right) \\ 
         &= - \eps \sum_{i = 1}^d \mu(x^{(i)}) \lambda_i + \log \left( \sum_{j = 1}^d \mu(x^{(j)}) e^{\eps \lambda_j} \right).
    \end{aligned}
    \]
    By Jensen's inequality, we have
    \[
    \begin{aligned}
    \mathbb E \left[ \sqrt{ \frac{1}{2}\mathrm{KL}(\mu \| \mu_\eps)} \right] &\leq \sqrt{ \frac{1}{2} \mathbb E \left[ \mathrm{KL}(\mu \| \mu_\eps) \right]} \\
    & = \sqrt{ \frac{1}{2} \mathbb E \left[ - \eps \sum_{i = 1}^d \mu(x^{(i)}) \lambda_i + \log \left( \sum_{j = 1}^d \mu(x^{(j)}) e^{\eps \lambda_j} \right) \right]} \\
    & = \sqrt{ \frac{1}{2} \mathbb E \left[ \log \left( \sum_{j = 1}^d \mu(x^{(j)}) e^{\eps \lambda_j} \right) \right] }.
    \end{aligned}
    \]
    Using Jensen's inequality again, we have
    \[
    \mathbb E \left[ \log \left( \sum_{j = 1}^d \mu(x^{(j)}) e^{\eps \lambda_j} \right) \right] 
    \leq \log \left( \sum_{j = 1}^d \mu(x^{(j)}) \mathbb E \left[ e^{\eps \lambda_j} \right] \right)
    = \log \left( \sum_{j = 1}^d \mu(x^{(j)}) e^{\frac{\eps^2}{2}} \right) = \frac{\eps^2}{2}.
    \]
    Hence $d_\tv (\mu, \mu_\eps) \leq \sqrt{\frac{\eps^2}{4}} = \frac{\eps}{2}$, which concludes the proof.
\end{proof}

\begin{proof}[Proof of Theorem \ref{theorem:perturbed-gradient}]
By performing the change of variables $\tilde{\lambda}_t = l^\pi_t + \eps \lambda_t$, the perturbed value function \eqref{eq:def-perturbed-value-function} can be written as 
\[
\begin{aligned}
    V_{\eps}^{\pi}(\mu_0) &= \int \left(\sum_{t = 0}^{T - 1} \tilde r(y_t, a_t, \tilde \lambda_t) + \tilde g(y_T, \tilde \lambda_T) \right) \prod_{\tau= 0}^{T }\frac{1}{(2 \pi)^{d/2}\eps^{d}}\exp \left( - \frac{ \| \tilde \lambda_\tau - l_\tau^\pi \|^2}{2 \eps^{2}} \right) \\
    & \quad \cdot \mu_0(\d y_0) \prod_{\tau = 0}^{T - 1} \mathrm{Leb}(\d \tilde \lambda_\tau) \tilde \pi (\d a_\tau \mid \tau,y_\tau, \tilde \lambda_\tau) \tilde P (\d y_{\tau + 1} \mid y_\tau, a_\tau, \tilde{\lambda}_\tau).
\end{aligned}
\]
Now, let's consider a policy $\pi_\theta$ for $\theta \in \R^D$. We write $V_\eps(\mu,\theta) = V_\eps^{\pi_\theta}(\mu)$, $\mu_t^\theta = \mu_t^{\pi_\theta}$, $l_t^\theta = l_t^{\pi_\theta}$, $\alpha_t^{\theta, \eps} = \alpha_t^{\pi_{\theta}, \Lambda, \eps}$ $Y_t^{\theta,\eps} =Y_t^{\pi_{\theta},\Lambda, \eps} $.
The differentiability of $V_\eps(\mu, \theta)$ follows from the differentiability of $\tilde p$ in $\theta$ and $l$ and the differentiability of $\tilde P$ in $l$.
Then, given the assumptions on $r,g,\pi_\theta$, one can differentiate under the integral sign to get
\[
\begin{aligned}
    \nabla_\theta V_{\eps}(\mu_0, \theta) &= 
    \int \left(\sum_{t = 0}^{T - 1} \tilde r(y_t, a_t, \tilde \lambda_t) + \tilde g(y_T, \tilde \lambda_t) \right)  \\
    &\quad  \cdot \sum_{t = 0}^{T } \eps^{-2} (l_t^\theta - \tilde \lambda_t) \nabla_{\theta} l_t^\theta
    \prod_{\tau = 0}^{T} \frac{1}{(2 \pi)^{(d-1)/2}\eps^{d-1}}\exp \left( - \frac{\| \tilde \lambda_\tau -  l_\tau^\theta \|^2}{2 \eps^{2}} \right) \\
    & \quad  \cdot  \mu_0(\d y_0) \prod_{\tau = 0}^{T - 1} \mathrm{Leb}(\d \tilde \lambda_\tau) \tilde \pi_{\theta} (\d a_\tau \mid \tau,y_\tau, \tilde \lambda_\tau) \tilde P (\d y_{\tau + 1} \mid y_\tau, a_\tau, \tilde{\lambda}_\tau) \\
    &\quad + \int \left(\sum_{t = 0}^{T - 1} \tilde r(y_t, a_t, \tilde \lambda_t) + \tilde g(y_T, \tilde \lambda_T) \right)\prod_{\tau = 0}^{T} \frac{1}{(2 \pi)^{d/2}\eps^{d}}\exp \left( - \frac{ \| \tilde \lambda_\tau - \mu_\tau^\theta \|^2}{2 \eps^{2}} \right) \\
    & \quad \cdot \nabla_{\theta} \left(\mu_0(\d y_0) \prod_{\tau = 0}^{T - 1} \mathrm{Leb}(\d \tilde \lambda_\tau) \tilde \pi_{\theta} (\d a_\tau \mid \tau,y_\tau, \tilde \lambda_\tau) \tilde P (\d y_{\tau + 1} \mid y_\tau, a_\tau, \tilde{\lambda}_\tau) \right).
\end{aligned}
\]
We have
\[
\begin{aligned}
&\nabla_{\theta} \left(\mu_0(\d y_0) \prod_{\tau = 0}^{T - 1} \mathrm{Leb}(\d \tilde \lambda_\tau) \tilde \pi_{\theta} (\d a_\tau \mid \tau,y_\tau, \tilde \lambda_\tau) \tilde P (\d y_{\tau + 1} \mid y_\tau, a_\tau, \tilde{\lambda}_\tau) \right) \\
&\qquad= \mu(\d y_0) \sum_{t = 0}^{T - 1} \nabla_{\theta} \log \tilde p (\theta, a_t, t,y_t, \tilde \lambda_t) \prod_{\tau =0}^{T - 1} \mathrm{Leb}(\d \tilde \lambda_\tau) \tilde \pi_{\theta} (\d a_\tau \mid \tau,y_\tau, \tilde \lambda_\tau) \tilde P (\d y_{\tau + 1} \mid y_\tau, a_\tau, \tilde{\lambda}_\tau).
\end{aligned}
\]
Hence

\begin{equation*}
\begin{aligned}
\nabla_{\theta} V_\eps(\mu_0, \theta) &= \mathbb E \Bigg[\left( \sum_{t = 0}^{T - 1} \tilde r(Y^{\theta,\eps}_t,  \alpha_t^{\theta,\eps},  l_t^\theta + \eps \Lambda_t) + \tilde g(Y_T^{\theta, \eps}, l_T^\theta+ \eps \Lambda_T) \right) \\
& \quad \cdot \left( \sum_{t = 0}^{T} \eps^{-1} \Lambda_t \left(\nabla_{\theta} l_t^\theta\right)\right)\Bigg] \\
&\quad + \mathbb E \Bigg[\left( \sum_{t = 0}^{T - 1} \tilde r(Y^{\theta,\eps}_t, \alpha_t^{\theta,\eps}, l_t^\theta + \eps \Lambda_t) + \tilde g(Y_T^{\theta, \eps},  l_T^\theta + \eps \Lambda_T) \right)  \\
& \quad \cdot \left( \sum_{t = 0}^{T - 1} \nabla_\theta \log \tilde p(\theta, \alpha_t^{\theta,\eps}, t, Y_t^{\theta,\eps}, l_t^\theta + \eps \Lambda_t \right)  \Bigg].
\end{aligned}
\end{equation*}
Finally, one can apply the causality trick to only consider the rewards from time $t$ to $T$ in the expectations, leading to \eqref{eq:policy-gradient-formula}.
\end{proof}

\begin{proof}[Proof of Theorem \ref{theorem:convergence-perturbed-gradient}]
Let's decompose the expression of the gradient of the perturbed value function \eqref{eq:policy-gradient-formula},
\begin{equation}
    \label{eq:policy-gradient-formula-decomposition}
    \nabla_{\theta} V_\eps(\mu_0, \theta) = \mathbb E \left[ \sum_{t = 0}^{T} \eps^{-1} \Lambda_t \nabla_{\theta} l_t^\theta G_t^\eps \right] + \mathbb E \left[ \sum_{t=0}^{T-1} \nabla_\theta \log \tilde p(\theta, \alpha_t^{\theta,\eps}, t,Y_t^{\theta,\eps}, l_t^\theta + \eps \Lambda_t) G_t^\eps \right].
\end{equation}

We first look at the second expectation in \eqref{eq:policy-gradient-formula-decomposition}. 
Clearly, since $Y_t^{\theta,\eps} \to X_t^\theta$ and $\alpha_t^{\theta, \eps} \to \alpha_t^\theta$ in distribution as $\eps \to 0$, we have, by expanding the expression of $G_t^\eps$, 
\[
\begin{aligned}
 &\mathbb E \Bigg[\left( \sum_{t = 0}^{T - 1} \tilde r(Y^{\theta,\eps}_t, \alpha_t^{\theta,\eps}, l_t^\theta + \eps \Lambda_t) + \tilde g(Y_T^{\theta, \eps},  l_T^\theta + \eps \Lambda_T) \right) \left( \sum_{t = 0}^{T - 1} \nabla_\theta \log \tilde p(\theta, \alpha_t^{\theta,\eps}, t,Y_t^{\theta,\eps}, l_t^\theta + \eps \Lambda_t \right)  \Bigg] \\
 &\qquad \to \mathbb E \Bigg[\left( \sum_{t = 0}^{T - 1} r(X^{\theta}_t, \alpha_t^{\theta}, \mathbb P_{X_t^\theta}) + g(X_T^{\theta, \eps},  \mathbb P_{X^{\theta}_T} ) \right) \left( \sum_{t = 0}^{T - 1} \nabla_\theta \log p(\theta, \alpha_t^{\theta}, t,X_t^{\theta}, \mathbb P_{X^{\theta}_t}) \right)  \Bigg] = \mathrm{RF}(\theta).
\end{aligned}
\]
We now look at the first expectation in \eqref{eq:policy-gradient-formula-decomposition} as $\eps \to 0$, which we rewrite using the following lemma.

\begin{lemma}
    \label{lemma:density-ratio-trick}
    Let $X,Y$ be random variables valued in a measurable space $E$, and let $\varphi: E \to \R$ be a bounded measurable function. Assume $X,Y$ have positive densities $p_X,p_Y$ with respect to a reference measure $q \in \mathcal P (E)$, and that $p_X$ is bounded away from $0$. Then $\mathbb E [\varphi(Y)] = \mathbb E \left[\varphi(X) \frac{p_Y(X)}{p_X(X)}\right]$.
\end{lemma}
We apply Lemma \ref{lemma:density-ratio-trick} to $X = \left(X_0^\theta, \alpha_0^\theta, \ldots, \alpha_{T-1}^\theta, X_{T}^\theta\right)$ and $Y = \left(Y_0^{\theta, \eps}, \alpha_0^{\theta, \eps}, \ldots, \alpha_{T-1}^{\theta, \eps}, Y_T^{\theta, \eps} \right)$ to get
\[
\begin{aligned}
    M_{\eps}(\theta) & \coloneq \mathbb E \left[\left( \sum_{t = 0}^{T - 1} \tilde r(Y^{\theta,\eps}_t,  \alpha_t^{\theta,\eps},  l_t^\theta + \eps \Lambda_t) + \tilde g(Y_T^{\theta, \eps}, l_T^\theta + \eps \Lambda_T) \right) 
    \cdot \left( \sum_{t = 0}^{T} \eps^{-1} \Lambda_t \nabla_{\theta} l_t^\theta \right)\right] \\
    &= \mathbb E \left[\left( \sum_{t = 0}^{T - 1} \tilde r(X_t^\theta,  \alpha_t^{\theta},  l_t^\theta + \eps \Lambda_t) + \tilde g(X_T^{\theta}, l^{\theta}_T + \eps \Lambda_T) \right) 
    \cdot \left( \sum_{t = 0}^{T} \eps^{-1} \Lambda_t \nabla_{\theta} l_t^\theta \right) R_{T}^{\theta, \eps}\right],
\end{aligned}
\]
where the density ratio is
\begin{equation}
    \label{eq:likelihood-ratio-T}
R_{T}^{\theta, \eps} = 
\prod_{\tau = 0}^{T-1} 
\frac{\tilde p(\theta, \alpha_\tau^\theta,\tau,X_\tau^\theta, l_\tau^\theta + \eps \Lambda_{\tau})
\tilde P(X_{\tau + 1}^\theta \mid X_{\tau}^\theta,\alpha_{\tau}^\theta, l_\tau^\theta + \eps \Lambda_\tau)}
{p(\theta, \alpha_\tau^\theta, \tau,X_{\tau}^{\theta}, \mathbb P_{X_{\tau}^\theta}) P (X_{\tau + 1}^\theta \mid X_{\tau}^\theta, \alpha_\tau^\theta,\mathbb P_{X_\tau^{\theta}})}.
\end{equation}
Since $l \mapsto \tilde P (x^\prime \mid x, a, l) \tilde p(\theta, a , t,x, l)$ is twice differentiable with uniformly bounded second derivative,
using Taylor's formula with Lagrange remainder as $\eps \to 0$ we have, almost-surely
\[
\begin{aligned}
R_T^{\theta,\eps} &= \prod_{\tau = 0}^{T-1} \Big[\left(1 + \eps \nabla_{l} \log \tilde p(\theta, \alpha_\tau^\theta,\tau,X_\tau^{\theta}, l_\tau^\theta ) \cdot \Lambda_\tau + o(\eps)\right) \\
&\quad \cdot \left(1 + \eps \nabla_{l} \log \tilde P(X_{\tau + 1}^\theta \mid  X_\tau^{\theta},\alpha_\tau^\theta, l_\tau^\theta) \cdot \Lambda_\tau + o(\eps)\right) \Big] \\
&= 1 + \eps \sum_{\tau = 0}^{T - 1} \left(\nabla_{l} \log \tilde P(X_{\tau + 1}^\theta \mid  X_\tau^{\theta},\alpha_\tau^\theta, l_\tau^\theta) + \nabla_{l} \log \tilde p(\theta, \alpha_\tau^\theta,\tau,X_\tau^{\theta}, l_\tau^\theta) \right) \cdot \Lambda_\tau    + o (\eps),
\end{aligned}
\]
where $o(\eps)$ corresponds to the Lagrange remainder uniformly bounded in $\eps$ and converging to $0$ almost surely as $\eps \to 0$.
This means that as $\eps \to 0$, by the dominated convergence theorem, $M_\eps(\theta)$ can be written as
\[
\begin{aligned}
M_{\eps}(\theta) &= \mathrm{MD}_{\eps}(\theta) + \mathrm{MFD}_{\eps}(\theta) + o(1), \quad \text{where} \\
\mathrm{MD}_{\eps}(\theta) & \coloneq \mathbb E \left[\left( \sum_{t = 0}^{T - 1} \tilde r(X_t^\theta,  \alpha_t^{\theta},  l_t^\theta + \eps \Lambda_t) + 
\tilde g(X_T^{\theta}, l^{\theta}_T + \eps \Lambda_T) \right) \cdot \left( \sum_{s = 0}^{T} \eps^{-1} \Lambda_{s} \nabla_{\theta} l_s^{\theta}  \right)\right],\\
\mathrm{MFD}_{\eps}(\theta) &\coloneq \mathbb E \Bigg[\left( \sum_{t = 0}^{T - 1} \tilde r(X_t^\theta,  \alpha_t^{\theta},  l_t^\theta + \eps \Lambda_t) + 
\tilde g(X_T^{\theta}, l^{\theta}_T + \eps \Lambda_T) \right) \cdot \left( \sum_{s = 0}^{T} \Lambda_s \nabla_{\theta} l_s^\theta   \right) \\
& \quad \cdot \left( \sum_{\tau = 0}^{T - 1} \left(\nabla_{l} \log \tilde P(X_{\tau + 1}^\theta \mid  X_\tau^{\theta},\alpha_\tau^\theta, l_\tau^\theta) + \nabla_l \log \tilde p(\theta, \alpha_\tau^\theta,\tau,X_\tau^{\theta}, l_\tau^\theta) \right) \cdot \Lambda_{\tau} \right) \Bigg].
\end{aligned}
\]
The off-diagonal terms in $\mathrm{MD}_{\eps}(\theta)$ cancel by independence of $(\Lambda_t)_{t = 0}^{T}$, hence
\[
    \mathrm{MD}_{\eps}(\theta) = 
    \eps^{-1}\mathbb E \left[ \sum_{t = 0}^{T - 1} r(X_t^\theta, \alpha_t^\theta,l_t^\theta + \eps \Lambda_t) \Lambda_t \nabla_{\theta}l_t^\theta  + g(X_{T}^\theta, l_t^\theta + \eps \Lambda_T) \Lambda_T \nabla_{\theta}l_t^\theta  \right].
\]
We now examine the term $t$ in the above sum.
By independence of $\Lambda_s$ and $(X_s^\theta, \alpha_s^\theta)$ we have
\[
\mathbb E [r(X_t^\theta, \alpha_t^\theta,\mathbb P_{X_t}^\theta) \Lambda_t] = \mathbb E[r(X_t^\theta, \alpha_t^\theta,\mathbb P_{X_t}^\theta)] \mathbb E[\Lambda_t]=0.
\]
Hence,
\[
\begin{aligned}
    &\mathbb E \left[\eps^{-1} \tilde r(X_t^\theta, \alpha_t^\theta,l_t^\theta + \eps \Lambda_t) \Lambda_t \nabla_{\theta}l_t^\theta  \right] \\
    &\qquad= \mathbb E \left[\eps^{-1} (\tilde r(X_t^\theta, \alpha_t^\theta,l_t^\theta + \eps \Lambda_t) - \tilde r(X_t^\theta, \alpha_t^\theta, l_t^\theta)) \Lambda_t \nabla_{\theta}l_t^\theta   \right] \\
    &\qquad= \mathbb E\left[ \nabla_l \tilde r(X_t^\theta, \alpha_t^\theta, l_t^\theta) \cdot \Lambda_t \left(\Lambda_t \nabla_{\theta}l_t^\theta \right)\right] + o(1) \\
    &\qquad= \mathbb E[\nabla_l \tilde r(X_t^\theta, \alpha_t^\theta, l_t^\theta)] \mathbb E [\Lambda_t^T \Lambda_t] \nabla_{\theta} l_t^\theta + o (1) \\
    &\qquad= \mathbb E[\nabla_l \tilde r(X_t^\theta, \alpha_t^\theta, l_t^\theta)] \nabla_{\theta} l_t^\theta + o(1).
\end{aligned}
\]
The same argument can be applied to the other terms in $\mathrm{MD}_{\eps}(\theta)$, which proves that $\mathrm{MD}_{\eps}(\theta) \to \mathrm{MD}(\theta)$ as $\eps \to 0$.

We now examine $\mathrm{MFD}_{\eps}(\theta)$, let
\[
 \tilde G^{\eps,\theta}_0 = \sum_{t = 0}^{T - 1} \tilde r(X_t^\theta,  \alpha_t^{\theta},  l_t^\theta + \eps \Lambda_t) + \tilde g(X_T^{\theta}, l^{\theta}_T + \eps \Lambda_T).
\]
Note that $\mathbb E[\tilde G^{\eps, \theta} (\theta) \Lambda_s^T \Lambda_\tau ] = 0$ whenever $s \neq \tau$ and that $\tilde G^{\eps, \theta}(\theta)$ is a random variable converging almost surely to $G_0^\theta$, hence
\[
\begin{aligned}
    \mathrm{MFD}_{\eps}(\theta) &= \mathbb E\left[\tilde G^{\eps, \theta}_0 \sum_{s = 0}^{T - 1} \left( \nabla_l \log \tilde P(X_{\tau + 1}^\theta \mid  X_\tau^{\theta},\alpha_\tau^\theta, l_\tau^\theta) + \nabla_l \log \tilde p(\theta, \alpha_\tau^\theta,\tau,X_\tau^{\theta}, l_\tau^\theta)\right) \cdot \Lambda_s \left(\Lambda_s \nabla_\theta l_t^\theta  \right)  \right] \\
    &= \mathbb E\left[G_0^\theta \sum_{s = 0}^{T - 1} \left( \nabla_l \log \tilde P(X_{\tau + 1}^\theta \mid  X_\tau^{\theta},\alpha_\tau^\theta,  l_\tau^\theta) \tilde p(\theta, \alpha_\tau^\theta,\tau,X_\tau^{\theta}, l_\tau^\theta)\right) \cdot \Lambda_s \left( \Lambda_s \nabla_\theta l_t^\theta \right) \right] + o (1) \\
    &=\mathbb E\left[G_0^\theta\sum_{s = 0}^{T - 1} \left( \nabla_l \log \tilde P(X_{\tau + 1}^\theta \mid  X_\tau^{\theta},\alpha_\tau^\theta, l_\tau^\theta ) \tilde p(\theta, \alpha_\tau^\theta,\tau,X_\tau^{\theta}, l_\tau^\theta)\right) \nabla_\theta l_t^\theta \right] + o (1),
\end{aligned}
\]
where we have used independence of $\Lambda_s$ and $(X_s^\theta, \alpha_s^\theta)$, and $\mathbb E[ \Lambda_s^T \Lambda_s ] = I_d$. 
In particular, this means that $\mathrm{MFD}_\eps(\theta) \to \mathrm{MFD(\theta)}$ as $\eps \to 0$.
Hence we have shown that
\[
\nabla_\theta V_\eps(\mu, \theta) \rightarrow \mathrm{RF}(\theta) + \mathrm{MD}(\theta) + \mathrm{MFD}(\theta) = \nabla _\theta V(\mu, \theta) \quad \text{as } \eps \to 0.
\]
\end{proof}

\subsection{Proofs of Results of Section \ref{section:pg-algorithm}}

The following discrete Gr\"onwall lemma will be useful. To be the best of our knowledge, it was first proven as Corollary 2.3 in \cite{hutzenthaler2022multilevel}.
\begin{lemma}
    \label{lemma:discrete-gronwall}
    Let $(x_n)_{n \geq 0}$ be a sequence of non-negative real numbers such that $x_0=0$ and satisfying, for some constants $c_1, c_2, c_3, c_4, c_5, c_6 \geq 0$ such that $(c_1,c_2) \neq (0,0)$ for all $n \geq 1$,
    \[
    x_n \leq  \sum_{k=0}^{n-1} \left(c_1 x_{k} +  \mathbf 1_{k \geq 1} c_2 x_{k-1} \right) + n c_3 + c_4 \sum_{k=0}^{n-1}c_5^k + c_6. 
    \]
Then, for all $n \geq 1$,
    \[
    x_n \leq \begin{cases}
        \frac{3}{2}c_6  \beta^n +\frac{3 c_3(\beta^n - 1)}{2(\beta - 1)} + \frac{3}{2} c_4 n  \beta^n & \text{if } c_5 = \beta, \\
        \frac{3}{2}c_6  \beta^n +\frac{3 c_3(\beta^n - 1)}{2(\beta - 1)} + \frac{3 c_4 (c_5^{n+1} - c_5 \beta^n)}{2 (c_5 - \beta)} & \text{otherwise},
    \end{cases}
    \]
where $\beta = \frac{1}{2}\left(1 + c_1 + \sqrt{(1 + c_1)^2 + 4c_2} \right) > 1$.   
\end{lemma}

\begin{proof}[Proof of Lemma \ref{lemma:bound-gradient-logits}]
    Let $t \in \{0, \ldots, T\}$ and $i \in \{1, \ldots, d\}$.
    We set the following notation, for all $x^\prime,x \in \cX, a \in \cA, l \in \R^d, \theta \in \Theta, s \in \{0, \ldots, T-1\}$,
    \[
    \varphi^{(i)}(\theta, x^\prime, a, t,x, l) \coloneq \mathbf 1_{x^\prime = x^{(i)}} \nabla_\theta \log \tilde p (\theta, a, t,x, l), \quad  \psi(\theta, x^\prime, t,x, a, l) \coloneq \tilde P (x^\prime \mid x, a, l) \tilde p (\theta, a, t,x, l). 
    \]
    Let $\nabla_\theta \mu_t^{(i)} = \nabla_\theta \P(X_t^\theta = x^{(i)})$. By Proposition \ref{proposition:exact-policy-gradient-formula} with reward function $x \mapsto 1_{x = x^{(i)}}$ at time $t$ and $0$ elsewhere, 
    \[
    \begin{aligned}
    \nabla_\theta \mu_t^{\theta, i} &= \mathbb E \left[
        \mathbf 1_{X_t^\theta = x^{(i)}} \sum_{s = 0}^{t-1} \nabla_\theta \log \tilde p (\theta, \alpha_s^\theta, s,X_s^\theta, l_s^\theta) \right] \\
      & \quad  + \mathbb E \left[\mathbf 1_{X_t^\theta = x^{(i)}} \sum_{s=0}^{t-1} \left(\nabla_l \log \tilde p (\theta, \alpha_s^\theta, s,X_s^\theta, l_s^\theta) + \nabla_l \log \tilde P (X_{s + 1}^\theta \mid X_s^\theta, \alpha_s^\theta, l_s^\theta)\right) \nabla l_s^\theta  \right]\\
     &= \mathbb E \left[ \sum_{s=0}^{t-1} \varphi^{(i)}(\theta, X_t^\theta, \alpha_s^\theta, s,X_s^\theta, l_s^\theta) + \mathbf 1_{X_t^\theta = x^{(i)}}\psi(\theta, X_{s+1}^\theta, s,X_s^\theta, \alpha_s^\theta, l_s^\theta) \nabla l_s^\theta \right].
    \end{aligned}
    \]
    Recall that under Assumption \ref{assumption:finite-state-space}, $ \|\nabla_\theta \log \tilde p \|_\infty \leq C $. The first term can simply be bounded by 
    \[
    \left\| \mathbb E \left[ \mathbf 1_{X_t^\theta = x^{(i)}} \sum_{s=0}^{t-1} \nabla_\theta \log \tilde p (\theta, \alpha_s^\theta, s,X_s^\theta, l_s^\theta) \right]  \right\|_\infty \leq t C.
    \]
    The second term can be bounded by
    \[
\begin{aligned}
    &\left\| \mathbb E \left[ \mathbf 1_{X_t^\theta = x^{(i)}} \sum_{s=0}^{t-1}\nabla_l \log \psi(\theta, X_{s + 1}^\theta, s,X_s^\theta, \alpha_s^\theta, l_s^\theta) \nabla_\theta l_s^\theta \right] \right\|_\infty \\
    &\qquad = \Bigg\| \int \mathbf 1_{x_t = x^{(i)}} \sum_{s=0}^{t-1}\frac{\nabla_l \psi(\theta, x_{s + 1}, s,x_s, a_s, l_s^\theta)}{\psi(\theta, x_{s+1}, s,x_s, a_s, l_s^\theta)} \nabla_\theta l_s^\theta \\
    &\qquad \quad \cdot \left(\prod_{\tau = 0}^{t - 1} \psi(\theta, x_{\tau+1}, \tau,x_\tau, a_\tau, l_\tau^\theta) \nu_\cX (\d x_{\tau + 1}) \nu_\cA (\d a_\tau) \right) \mu_0(\d x_0) \Bigg\|_\infty\\
    &\qquad  \leq \int \mathbf 1_{x_t = x^{(i)}} \sum_{s=0}^{t-1} B_1 \left\| \nabla_\theta l_s^\theta \right\| \\ 
    & \qquad \quad \cdot \prod_{\substack{\tau = 0 \\ \tau \neq s}}^{t - 1} \psi(\theta, x_{\tau+1}, \tau,x_\tau, a_\tau, l_\tau^\theta) \nu_\cX (\d x_{\tau + 1}) \nu_\cA (\d a_\tau)  
    \nu_\cX (\d x_{s + 1}) \nu_\cA (\d a_s) \mu_0(\d x_0) \\
    &\qquad \leq  \left\| \nabla_l \psi \right\|_\infty | \cA | \sum_{s=0}^{t-1} d^{\mathbf 1_{s < t - 1}} \left\| \nabla l_s^\theta \right\|.
\end{aligned}
\]
Hence we have the recursive bound, for $t \geq 1$,
\[
\begin{aligned}
    &\left\| \nabla_\theta l_t^\theta \right\| \leq  \sum_{s = 0}^{t-1} \left(c_1 \left\| \nabla_\theta l_s^\theta \right\| + \mathbf 1_{s \geq 1} c_2 \left\| \nabla_\theta l_{s-1}^\theta   \right\| \right)   + t c_3, \\
    \text{with } &c_1 = \mathbf s^\theta B_1 | \cA | , \quad c_2 = \mathbf s^\theta B_1| \cA | (d - 1), \quad c_3 = \mathbf s^\theta C, \\
    & \mathbf s^\theta \coloneq \sup_{0 \leq t \leq T} \sum_{i=1}^d \mathbb P (X_t^\theta = x^{(i)})^{-1}.
\end{aligned}
\]
If $ B_1  > 0$, we can apply Lemma \ref{lemma:discrete-gronwall} to get
\[
\begin{aligned}
&\left\| \nabla_\theta l_t^\theta \right\|_\infty \leq \frac{3 C(\beta^t - 1)}{2(\beta - 1)}, \\
\text{with } &\beta = \frac{1}{2} \left( 1 + \mathbf s^\theta B_1 | \cA | + \sqrt{(1 + \mathbf s^\theta B_1| \cA |)^2 + \mathbf s^\theta B_1 | \cA | (d - 1)} \right).
\end{aligned}
\]
If $ B_1 = 0$, then $\left\| \nabla_\theta l_t^\theta \right\| \leq t C$.
\end{proof}

\subsubsection{Proofs on Estimation Biases}

\begin{proof}[Proof of Proposition \ref{proposition:bias-state-distribution-gradient}]
    Let $\theta \in \Theta$ and fix it for the rest of the proof.
    For the whole proof, we denote by $X_t^{\theta}, Y_t^{\theta, \eps}, \alpha_t^{\theta}, \alpha_t^{\theta, \eps}$ processes distributed according to \eqref{eq:perturbed-process} using perturbations $\Lambda_t, t = 0, \ldots, T$,
    independent of the processes and perturbations used in the estimators $\widehat{\nabla_\theta l_t^\theta}$.
    The bound is trivial for $t = 0$.
    We first assume $B_1 > 0$.
    Let $t \geq 1$.
    For $i = 1, \ldots, d$, let $\mu^{(i)}_t = \mathbb P(X_t^\theta = x^{(i)})$.
    We set the following notations for the rest of the proof:
    \begin{equation}
        \label{eq:bias-notation}
    E_t(\theta, \eps) \coloneq \left\| \mathbb E \left[ \widehat{\nabla_{\theta} l_t^\theta} \right] - \nabla_\theta l_t^\theta \right\|, \: b_t^{(i)}(\theta) \coloneq  \mathbb E \left[ \widehat{\nabla_{\theta} \mu_t^{(i)}} \right] - \nabla_\theta \mu^{(i)}_t.
    \end{equation}
    Recall from the proof of Lemma \ref{lemma:bound-gradient-logits} that, by Proposition \ref{proposition:exact-policy-gradient-formula}, we have
    \[
    \begin{aligned}
    \nabla_\theta \mu^{(i)}_t &= \mathbb E \Bigg[
        \mathbf 1_{X_t^\theta = x^{(i)}} \sum_{s = 0}^{t-1} \Big(
            \nabla_\theta \log \tilde p (\theta, \alpha_s^\theta, s,X_s^\theta, l_s^\theta) \\
            &\quad + \left(\nabla_l \log \tilde p(\theta, \alpha_s^\theta, s,X_s^\theta, l_s^\theta)
            + \nabla_l \log \tilde P(X_{s + 1}^\theta \mid X_s^\theta, \alpha_s^\theta, l_s^\theta) \right) \nabla_\theta l_s^\theta
            \Big)
            \Bigg].
    \end{aligned}
\]
On the other hand, by taking expectation in \eqref{eq:logit-gradient-estimator} we have
\[
    \mathbb E \left[ \widehat{\nabla_{\theta} \mu_t^{(i)}} \right] = \mathbb E \left[ 
         \mathbf 1_{Y_t^{\theta, \eps}=x^{(i)}} 
         \sum_{s = 0}^{t-1} \left(
            \nabla_\theta \log \tilde p (\theta, \alpha_s^{\theta, \eps}, s,Y_s^{\theta, \eps}, l_s^{\theta} + \eps \Lambda_s) 
            + \eps^{-1} \left(\widehat{\nabla l_s^\theta}\right)^T \Lambda_s 
            \right)
         \right].
\]
We split $b_t^{(i)}$ into two terms:
\begin{equation}
\label{eq:bt-decomposition}
\begin{aligned}
    b_t^{(i)}(\theta) &= b_{t,1}^{(i)}(\theta) + b_{t,2}^{(i)}(\theta),\\
    b_{t,1}^{(i)}(\theta) & \coloneq \mathbb E \left[ 
         \mathbf 1_{Y_t^{\theta, \eps} = x^{(i)}} \sum_{s = 0}^{t-1}
            \nabla_\theta \log \tilde p (\theta, \alpha_s^{\theta, \eps}, s,Y_s^{\theta, \eps}, l_s^{\theta} + \eps \Lambda_s) \right] \\
            & \quad \:\:- \mathbb E \left[ \mathbf 1_{X_t^{\theta} = x^{(i)}} \sum_{s = 0}^{t-1} 
            \nabla_\theta \log \tilde p (\theta, \alpha_s^{\theta}, s,X_s^{\theta}, l_s^{\theta}) 
    \right], \\
    b_{t,2}^{(i)}(\theta) & \coloneq \mathbb E \left[ 
        \mathbf 1_{Y_t^{\theta, \eps} = x^{(i)}} \sum_{s = 0}^{t-1} 
            \eps^{-1} \Lambda_s \widehat{\nabla_\theta l_s^\theta}  \right] \\
       & \quad \:\: - \mathbb E \left[ \mathbf 1_{X_t^{\theta} = x^{(i)}} \sum_{s = 0}^{t-1} \left(
            \nabla_l \log \tilde p(\theta, \alpha_s^\theta, s,X_s^\theta, l_s^\theta) + \nabla_l \log \tilde P(X_{s + 1}^\theta \mid X_s^\theta, \alpha_s^\theta, l_s^\theta)
            \right) \nabla_\theta l_s^\theta
    \right].
\end{aligned}
\end{equation}
We first examine $b_{t,1}^{(i)}(\theta)$,
\[
\begin{aligned}
    b_{t,1}^{(i)}(\theta) & = \Delta_1^{(i)}(\theta) + \Delta_2^{(i)}(\theta), \\
     \Delta_1^{(i)}(\theta) & \coloneq \mathbb E \left[
        \mathbf 1_{Y_t^{\theta, \eps} = x^{(i)}} \sum_{s = 0}^{t-1}  \left(
        \nabla_\theta \log \tilde p (\theta, \alpha_s^{\theta, \eps},s, Y_s^{\theta, \eps}, l_s^{\theta} + \eps \Lambda_s)
        - \nabla_\theta \log \tilde p (\theta, \alpha_s^{\theta, \eps},s, Y_s^{\theta, \eps}, l_s^{\theta}) \right) \right], \\
        \Delta_2^{(i)}(\theta) & \coloneq \mathbb E \left[
        \mathbf 1_{Y_t^{\theta, \eps} = x^{(i)}} \sum_{s = 0}^{t-1}
            \nabla_\theta \log \tilde p (\theta, \alpha_s^{\theta, \eps}, s,Y_s^{\theta, \eps}, l_s^{\theta}) 
            - \mathbf 1_{X_t^{\theta} = x^{(i)}} \sum_{s = 0}^{t-1}
            \nabla_\theta \log \tilde p (\theta, \alpha_s^{\theta}, s,X_s^{\theta}, l_s^{\theta}).
        \right]
\end{aligned}
\]
Since $\nabla_\theta \log p$ is Lipschitz in the measure argument, we have 
\[
\left\| \Delta_1^{(i)}(\theta) \right\|_\infty \leq  \sum_{s=0}^{t-1} L^\prime_p \mathbb E \left[d_\tv \left(\mu_t^\theta, \mu_t^{\theta, \eps}\right)\right] \leq t L^\prime_p \frac{\eps}{2},
\]
where we have used Lemma \ref{lemma:tv-bound}.
For $\Delta_2^{(i)}$, we define
\begin{equation}
    \label{eq:likelihood-ratio-t}
\begin{aligned}
    R_{t}^{\theta, \eps} &\coloneq \prod_{s = 0}^{t-1} \frac{\tilde P (X^\theta_{s+1} \mid X^\theta_s, \alpha_s^\theta, l_s^\theta + \eps \Lambda_s)\tilde p (\theta, \alpha_s^\theta, s,X_s^\theta,l_s^\theta + \eps \Lambda_s)}{\tilde P (X^\theta_{s+1} \mid X^\theta_s, \alpha_s^\theta, l_s^\theta )\tilde p (\theta, \alpha_s^\theta,s, X_s^\theta,l_s^\theta )} \\
    &= \prod_{s=0}^{t-1} \frac{\psi(\theta, X^\theta_{s+1}, s,X^\theta_s, \alpha_s^\theta, l_s^\theta + \eps \Lambda_s)}{\psi(\theta, X^\theta_{s+1}, s,X^\theta_s, \alpha_s^\theta, l_s^\theta)}.    
\end{aligned}
\end{equation}
By Assumption \ref{assumption:grad-log-policy} and performing a Taylor expansion with Lagrange remainder of $\psi$ in its last argument,
there exist random variables $\xi^\theta_{s,j,k}, s = 0, \ldots, t-1, j,k = 1, \ldots, d$ bounded by $B_2$ such that $\xi^\theta_{s,j,k}$ is $\sigma(X^\theta_{s+1}, X^\theta_s, \alpha_s^\theta, \Lambda_s)$-measurable and
\begin{equation}
    \label{eq:likelihood-ratio-expansion}
\begin{aligned}
R_{t}^{\theta, \eps} &= \prod_{s=0}^{t-1} \left(1 + \eps \frac{\nabla_l \psi(\theta, X^\theta_{s + 1}, s,X_s^\theta, \alpha_s^\theta, l_s^\theta)}{\psi(\theta, X^\theta_{s + 1}, s,X_s^\theta, \alpha_s^\theta, l_s^\theta)} \cdot \Lambda_s 
+ \frac{\eps^2}{2} \frac{\sum_{j,k=1}^d \xi^\theta_{s,j,k} \Lambda_s^{(j)} \Lambda_s^{(k)}}{\psi(\theta, X_{s + 1}^\theta, s,X_s^\theta, \alpha_s^\theta, l_s^\theta)}\right) \\
&= 1 + \sum_{s=0}^{t-1} \eps \nabla_l \log \psi(\theta, X^\theta_{s + 1},s, X_s^\theta, \alpha_s^\theta, l_s^\theta) \cdot \Lambda_s + \zeta_t^{\theta,\eps},
\end{aligned}
\end{equation}
where $\zeta_t^{\theta,\eps}$ is a remainder term given by
\begin{equation}
    \label{eq:remainder-term-likelihood-ratio}
\begin{aligned}
\zeta_t^{\theta, \eps} &= \frac{\eps^2}{2}\sum_{s=0}^{t-1} \frac{\sum_{j,k=1}^d \xi^\theta_{s,j,k} \Lambda_s^{(j)} \Lambda_s^{(k)}}{\psi(\theta, X_{s + 1}^\theta, s,X_s^\theta, \alpha_s^\theta, l_s^\theta)}&\\
& \quad + \sum_{k = 2}^{t-1} \sum_{0 \leq s_1 < s_2 < \ldots < s_k \leq t-1} \prod_{r = 1}^k \Bigg( \eps \nabla_l \log \psi(\theta, X^\theta_{s_r + 1}, s_r,X_{s_r}^\theta, \alpha_{s_r}^\theta, l_{s_r}^\theta) \cdot \Lambda_{s_r} \\
& \quad+ \frac{\eps^2}{2} \frac{\sum_{j,k=1}^d \xi^\theta_{s_r,j,k} \Lambda_{s_r}^{(j)} \Lambda_{s_r}^{(k)}}{\psi(\theta, X_{s_r + 1}^\theta,s_r,  X_{s_r}^\theta, \alpha_{s_r}^\theta, l_{s_r}^\theta)} \Bigg).
\end{aligned}
\end{equation}

In the following lemma, we prove that $\zeta_t^{\theta, \eps}$ is of order $\eps^2$ in expectation.
\begin{lemma}
\label{lemma:zeta-remainder}
     There exist constants $A_{1,\eps},A_{2,\eps} > 0$ defined by \eqref{eq:def-A1} and \eqref{eq:def-A2} respectively, depending only on $ \eps,B_1,B_2, C, d, | \cA |, T$, bounded as $\eps \to 0$, such that, for all $i \in \{1, \ldots, d\}$, $0 \leq s < t \leq T$, 
\[
\left\| \mathbb E \left[ \mathbf 1_{X_t^\theta = x^{(i)}} \zeta_t^{\theta, \eps} \Lambda_s  \right] \right\|_{\infty} \leq A_{1,\eps} \eps^2, \quad \text{and} \quad \left\| \mathbb E \left[ \mathbf 1_{X_t^\theta = x^{(i)}} \zeta_t^{\theta, \eps} \nabla_\theta \log \tilde p (\theta, \alpha_s^{\theta}, s,X_s^\theta, l_s^\theta) \right] \right\|_{\infty} \leq A_{2,\eps} \eps^2.
\]
\end{lemma}

\begin{proof}[Proof of Lemma \ref{lemma:zeta-remainder}]
    We examine the first bound. We have
    \[
        \mathbb E \left[ \mathbf 1_{X_t^\theta = x^{(i)}} \zeta_t^{\theta, \eps} \Lambda_s \right]
        = \frac{\eps^2}{2} \mathbb E \left[\mathbf 1_{X_t^\theta = x^{(i)}} \Lambda_s  \sum_{\tau=0}^{t-1} U_\tau^\theta \right]
        + \sum_{k = 2}^{t-1} \sum_{0 \leq s_1 < s_2 < \ldots < s_k \leq t-1} \mathbb E \left[\mathbf 1_{X_t^\theta = x^{(i)}} \Lambda_s\prod_{r = 1}^k V^\theta_{s_r}\right],
    \]
    where 
    \[ 
    \begin{aligned}
    U^\theta_\tau &= \frac{\sum_{j,k=1}^d \xi^\theta_{\tau,j,k} \Lambda_\tau^{(j)} \Lambda_\tau^{(k)}}{\psi(\theta, X_{\tau + 1}^\theta, \tau, X_\tau^\theta, \alpha_\tau^\theta, l_\tau^\theta)}, \\
    V^\theta_{\tau} &=  \eps \nabla_l \log \psi(\theta, X^\theta_{\tau + 1}, \tau,X_{\tau}^\theta, \alpha_{\tau}^\theta, l_{\tau}^\theta) \cdot \Lambda_{\tau} 
 + \frac{\eps^2}{2}U^\theta_{\tau}.
    \end{aligned}
    \]
    For the first term, we have by independence and zero-mean of $(\Lambda_\tau)_\tau$ that 
    \[
    \left\|\mathbb E \left[\mathbf 1_{X_t^\theta = x^{(i)}} \Lambda_s \sum_{\tau = 0}^{t-1}U_\tau^\theta \right] \right\|_\infty 
    = \left\| \mathbb E \left[\mathbf 1_{X_t^\theta = x^{(i)}} \Lambda_s U_s^\theta \right] \right\|_\infty 
    \leq B_2 d^3 | \cA | \mathbb E[|Z|^3],
\]
    where $Z$ is a standard normal random variable.
    We now examine the second term.
    Fix $k \in \{2, \ldots,t-1\}$ and let $s$ and $0 \leq s_1 < \ldots < s_k < t$ such that $s \in \{s_1, \ldots, s_k\}$. Then, we have
    \[
    \begin{aligned}
    &\mathbb E \left[ \mathbf 1_{X_t^\theta = x^{(i)}} \eps \Lambda_s\prod_{r = 1}^k V^\theta_{s_r}\right] \\
    & \qquad =\mathbb E \left[ \mathbf 1_{X_t^\theta = x^{(i)}} \Lambda_s \left( \eps \nabla_l \log \psi(\theta, X^\theta_{s + 1}, s,X_{s}^\theta, \alpha_{s}^\theta, l_{s}^\theta) \cdot \Lambda_{s} \prod_{r = 1, s_r \neq s}^{k} \frac{\eps^2}{2} U_{s_r}^\theta
    +\prod_{r=1}^{k} \frac{\eps^2}{2} U^\theta_{s_r} \right) \right],
    \end{aligned}
    \]
    since all the other terms cancel by independence and zero-mean of $(\Lambda_\tau)_\tau$.
    We have 
    \[
    \begin{aligned}
    &\left\|\mathbb E \left[ 
        \mathbf 1_{X_t^\theta = x^{(i)}} \Lambda_s \eps \nabla_l \log \psi(\theta, X^\theta_{s + 1}, s,X_{s}^\theta, \alpha_{s}^\theta, l_{s}^\theta) \cdot \Lambda_{s} 
        \prod_{r=1, s_r \neq s}^{k} \frac{\eps^2}{2} U^\theta_{s_r} \right] \right\|_\infty \\
        & \qquad \leq \frac{\eps^{2k - 1}}{2^{k-1}}B_1 B_2^{k-1} d^{2(k-1)} d^{k} | \cA |^k \mathbb E [|Z|^{2k}],
    \end{aligned}
    \] 
    and 
    \[
    \left\|\mathbb E \left[ \mathbf 1_{X_t^\theta = x^{(i)}} \Lambda_s \prod_{r=1}^{k} \frac{\eps^2}{2} U^\theta_{s_r} \right] \right\|_\infty 
        \leq \frac{\eps^{2k}}{2^{k}} B_2^{k} d^{3k} | \cA |^k \mathbb E [|Z|^{2k+1}].
    \]
    Denote by $\sigma_{k} \coloneq \mathbb E |Z|^k$. Summing over $k$ and all possible choices of $s_1, \ldots, s_k$ including $s$, we get the following bound
    \[
    \label{eq:bound-A1}
    \begin{aligned}
    \mathbb E \left[ \mathbf 1_{X_t^\theta = x^{(i)}} \zeta_t^{\theta, \eps} \Lambda_s \right] &\leq \frac{\eps^2}{2} B_2 d^3 | \cA | \sigma_3 \\
    & \quad + \sum_{k=2}^{t-1} \binom{t-1}{k-1} \left( \frac{\eps^{2k - 1}}{2^{k-1}}B_1 B_2^{k-1} d^{2(k-1)} d^{k} | \cA |^k \sigma_{2k} + \frac{\eps^{2k}}{2^{k}} B_2^{k} d^{3k} | \cA |^k \sigma_{2k + 1} \right),
    \end{aligned}
    \]
    which proves the first bound of the lemma for a constant $A_{1,\eps}$ depending only on $\eps, B_1,B_2, d, | \cA |, T$ defined by
    \begin{equation}
        \label{eq:def-A1}
         A_{1,\eps} \coloneq \frac{B_2 d^3 | \cA | \sigma_3}{2} + \sum_{k=2}^{t-1} \binom{t-1}{k-1} \left( \frac{\eps^{2k - 3}}{2^{k-1}}B_1 B_2^{k-1} d^{2(k-1)} d^{k} | \cA |^k \sigma_{2k} + \frac{\eps^{2k-2}}{2^{k}} B_2^{k} d^{3k} | \cA |^k \sigma_{2k + 1} \right).
    \end{equation}
    We now examine the second bound. We have
    \[
    \begin{aligned}
        &\mathbb E \left[ \mathbf 1_{X_t^\theta = x^{(i)}} \zeta_t^{\theta, \eps} \nabla_\theta \log \tilde p (\theta, \alpha_s^{\theta}, s,X_s^\theta, l_s^\theta) \right] \\ 
        &\qquad = \frac{\eps^2}{2} \mathbb E \left[\mathbf 1_{X_t^\theta = x^{(i)}} \nabla_\theta \log \tilde p (\theta, \alpha_s^{\theta}, s,X_s^\theta, l_s^\theta)  \sum_{\tau=0}^{t-1} U_\tau^\theta \right] \\
         &\qquad \quad + \sum_{k=2}^{t-1} \sum_{0 \leq s_1 < \ldots < s_k < t} \mathbb E \left[ \mathbf 1_{X_t^\theta = x^{(i)}} \nabla_\theta \log \tilde p (\theta, \alpha_s^{\theta},s, X_s^\theta, l_s^\theta) \prod_{r=1}^{k} V_{s_r}^\theta \right].
    \end{aligned}
    \]
    For the first term, we have
    \[
    \left\|\mathbb E \left[\mathbf 1_{X_t^\theta = x^{(i)}} \nabla_\theta \log \tilde p (\theta, \alpha_s^{\theta},s, X_s^\theta, l_s^\theta)  \sum_{\tau=0}^{t-1} U_\tau^\theta \right] \right\|_\infty 
    \leq t C B_2 d^3 | \cA | \sigma_2.
    \]
    For the second term, fix $k \in \{2, \ldots, t-1\}$ and let $s$ and $0 \leq s_1 < \ldots < s_k < t$. Then, we have
   \[
   \begin{aligned}
    \mathbb E \left[ \mathbf 1_{X_t^\theta = x^{(i)}} \nabla_\theta \log \tilde p (\theta, \alpha_s^{\theta}, s,X_s^\theta, l_s^\theta) \prod_{r=1}^{k} V_{s_r}^\theta \right]
    = \mathbb E \left[ \mathbf 1_{X_t^\theta = x^{(i)}} \nabla_\theta \log \tilde p (\theta, \alpha_s^{\theta},s, X_s^\theta, l_s^\theta) \prod_{r=1}^{k} \frac{\eps^2}{2} U_{s_r}^\theta \right].
   \end{aligned}
   \] 
   Hence 
   \[
   \begin{aligned}
    \left\| \mathbb E \left[ \mathbf 1_{X_t^\theta = x^{(i)}} \nabla_\theta \log \tilde p (\theta, \alpha_s^{\theta},s, X_s^\theta, l_s^\theta) \prod_{r=1}^{k} V_{s_r}^\theta \right] \right\|_\infty 
    \leq \frac{\eps^{2k}}{2^k} C B_2^{k} d^{3k} | \cA |^k \sigma_{2k}.
   \end{aligned}
   \]
    Summing over $k$ and all possible choices of $s_1, \ldots, s_k$, we get the following bound
    \[
    \left\| \mathbb E \left[
    \mathbf 1_{X_t^\theta = x^{(i)}} \zeta_t^{\theta, \eps} \nabla_\theta \log \tilde p (\theta, \alpha_s^{\theta}, s,X_s^\theta, l_s^\theta)
    \right] \right\|_\infty 
    \leq  \sum_{k=1}^{t-1} \binom{t}{k} \frac{\eps^{2k}}{2^k} C B_2^{k} d^{3k} | \cA |^k \sigma_{2k},
    \]
    which proves the second bound of the lemma for a constant $A_{2,\eps}$ depending only on $\eps$, $B_2$, $C$, $d$, $| \cA |$,$T$ defined by
    \begin{equation}
        \label{eq:def-A2}
        A_{2, \eps} \coloneq \sum_{k=1}^{t-1} \binom{t}{k} \frac{\eps^{2k - 2}}{2^k} C B_2^{k} d^{3k} | \cA |^k \sigma_{2k},
    \end{equation}
    concluding the proof of Lemma \ref{lemma:zeta-remainder}.
\end{proof}
Notice that
\[
\begin{aligned}
\Delta_2^{(i)}(\theta) &= \mathbb E \left[ (1 - R^{\theta, \eps}_t)\mathbf 1_{X_t^\theta = x^{(i)}} \sum_{s=0}^{t-1} \nabla_\theta \log \tilde p (\theta, \alpha_s^{\theta}, s,X_s^\theta, l_s^\theta)  \right] \\
&= \mathbb E \left[ \zeta_t^{\theta, \eps} \mathbf 1_{X_t^\theta = x^{(i)}} \sum_{s=0}^{t-1} \nabla_\theta \log \tilde p (\theta, \alpha_s^{\theta},s, X_s^\theta, l_s^\theta) \right],
\end{aligned}
\]
since the term in $\eps$ in the expression of $R_t^{\theta, \eps}$ \eqref{eq:likelihood-ratio-expansion} vanishes in expectation by independence and zero-mean of $(\Lambda_s)_s$.
Hence $\left\| \Delta_2^{(i)}(\theta) \right\|_\infty \leq t  A_{2, \eps} \eps^2$ by Lemma \ref{lemma:zeta-remainder}.
We now examine $b_{t,2}^{(i)}$. 
Using the likelihood ratio $R_{t}^{\theta, \eps}$, we can rewrite the first expectation in the expression of $b_{t,2}^{(i)}$ \eqref{eq:bt-decomposition} as
\[
    \begin{aligned}
        &\mathbb E \left[ \mathbf 1_{Y_{t}^{\theta, \eps} = x^{(i)}} \sum_{s=0}^{t-1} \eps^{-1}  \Lambda_s \widehat{\nabla_\theta l_s^\theta}  \right] \\
        &\qquad = \mathbb E \left[ \mathbf 1_{X_t^\theta = x^{(i)}}  \sum_{s=0}^{t-1} \eps^{-1} \Lambda_s \widehat{\nabla_\theta l_s^\theta}   R_{t}^{\theta, \eps}\right] \\
        &\qquad = \mathbb E \left[ \mathbf 1_{X_t^\theta = x^{(i)}} \sum_{s=0}^{t-1} \eps^{-1} \Lambda_s \widehat{\nabla_\theta l_s^\theta}   \left(1 + \sum_{s^\prime = 0}^{t-1} \eps \nabla_l \log \psi(\theta, X^\theta_{s^\prime + 1}, s^\prime,X_{s^\prime}^\theta, \alpha_{s^\prime}^\theta, l_{s^\prime}^\theta) \cdot \Lambda_{s^\prime} + \zeta_t^{\theta, \eps} \right) \right]. \\
    \end{aligned}
\]
Now, observe that since $\Lambda_s$ is independent of $X_t^\theta$ and has zero mean, and that $\mathbb E \left[ \Lambda_s^{(i)} \Lambda_{s^\prime}^{(j)} \right] = \delta_{s,s^\prime} \delta_{i,j}$, we have
\[
\begin{aligned}
 \mathbb E \left[ \mathbf 1_{Y_{t}^{\theta, \eps} = x^{(i)}} \sum_{s=0}^{t-1} \eps^{-1} \Lambda_s \widehat{\nabla_\theta l_s^\theta} \right] =& \sum_{s=0}^{t-1} \mathbb E \left[ \mathbf 1_{X_t^\theta = x^{(i)}} \nabla_l \log \psi(\theta, X^\theta_{s + 1}, s,X_s^\theta, \alpha_s^\theta, l_s^\theta) \right] \mathbb E \left[ \widehat{\nabla_\theta l_s^\theta}  \right]  \\
&+ \mathbb E \left[ \mathbf 1_{X_t^\theta = x^{(i)}} \sum_{s=0}^{t-1} \eps^{-1} \zeta_\eps^{\theta, t} \Lambda_s \widehat{\nabla_\theta l_s^\theta}  \right].
\end{aligned}
\]
Hence, using Lemma \ref{lemma:zeta-remainder}, we can upper bound $b_{t,2}^{(i)}$ as
\[
\begin{aligned}
    \left\|b_{t,2}^{(i)}(\theta) \right\|_{\infty} &\leq \left\| \sum_{s = 0}^{t-1} \mathbb E \left[ \mathbf 1_{X_t^\theta = x^{(i)}} \nabla_l \log \psi(\theta, X_{s + 1}^\theta,s, X_s^\theta, \alpha_s^\theta, l_s^\theta) \right] \left(\mathbb E \left[ \widehat{\nabla_\theta l_s^\theta} \right] - \nabla_\theta l_s^\theta \right) \right\|_{\infty} \\
    & \quad + \eps \sum_{s=0}^{t-1} A_{1, \eps} \left\|\mathbb E \left[\widehat{\nabla_\theta l_s^\theta}\right] \right\|  \\
    & \leq \sum_{s = 0}^{t-1} \left\| \mathbb E \left[ \mathbf 1_{X_t^\theta = x^{(i)}} \nabla_l \log \psi(\theta, X_{s + 1}^\theta, s,X_s^\theta, \alpha_s^\theta, l_s^\theta) \right] \right\|_\infty  E_s(\theta, \eps) \\
    &\quad + \eps \sum_{s=0}^{t-1} A_{1, \eps} \left(E_s(\theta, \eps) + \left\| \nabla_\theta l_s^\theta \right\|\right). \\ 
\end{aligned}
\]
Notice that 
\[
    \left\| \mathbb E \left[ \mathbf 1_{X_t^\theta = x^{(i)}} \nabla_l \log \psi(\theta, X_{s + 1}^\theta, s,X_s^\theta, \alpha_s^\theta, l_s^\theta) \right] \right\|_\infty 
    \leq B_1| \cA | d^{\mathbf 1_{s < t - 1}},
\]
by a similar argument as in the proof of Lemma \ref{lemma:bound-gradient-logits}. 
Hence, we have the following recursive bound for $E_t(\theta, \eps)$,
\begin{equation}
    \label{eq:recursive-bound-Et-raw}
\begin{aligned}
 E_t (\theta, \eps) & \leq \sum_{i = 1}^d \frac{1}{\mathbb P (X_t^\theta = x^{(i)})} \left(\left\| b_{t,1}^{(i)} (\theta) \right\|_{\infty} + \left\| b_{t,2}^{(i)} (\theta) \right\|_{\infty} \right) \\
 & \leq \sum_{i=1}^d \frac{1}{\mathbb P (X_t^\theta = x^{(i)})} \Bigg( t L^\prime_p \frac{\eps}{2} + t A_{2, \eps} \eps^2 \\
 &  \quad + \sum_{s = 0}^{t-1} B_1 | \cA | d^{\mathbf 1_{s < t - 1}} E_s (\theta, \eps) + \eps \sum_{s=0}^{t-1} A_{1, \eps} \left( E_s (\theta, \eps) + \left\| \nabla_\theta l_s^\theta \right\| \right)\Bigg).\\
\end{aligned}
\end{equation}
We focus on the last term of this inequality, which we upper bound using Lemma \ref{lemma:bound-gradient-logits},
\[
\sum_{s=0}^{t-1} \left\| \nabla_\theta l_s^\theta \right\| \leq \sum_{s=0}^{t-1} \frac{3 C (\beta^s - 1)}{2(\beta - 1)} \leq \frac{3 C}{2(\beta - 1)} \sum_{s=0}^{t-1}\beta^s,
\]
where $\beta$ is defined in \eqref{eq:definition-beta}.
Hence, the recursive upper bound on $E_t(\theta, \eps)$ \eqref{eq:recursive-bound-Et-raw} can be written such that Lemma \ref{lemma:discrete-gronwall} can be applied with the following constants
\[
\begin{aligned}
    \tilde c_1 &=  \mathfrak s^\theta \left( A_{1, \eps} \eps + B_1 | \cA | \right), \quad 
    \tilde c_2 =  \mathfrak s^\theta B_1 | \cA | (d - 1), \\
    \tilde c_3 &= \mathfrak s^\theta \left(L^\prime_p \frac{\eps}{2} + A_{2, \eps} \eps^2 \right), \quad \tilde c_4 = \mathfrak s^\theta \left( A_{1, \eps} \eps \frac{3C}{2 (\beta - 1)} \right), \quad 
    \tilde c_5 = \beta, \quad, \tilde c_6 = 0, \\
    \text{where }& \mathfrak s^\theta \coloneq \sup_{t \geq 0} \sum_{i=1}^d \mathbb P (X_t^\theta = x^{(i)})^{-1} < \infty.
\end{aligned}
\]
Now, with these notations let $\beta_\eps \coloneq \frac{1}{2} ( 1 + \tilde c_1 + \sqrt{(1 + \tilde c_1)^2 + 4 \tilde c_2})$.
Since $ A_{1, \eps} \eps > 0$, we have $\beta_\eps > \beta = \tilde c_5$,
therefore
\[
E_t(\theta, \eps) \leq \frac{3 \tilde c_3 (\beta_\eps ^t - 1)}{2 (\beta_\eps - 1)} + \frac{3 \tilde c_4 (\tilde c_5 \beta_\eps^t - \tilde c_5^{t+1})}{2 (\beta_\eps - \tilde c_5)}.
\]
We can rewrite this upper bound in the following way to highlight the dependence on $\eps$ and $t$
\[
E_t (\theta, \eps) \leq \mathfrak s^\theta \left[ \frac{3 L_p^\prime}{4 (\beta - 1)}\eps + \frac{3 A_{2, \eps}}{2(\beta_\eps - 1)}\eps^2 \right] \left(\beta_\eps^t - 1\right) + \mathfrak s^\theta \frac{9 C A_{1, \eps} \beta}{4(\beta - 1)(\beta_\eps - \beta)}\eps \left(\beta_\eps^t - \beta^t \right).
\]
This gives the desired inequality \eqref{eq:bias-logit-gradient-estimator} with the following constants:
\begin{equation}
    \label{eq:constants-bias-logits}
    \begin{aligned}
    K_{1, \eps} &= \mathfrak s^\theta \frac{9 C A_{1, \eps} \beta}{4(\beta - 1)(\beta_\eps - \beta)}, \quad
    K_{2, \eps} = \mathfrak s^\theta \left[ \frac{3 L_p^\prime}{4 (\beta - 1)} + \frac{3 A_{2, \eps}}{2(\beta_\eps - 1)}\eps \right], \\
    \beta_\eps &= \frac{1}{2} \left(
            1 + \mathfrak s^\theta \left(  A_{1, \eps} \eps + B_1 |\mathcal A| \right) + \sqrt{
            \left[ 1 + \mathfrak s^\theta \left( A_{1, \eps} \eps + B_1 |\mathcal A| \right) \right]^2
            + 4 \mathfrak s^\theta B_1 |\mathcal A| (d - 1)}\right). \\
    \end{aligned}
\end{equation}
Now, if we assume that $ B_1 = 0$, then we simply have  
$b_t^{(i)} = \Delta_1^{(i)}$, hence 
\[
E_t(\theta, \eps) \leq \mathfrak s^\theta \frac{L^\prime_p}{2} t \eps, 
\]
which concludes the proof of Proposition \ref{proposition:bias-state-distribution-gradient}.
\end{proof}

\begin{proof}[Proof of Theorem \ref{theorem:main-bias-result}]
Assume that $B_1 > 0$. We decompose the bias as follows:
\[
\begin{aligned}
    &\left\| \mathbb E \left[\widehat{\nabla V}(\theta) \right] - \nabla_\theta V (\mu, \theta)\right\|_\infty \leq \Delta_{\mathrm{RF}}(\theta) + \Delta_{\mathrm M}(\theta), \\
    &\Delta_{\mathrm{RF}}(\theta) \coloneq \left\|\mathbb E \left[\sum_{t=0}^{T-1} \nabla_\theta \log \tilde p(\theta, \alpha_t^{\theta, \eps},t, Y_t^\theta, l_t^\theta + \eps \Lambda_t)G_t^{\theta, \eps} \right]- \mathrm{RF}(\theta) \right\|_\infty, \\
    &\Delta_{\mathrm M}(\theta) \coloneq \left\| \mathbb E \left[\sum_{t=0}^{T} \eps^{-1} \Lambda_t \widehat{\nabla_\theta l^\theta_t} G_t^{\theta, \eps} \right]  - (\mathrm{MD}(\theta) + \mathrm{MFD}(\theta)) \right\|_\infty.
\end{aligned}
\]
We first examine $\Delta_{\mathrm{RF}}(\theta)$.
We proceed in a similar way to the study of $b_{t,1}^{(i)}$ in the proof of Theorem \ref{proposition:bias-state-distribution-gradient}. 
We define $H_t^{\theta, \eps} \coloneq \sum_{s=t}^{T-1} \tilde r(Y_s^{\theta, \eps}, \alpha_s^{\theta, \eps}, l_s^{\theta}) + \tilde g(Y_T^{\theta, \eps}, l_T^{\theta})$.
Then
\[
\begin{aligned}
    \Delta_{\mathrm{RF}}(\theta) &\leq \Delta_{\mathrm{RF}}^{(1)}(\theta) + \Delta_{\mathrm{RF}}^{(2)}(\theta), \\
    \Delta_{\mathrm{RF}}^{(1)}(\theta) &\coloneq \left\| \mathbb E \left[\sum_{t=0}^{T-1} \left(\nabla_\theta \log \tilde p(\theta, \alpha_t^{\theta, \eps}, t,Y_t^\theta, l_t^\theta + \eps \Lambda_t)G_t^{\theta, \eps} - \nabla_\theta \log \tilde p (\theta, \alpha_t^{\theta, \eps}, t,Y_t^{\theta, \eps}, l_t^\theta)H_t^{\theta, \eps} \right)\right] \right\|_\infty, \\
    \Delta_{\mathrm{RF}}^{(2)}(\theta) &\coloneq \left\| \mathbb E \left[\sum_{t=0}^{T-1} \nabla_\theta \log \tilde p (\theta, \alpha_t^{\theta, \eps},t, Y_t^{\theta, \eps}, l_t^\theta)H_t^{\theta, \eps} \right] - \mathrm{RF}(\theta) \right\|_\infty.
\end{aligned}
\]
Since the reward functions and the log-policy are assumed to be Lipschitz in the measure argument, we have
\[
\begin{aligned}
 \Delta_{\mathrm{RF}}^{(1)}(\theta) &\leq \sum_{t=0}^{T-1} \left(L_{p}^\prime (T - t + 1) M_0 + \left((T - t)L_r + L_g \right)C\right) \mathbb E \left[ d_\tv(\mu_t^\theta, \mu_t^{\theta, \eps}) \right] \\
 & \leq \sum_{t=0}^{T-1} \left(L_{p}^\prime (T - t + 1) M_0 + \left((T - t)L_r + L_g \right)C\right) \frac{\eps}{2} \quad \quad \text{(by Lemma \ref{lemma:tv-bound})}\\
 &=  A_3\eps,
\end{aligned}
\]
where $A_3 = \left(M_0 L^\prime_p  \frac{T^2 + 3T}{2} + C L_r \frac{T^2 + T}{2} + C L_g T \right)/2$.
For $\Delta_{\mathrm{RF}}^{(2)}(\theta)$, we use yet again the likelihood ratio trick. Recall $R_T^{\theta, \eps}$ defined in \eqref{eq:likelihood-ratio-t} and its expansion in \eqref{eq:likelihood-ratio-expansion}.
We have
\[
\begin{aligned}
\Delta_{\mathrm{RF}}^{(2)}(\theta) &= \left\| \mathbb E \left[ (1 - R_T^{\theta, \eps})\sum_{t=0}^{T-1} \nabla_\theta \log \tilde p (\theta, \alpha_t^{\theta}, t,X_t^{\theta}, l_t^\theta) G_t^{\theta}  \right] \right\|_\infty \\
&= \left\| \mathbb E \left[\zeta_T^{\theta, \eps}\sum_{t=0}^{T-1} \nabla_\theta \log \tilde p (\theta, \alpha_t^{\theta}, t,X_t^{\theta}, l_t^\theta) G_t^{\theta} \right] \right\|_\infty, \\ 
\end{aligned}
\]
where $G_t^{\theta} = \sum_{s=t}^{T-1} \tilde r(X_s^{\theta}, \alpha_s^{\theta}, l_s^{\theta}) + \tilde g(X_T^{\theta}, l_T^{\theta})$. By similar arguments as in the proof of Lemma \ref{lemma:zeta-remainder},
we can get hold of a constant $A_{4,\eps} > 0$ depending only on $\eps, B_2, C, M_0, d, | \cA |, T$, bounded as $\eps \to 0$, such that
\[
 \Delta_{\mathrm{RF}}^{(2)}(\theta) = \left\| \mathbb E \left[\zeta_T^{\theta, \eps}\sum_{t=0}^{T-1} \nabla_\theta \log \tilde p (\theta, \alpha_t^{\theta},t, X_t^{\theta}, l_t^\theta) G_t^{\theta} \right] \right\|_\infty \leq A_{4, \eps} \eps^2.
\]
We now examine $\Delta_{\mathrm M}(\theta)$, which we decompose as
\[
\begin{aligned}
 \Delta_{\mathrm{M}}(\theta) &\leq \Delta_{\mathrm{MD}}(\theta) + \Delta_{\mathrm{MFD}}(\theta), \\
    \Delta_{\mathrm{MD}}(\theta) &\coloneq \left\| \mathbb E \left[\sum_{t=0}^{T} \eps^{-1} \Lambda_t \widehat{\nabla_\theta l^\theta_t} (G_t^{\theta, \eps} - H_t^{\theta, \eps}) \right] - \mathrm{MD}(\theta) \right\|_\infty, \\
    \Delta_{\mathrm{MFD}}(\theta) &\coloneq \left\| \mathbb E \left[\sum_{t=0}^{T} \eps^{-1} \Lambda_t \widehat{\nabla_\theta l^\theta_t} H_t^{\theta, \eps} \right] - \mathrm{MFD}(\theta) \right\|_\infty.
\end{aligned}
\]
Using the likelihood ratio $R_T^{\theta, \eps}$ we have
\[
\mathbb E \left[\sum_{t=0}^{T} \eps^{-1} \Lambda_t \widehat{\nabla_\theta l^\theta_t} (G_t^{\theta, \eps} - H_t^{\theta, \eps}) \right] = \mathbb E \left[\sum_{t=0}^{T} \eps^{-1} \Lambda_t \widehat{\nabla_\theta l^\theta_t} (\tilde G_t^{\theta, \eps} - G_t^{\theta}) R_T^{\theta, \eps} \right],
\]
where $\tilde G_t^{\theta, \eps} = \sum_{s=t}^{T-1} \tilde r(X_s^{\theta}, \alpha_s^{\theta}, l_s^{\theta} + \eps \Lambda_s) + \tilde g(X_T^{\theta}, l_T^{\theta} + \eps \Lambda_T)$.
Since $\tilde r, \tilde g$ are assumed to be twice differentiable by Assumption \ref{assumption:lipschitz-reward}, we can use a Taylor expansion to get hold of random variables 
$\chi^\theta_{s,j,k}$ for $s=0, \ldots, T$ and $j,k = 1, \ldots, d$ bounded by $M_2$ such that $\chi^\theta_{t,j,k}$ is $\sigma(X^\theta_t, \alpha_t^\theta, \Lambda_t)$-measurable when $t \leq T-1$ 
and $\sigma (X_T^\theta, \Lambda_T)$-measurable when $t = T$, and
\[
\tilde G_t^{\theta, \eps} = G_t^\theta + \eps \sum_{s = t}^{T-1} \nabla_l \tilde r(X_s^\theta, \alpha_s^\theta, l_s^\theta) \cdot \Lambda_s + \eps \nabla_l \tilde g(X_T^\theta, \alpha_T^\theta) \cdot \Lambda_T
 + \frac{\eps^2}{2} \sum_{s=t}^{T}\sum_{j,k=1}^d \chi^\theta_{s,j,k} \Lambda_s^{(j)} \Lambda_s^{(k)}.
\]
Hence, 
\begin{equation}
    \label{eq:decomposition-md-approx}
\begin{aligned}
    &\mathbb E \left[\sum_{t=0}^{T} \eps^{-1} \Lambda_t \widehat{\nabla_\theta l^\theta_t} (G_t^{\theta, \eps} - H_t^{\theta}) \right] \\
    &\qquad = \mathbb E \left[ \left\{\sum_{t=0}^{T} \Lambda_t \widehat{\nabla_\theta l^\theta_t} \left(\sum_{s=t}^{T-1}\nabla_l \tilde r(X_s^\theta, \alpha_s^\theta, l_s^\theta) \cdot \Lambda_s 
    + \nabla_l \tilde g(X_T^\theta, l_T^\theta) \cdot \Lambda_T\right)  \right\} R_T^{\theta, \eps}\right] \\
    & \qquad \quad + \frac{\eps}{2}\mathbb E \left[ \left\{\sum_{t=0}^{T} \Lambda_t \widehat{\nabla_\theta l^\theta_t} \sum_{s=t}^{T}\sum_{j,k=1}^d \chi^\theta_{s,j,k} \Lambda_s^{(j)} \Lambda_s^{(k)} \right\} R_T^{\theta, \eps}\right].
\end{aligned}
\end{equation}
The second term in the right-hand side is of order $\eps$ since all the random variables inside the expectation are bounded.
 Using similar arguments as in the proof of Lemma \ref{lemma:zeta-remainder}, one can get hold of a constant $A_{5,\eps}$ depending only on $\eps, M_2, B_1,B_2, C, d, | \cA |, T$, bounded as $\eps \to 0$ such that 
\[
    \left\| \frac{\eps}{2}\mathbb E \left[ \left\{\sum_{t=0}^{T} \Lambda_t \widehat{\nabla_\theta l^\theta_t} \sum_{s=t}^{T}\sum_{j,k=1}^d \chi^\theta_{s,j,k} \Lambda_s^{(j)} \Lambda_s^{(k)} \right\} R_T^{\theta, \eps}\right] \right\|_\infty 
    \leq A_{5,\eps} \eps.
\]
For the first term, we plug in the expansion of $R_T^{\theta, \eps}$ from \eqref{eq:likelihood-ratio-expansion} to get
\[
\begin{aligned}
&\mathbb E \left[ \left\{\sum_{t=0}^{T} \Lambda_t \widehat{\nabla_\theta l^\theta_t} \left(\sum_{s=t}^{T-1}\nabla_l \tilde r(X_s^\theta, \alpha_s^\theta, l_s^\theta) \cdot \Lambda_s 
    + \nabla_l \tilde g(X_T^\theta, l_T^\theta) \cdot \Lambda_T\right)  \right\} R_T^{\theta, \eps}\right] \\
    &\qquad = \mathbb E \left[ \sum_{t=0}^{T} \Lambda_t \widehat{\nabla_\theta l^\theta_t} \left(\sum_{s=t}^{T-1}\nabla_l \tilde r(X_s^\theta, \alpha_s^\theta, l_s^\theta) \cdot \Lambda_s 
    + \nabla_l \tilde g(X_T^\theta, l_T^\theta) \cdot \Lambda_T\right)   \right] + \mathcal O (\eps).
\end{aligned}
\]
Using the independence of $(\Lambda_t)_{t=0}^T$ and $(\widehat{\nabla_\theta l_t^\theta}, X_t^\theta, \alpha_t^\theta)_{t=0}^T$, as well as the fact that $\mathbb E[\Lambda_t^{(i)} \Lambda_s^{(j)}] = \delta_{t,s} \delta_{i,j}$, we have
\[
\begin{aligned}
 &\mathbb E \left[ \sum_{t=0}^{T} \Lambda_t \widehat{\nabla_\theta l^\theta_t} \left(\sum_{s=t}^{T-1}\nabla_l \tilde r(X_s^\theta, \alpha_s^\theta, l_s^\theta) \cdot \Lambda_s 
    + \nabla_l \tilde g(X_T^\theta, l_T^\theta) \cdot \Lambda_T\right)\right] \\
&\qquad = \sum_{t=0}^{T-1} \mathbb E \left[ \nabla_l \tilde r(X_t^\theta, \alpha_t^\theta, l_t^\theta) \right] \mathbb E \left[ \widehat{\nabla_\theta l^\theta_t} \right] 
+ \mathbb E \left[ \nabla_l \tilde g(X_T^\theta, l_T^\theta) \right] \mathbb E \left[ \widehat{\nabla_\theta l^\theta_T} \right].
\end{aligned}
\]
Once again, one can get hold of a constant $A_{6,\eps}$ depending only on $\eps, M_1, B_1, B_2, C, d, | \cA |, T$, bounded as $\eps \to 0$ such that
\[
\begin{aligned}
    \Delta_{\mathrm{MD}}(\theta) &\leq \sum_{t=0}^{T-1}  \left\|\mathbb E \left[ \nabla_l \tilde r(X_t^\theta, \alpha_t^\theta, l_t^\theta) \right] \left(\mathbb E \left[ \widehat{\nabla_\theta l^\theta_t} \right] - \nabla_\theta l_t^\theta\right)\right\|_\infty \\
    & \quad +\left\|\mathbb E \left[ \nabla_l \tilde g(X_T^\theta, l_T^\theta) \right] \left(\mathbb E \left[ \widehat{\nabla_\theta l^\theta_T} \right] - \nabla_\theta l_T^\theta\right)\right\|_\infty + (A_{5,\eps} + A_{6,\eps}) \eps.  
\end{aligned}
\]
Now, we examine $\Delta_{\mathrm{MFD}}(\theta)$.
Using the likelihood ratio $R_T^{\theta, \eps}$ again, we have
\[
\begin{aligned}
&\mathbb E \left[ \sum_{t=0}^T \eps^{-1} \Lambda_t \widehat{\nabla_\theta l_t^\theta} H_t^{\theta, \eps} \right] = \mathbb E \left[ \sum_{t=0}^T \eps^{-1} \Lambda_t \widehat{\nabla_\theta l_t^\theta} G_t^\theta R_t^{\theta, \eps} \right] \\
&\qquad = \mathbb E \left[ \sum_{t=0}^T \eps^{-1}\Lambda_t \widehat{\nabla_\theta l_t^\theta} G_t^\theta \right] + \mathbb E \left[ \sum_{t=0}^T \Lambda_t \widehat{\nabla_\theta l_t^\theta} G_t^\theta \sum_{s=0}^{T-1} \nabla_l \log \psi(\theta, X_{s+1}^\theta, s,X_s^\theta, \alpha_s^\theta, l_s^\theta) \cdot \Lambda_s \right] \\
& \qquad \quad + \mathbb E \left[ \sum_{t=0}^T \eps^{-1} \Lambda_t \widehat{\nabla_\theta l_t^\theta} G_t^\theta \zeta_T^{\theta, \eps} \right] \\
&\qquad = \sum_{t=0}^{T - 1}  \mathbb E \left[ \nabla_l \log \psi(\theta, X_{t+1}^\theta, t,X_t^\theta, \alpha_t^\theta, l_t^\theta) G_t^\theta \right] \mathbb E \left[\widehat{\nabla_\theta l_t^\theta} \right] + \mathcal O (\eps).  
\end{aligned}
\]
Once again, one can get hold of a constant $A_{7, \eps}$ depending only on $M_0, B_1, B_2, C, d, | \cA |, T$, bounded as $\eps \to 0$, such that
\[
\Delta_{\mathrm{MFD}}(\theta) \leq \sum_{t=0}^{T - 1}  \left\|\mathbb E \left[ \nabla_l \log \psi(\theta, X_{t+1}^\theta,t, X_t^\theta, \alpha_t^\theta, l_t^\theta) G_t^\theta \right] \left(\mathbb E \left[\widehat{\nabla_\theta l_t^\theta} \right] - \nabla_\theta l_t^\theta\right) \right\|_\infty + A_{7,\eps} \eps.
\]
Hence, the total bias can be upper bounded by
\[
\begin{aligned}
& \left\| \mathbb E \left[\widehat{\nabla V}(\theta) \right] - \nabla_\theta V (\mu, \theta)\right\|_\infty \\
& \qquad \leq \sum_{t=0}^{T-1} \left(\left\| \mathbb E \left[ \nabla_l \log \psi(\theta, X_{t+1}^\theta, t,X_t^\theta, \alpha_t^\theta, l_t^\theta) G_t^\theta \right] \right\|_\infty + \left\|\mathbb E \left[\nabla_l \tilde r(X_t^\theta, \alpha_t^\theta, l_t^\theta) \right] \right\|_\infty \right) E_t(\theta, \eps) \\
& \qquad \quad + \left\|\mathbb E \left[\nabla_l \tilde g(X_T^\theta, l_T^\theta) \right] \right\|_\infty E_T(\theta, \eps) + (A_3 + A_{5, \eps} + A_{6, \eps} + A_{7, \eps}) \eps + A_{4, \eps} \eps^2 \\
& \qquad \leq \sum_{t=0}^{T-1} \left(B_1 d | \cA | (T - t + 1) M_0 + M_1 \right)E_t(\theta, \eps) + M_1 E_T(\theta, \eps) + \tilde A_{1, \eps} \eps + \tilde A_{2, \eps} \eps^2,
\end{aligned}
\]
where $\tilde A_{1, \eps} = A_3 + A_{5, \eps} + A_{6, \eps} + A_{7, \eps}$ and $\tilde A_{2, \eps} = A_{4, \eps}$.
Using Proposition \ref{proposition:bias-state-distribution-gradient}, we can rewrite the upper bound in the following way

\[
\begin{aligned}
    \left\| \mathbb E \left[\widehat{\nabla V}(\theta) \right] - \nabla_\theta V (\mu, \theta)\right\|_\infty
    &\leq \eps \kappa_1 \sum_{t=0}^{T-1} (T - t + 1)\left( \left[ K_{1, \eps} + K_{2} \right]\beta_\eps^t - K_{1, \eps} \beta^t - K_{2} \right) \\
    & \quad + \eps \kappa_2\sum_{t=0}^{T} \left( \left[ K_{1, \eps} + K_{2} \right]\beta_\eps^t - K_{1, \eps} \beta^t - K_{2} \right) + \tilde A_{1, \eps} \eps + \tilde A_{2, \eps} \eps^2 \\
    &\leq \eps \kappa_1 \left[(K_{1, \eps} + K_{2}) \beta_\eps^T \frac{2 \beta_\eps - 1}{(\beta_\eps - 1)^2}
    - 2  K_{1, \eps} \beta^{T-1} - K_{2} \frac{T^2 + 3T}{2}\right] \\
    & \quad + \eps \kappa_2 \left[\left(K_{1, \eps} + K_{2}\right) \frac{\beta_\eps^{T + 1} - 1}{ \beta_\eps - 1} - K_{1, \eps} \beta^{T} - (T + 1) K_{2}\right] \\
    &\quad + \tilde A_{1, \eps} \eps + \tilde A_{2, \eps} \eps^2 \\
    & \leq \eps \beta_\eps^T \left[\kappa_1 (K_{1, \eps} + K_{2}) \frac{2 \beta_\eps - 1}{(\beta_\eps - 1)^2} + \kappa_2 \beta_\eps\frac{K_{1, \eps} + K_{2}}{\beta_\eps - 1} \right] \\
    & \quad - \eps \beta^T \left(2 \kappa_1 K_{1, \eps} + \kappa_2 K_{1, \eps} \right) - \eps T^2 \frac{\kappa_1 K_{2}}{2} - \eps T \frac{3 \kappa_1 + 2 \kappa_2}{2} K_{2} \\
    & \quad  - \eps \kappa_2 K_{2} + \tilde A_{1, \eps} \eps + \tilde A_{2, \eps} \eps^2,
\end{aligned}
\]
where $K_{1, \eps}, K_{2, \eps}$ are defined in \eqref{eq:constants-bias-logits} and 
$
    \kappa_1 = B_1 d | \cA | M_0, \kappa_2 = M_1.
$
This gives the upper bound \eqref{eq:theorem-bound-bias-pg-estimator} with the following constants
\begin{equation}
    \label{eq:constant-bias-pg-expression}
\begin{aligned}
    k_{1, \eps} & \coloneq \kappa_1 (K_{1, \eps} + K_{2, \eps}) \frac{2 \beta_\eps - 1}{(\beta_\eps - 1)^2} + \kappa_2 \beta_\eps\frac{K_{1, \eps} + K_{2, \eps}}{\beta_\eps - 1}, \\
    k_{2, \eps} & \coloneq 2 \kappa_1 K_{1, \eps} + \kappa_2 K_{1, \eps}, \\
    k_{3, \eps} & \coloneq - \frac{\kappa_1 }{2} K_{2,\eps} T^2 - \frac{(3 \kappa_1 + 2 \kappa_2)}{2} K_{2, \eps} T - \kappa_2 K_{2, \eps} + \tilde A_{1, \eps}, \\
    k_{4, \eps} & \coloneq \tilde A_{2, \eps},
\end{aligned}
\end{equation}
which concludes the proof.
\end{proof}

\subsubsection{Proofs on Estimation Errors}

\begin{proof}[Proof of Proposition \ref{proposition:l2-logit-gradient-estimator}]
As all the terms in the definition of $\widehat{\nabla _\theta l_t^\theta}$ \eqref{eq:logit-gradient-estimator} involve the same version of $\widehat{\nabla_\theta l_s^\theta}$ for $s < t$, we proceed by using the law of total variance. For any square-integrable random variable $Y$, we consider
\[
\var ( Y ) = \mathbb E \left[ \var_{t}(Y) \right] + \var \left( \mathbb E_{t} \left[ Y \right] \right),
\]
where $\var_{t}, \mathbb E_t$ are the conditional variance and conditional expectation respectively given $\widehat{\nabla_\theta l_s^\theta}$ for $0 \leq s < t$.
Fix $t \in \{0, \ldots, T\}$, $i \in \{1, \ldots, d\}$ and $j \in \{1, \ldots, D\}$.
We have 
\[
\begin{aligned}
    \var_t \left(\left(\widehat{\nabla_\theta l_t^\theta}\right)_{i,j}\right) &= 
    \frac{1}{n \mathbb P \left( X^\theta_t = x^{(i)} \right)^2} \var_t \left(  \sum_{s=0}^{t-1}  Q_{s,i,j}^{\theta, \eps, t} + S_{s,i,j}^{\theta, \eps, t} \right) \\
    & \leq \frac{1}{n \mathbb P \left( X^\theta_t = x^{(i)} \right)^2} \left( \sum_{s=0}^{t-1} \left( \var_t \left( Q_{s,i,j}^{\theta, \eps, t} + S_{s,i,j}^{\theta, \eps, t} \right) \right)^{1/2} \right)^2 \\
    & \leq \frac{1}{n \mathbb P \left( X^\theta_t = x^{(i)} \right)^2} \left( \sum_{s=0}^{t-1} \left(2 \var_t \left( Q_{s,i,j}^{\theta, \eps, t} \right) + 2 \var_t \left( S_{s,i,j}^{\theta, \eps, t} \right) \right)^{1/2} \right)^2 \\
    & \leq \frac{1}{n \mathbb P \left( X^\theta_t = x^{(i)} \right)^2} \left( \sum_{s=0}^{t-1}  \var_t \left( Q_{s,i,j}^{\theta, \eps, t} \right)^{1/2} + \var_t \left( S_{s,i,j}^{\theta, \eps, t} \right)^{1/2}  \right)^2,
\end{aligned}
\]
where
\[
    Q_{s,i,j}^{\theta, \eps, t} \coloneq \eps^{-1} \mathbf 1_{Y_{t}^{\theta, \eps, t} = x^{(i)}} \left(\Lambda_s^{t} \widehat{\nabla_\theta l^\theta_s}\right)_j,
     \quad S_{s,i,j}^{\theta, \eps, t} \coloneq \mathbf 1_{Y_{t}^{\theta, \eps, t} = x^{(i)}} \nabla_\theta \log \tilde p (\theta, \alpha_s^{\theta, \eps,t}, Y_s^{\theta, \eps,t}, l_s^\theta + \eps \Lambda_s^{t})_j,
\]
and we have used $\var_t (\sum_\iota X_\iota) \leq (\sum_\iota \var_t (X_\iota)^{1/2})^2$, and $\var_t (X_1 + X_2) \leq 2 \var_t (X_1) + 2 \var_t (X_2)$ for any square-integrable random variables $X_\iota$.
We have $\var_t ( S_{s,i}^{\theta, \eps, t} ) =  \var ( S_{s,i}^{\theta, \eps, t} ) \leq C^2$
since $\nabla_ \theta \log p$ is bounded by Assumption \ref{assumption:finite-state-space}.
For $\var_t ( Q_{s,i,j}^{\theta, \eps, t} )$, we have 
\[
\begin{aligned}
    \var_t \left( Q_{s,i,j}^{\theta, \eps, t} \right) &= \eps^{-2} \var_t \left( \sum_{i^\prime = 1}^d 
    \mathbf 1_{Y_t^{\theta, \eps, t} = x^{(i)}}\left(\Lambda_s^t\right)_{i^\prime} \left(\widehat{\nabla_\theta l_s^\theta}\right)_{i^\prime,j} \right) \\
    & \leq \eps^{-2}\left( \sum_{i^\prime=1}^d \var_t \left( \mathbf 1_{Y_t^{\theta, \eps, t} = x^{(i)}}\left(\Lambda_s^t\right)_{i^\prime} \left(\widehat{\nabla_\theta l_s^\theta}\right)_{i^\prime,j} \right)^{1/2} \right)^2.
\end{aligned}
\]
We use the independence of $\widehat{\nabla_\theta l_s^\theta}$ and $\mathbf 1_{Y_{t}^{\theta, \eps, t} = x^{(i)}} \Lambda_s^t$ to write, for all $i^\prime = 1, \ldots, d$,
\[
\begin{aligned}
    \var_t \left( \mathbf 1_{Y_t^{\theta, \eps, t} = x^{(i)}}\left(\Lambda_s^t\right)_{i^\prime} \left(\widehat{\nabla_\theta l_s^\theta}\right)_{i^\prime,j} \right) = \var \left(\mathbf 1_{Y_t^{\theta, \eps, t} = x^{(i)}}\left(\Lambda_s^t\right)_{i^\prime}\right) \left(\widehat{\nabla_\theta l_s^\theta}\right)^2_{i^\prime,j},
\end{aligned}
\]
hence
\[
\begin{aligned}
     \var_t \left( Q_{s,i,j}^{\theta, \eps, t} \right) \leq \eps^{-2} \left( \sum_{i^\prime = 1}^d \left|\widehat{\nabla_\theta l_s^\theta}\right|_{i^\prime,j}  \right)^2,
\end{aligned}
\]
where we have used that $\var ( \mathbf 1_{Y_t^{\theta, \eps, t} = x^{(i)}}\left(\Lambda_s^t\right)_{i^\prime} ) \leq \mathbb E [(\Lambda_s^t)_{i^\prime}^2 ] = 1$.
This implies
\begin{equation}
\label{eq:exp-of-cond-var-bound}
\begin{aligned}
    \mathbb E \left[ \var_t \left(\left(\widehat{\nabla_\theta l_t^\theta}\right)_{i,j}\right) \right]
    & \leq \frac{1}{\eps^{2}n \mathbb P \left(X_t^\theta = x^{(i)} \right)^2} \mathbb E \left[ \left(\sum_{s = 0}^{t-1} \sum_{i^\prime = 1}^d  \left|\widehat{\nabla_\theta l_s^\theta}\right|_{i^\prime,j}  + \eps C\right)^2 \right].
\end{aligned}
\end{equation}
Next, we have
\[
\mathbb E_t  \left[ \left(\widehat{\nabla_\theta l_t^\theta}\right)_{i,j} \right] 
= \frac{1}{\mathbb P \left(X_t^\theta = x^{(i)} \right)}\sum_{s=0}^{t-1}  \left(\sum_{i^\prime = 1}^d \E \left[ \eps^{-1} \mathbf 1_{Y_t^{\theta, \eps, k} = x^{(i)}} \left(\Lambda_s^t\right)_{i^\prime}\right] \left( \widehat{\nabla_\theta l_s^\theta} \right)_{i^\prime, j} + \E \left[ S_{s,i,j}^{\theta, \eps, t} \right] \right),
\]
hence
\[
\begin{aligned}
    &\var \left( \mathbb E_t  \left[ \left(\widehat{\nabla_\theta l_t^\theta}\right)_{i,j} \right] \right)
    = \frac{1}{\mathbb P \left(X_t^\theta = x^{(i)} \right)^2} \var \left( 
    \sum_{s=0}^{t-1}  \sum_{i^\prime = 1}^d \E \left[ \eps^{-1} \mathbf 1_{Y_t^{\theta, \eps, k} = x^{(i)}} \left(\Lambda_s^t\right)_{i^\prime}\right] \left( \widehat{\nabla_\theta l_s^\theta} \right)_{i^\prime, j} 
    \right) \\
    & \qquad \leq  \frac{1}{\eps^2\mathbb P \left(X_t^\theta = x^{(i)} \right)^2}\left(\sum_{s = 0}^{t-1} \var \left( \sum_{i^\prime = 1}^d \E \left[ \mathbf 1_{Y_t^{\theta, \eps, k} = x^{(i)}} \left(\Lambda_s^t\right)_{i^\prime}\right] \left( \widehat{\nabla_\theta l_s^\theta} \right)_{i^\prime, j}  \right)^{1/2}\right)^2.
\end{aligned}
\]
By writing $v^{\theta, \eps}_{s, i^\prime,j} = (\Sigma_s^{\theta, \eps})_{i^\prime, j}=\var ( (\widehat{\nabla_\theta l_s^\theta})_{i^\prime,j})$, we have
\[
\begin{aligned}
\var \left( \sum_{i^\prime = 1}^d \E \left[ \mathbf 1_{Y_t^{\theta, \eps, k} = x^{(i)}} \left(\Lambda_s^t\right)_{i^\prime}\right] \left( \widehat{\nabla_\theta l_s^\theta} \right)_{i^\prime, j}  \right) 
&\leq \left( \sum_{i^\prime = 1}^d \E \left[ \mathbf 1_{Y_t^{\theta, \eps, k} = x^{(i)}} \left(\Lambda_s^t\right)_{i^\prime}\right] \left(v_{s,i^\prime,j}^{\theta, \eps}\right)^{1/2} \right)^2 \\
&\leq \left( \sum_{i^\prime = 1}^d  \left(v_{s,i^\prime,j}^{\theta, \eps}\right)^{1/2} \right)^2,
\end{aligned}
\]
where we have used that $\mathbb E [ \mathbf 1_{Y_t^{\theta, \eps, t} = x^{(i)}}\left(\Lambda_s^t\right)_{i^\prime} ] \leq 1$.
Hence
\begin{equation}
\label{eq:var-of-cond-expectation-bound}
\begin{aligned}
    \var \left( \mathbb E_t  \left[ \left(\widehat{\nabla_\theta l_t^\theta}\right)_{i,j} \right] \right) 
    \leq \frac{1}{\eps^2\mathbb P \left(X_t^\theta = x^{(i)} \right)}\left( \sum_{s = 0}^{t-1}\sum_{i^\prime = 1}^d \sqrt{v^{\theta, \eps}_{s, i^\prime, j}} \right)^2.
\end{aligned}
\end{equation}
Combining the bounds \eqref{eq:exp-of-cond-var-bound} and \eqref{eq:var-of-cond-expectation-bound}, we get
\[
\begin{aligned}
    v^{\theta, \eps}_{s, i, j}
    &\leq  \frac{1}{\eps^2\mathbb P \left(X_t^\theta = x^{(i)} \right)} \left\{\left( \sum_{s = 0}^{t-1}\sum_{i^\prime = 1}^d \sqrt{v^{\theta, \eps}_{s, i^\prime, j}} \right)^2
    + \frac{1}{n} \mathbb E \left[ \left(\sum_{s = 0}^{t-1} \sum_{i^\prime = 1}^d  \left|\widehat{\nabla_\theta l_s^\theta}\right|_{i^\prime,j}  + \eps C\right)^2 \right] \right\} \\
    & \leq \frac{1}{\eps^2\mathbb P \left(X_t^\theta = x^{(i)} \right)} \left\{d T \sum_{s = 0}^{t-1}\sum_{i^\prime = 1}^d v^{\theta, \eps}_{s, i^\prime, j}
    + \frac{1}{n} \mathbb E \left[ \left(\sum_{s = 0}^{t-1} \sum_{i^\prime = 1}^d  \left|\widehat{\nabla_\theta l_s^\theta}\right|_{i^\prime,j}  + \eps C\right)^2 \right] \right\} \\
    & \leq \frac{1}{\eps^2\mathbb P \left(X_t^\theta = x^{(i)} \right)} \left\{d T \sum_{s = 0}^{t-1}\left\| \Sigma_s^{\theta, \eps} \right\|
    + \frac{1}{n} \mathbb E \left[ \left(\sum_{s = 0}^{t-1} \sum_{i^\prime = 1}^d  \left|\widehat{\nabla_\theta l_s^\theta}\right|_{i^\prime,j}  + \eps C\right)^2 \right] \right\}.
\end{aligned}
\]
\noindent
Note that, given the definition of $\widehat{\nabla_\theta l_t^\theta}$ \eqref{eq:logit-gradient-estimator} and Assumption \ref{assumption:finite-state-space}, it is clear that $\widehat{\nabla_\theta l_t^\theta}$ has finite second moment for a fixed $\eps > 0$.
By taking a sum over $i = 1, \ldots, d$, we get the recursive bound
$\| \Sigma_t^{\theta, \eps}\| \leq  c_1 \sum_{s=0}^{t-1}  \| \Sigma_s^{\theta, \eps} \| + c_6$, 
where
\begin{equation}
    \begin{aligned}
c_1 &\coloneq \frac{dT}{\eps^2} \sup_{0 \leq t \leq T} \sum_{i=1}^d \mathbb P \left( X_t^\theta = x^{(i)} \right)^{-2} > 0 , \\
c_6 &\coloneq \frac{1}{n \eps^2}\sup_{0 \leq t \leq T} \mathbb E \left[ \left(\sum_{s = 0}^{t-1} \sum_{i^\prime = 1}^d  \left|\widehat{\nabla_\theta l_s^\theta}\right|_{i^\prime,j}  + \eps C\right)^2 \right] \sum_{i=1}^d \mathbb P \left( X_t^\theta = x^{(i)} \right)^{-2}  \in (0,\infty).
\end{aligned}
\end{equation}
We can apply Lemma \ref{lemma:discrete-gronwall} with $c_2 = c_3 = c_4 = c_5  = 0$ to get
\[
\left\| \Sigma^{\theta, \eps}_t \right\| \leq \frac{\gamma_{1, \eps}}{n \eps^2}  \left(1 + \frac{\gamma_{2} T}{\eps^2}\right)^t ,
\]
where 
\begin{equation}
    \label{eq:constants-variance-logit-gradient-estimator}
    \begin{aligned}
    \gamma_{1, \eps} &= \frac{3}{2} \sup_{0 \leq t \leq T} \mathbb E \left[ \left(\sum_{s = 0}^{t-1} \sum_{i^\prime = 1}^d  \left|\widehat{\nabla_\theta l_s^\theta}\right|_{i^\prime,j}  + \eps C\right)^2 \right] \sum_{i=1}^d \mathbb P \left( X_t^\theta = x^{(i)} \right)^{-2} , \\
    \gamma_2 &= d \sup_{0 \leq t \leq T} \sum_{i=1}^d \mathbb P \left( X_t^\theta = x^{(i)} \right)^{-2}.
    \end{aligned}
\end{equation}
which concludes the proof. 
\end{proof}

\begin{proof}[Proof of Theorem \ref{theorem:main-mse-result}]
    Fix $j \in \{1, \ldots, D\}$. We use the bias-variance decomposition to write
    \[
\mathbb E \left[\left| \widehat{\nabla V}(\theta)_j - \nabla_\theta V (\mu, \theta)_j \right|^2 \right]  = \var \left( \widehat{\nabla V}(\theta)_j \right) + \left| \mathbb E \left[ \widehat{\nabla V}(\theta)_j \right] - \nabla_\theta V (\mu, \theta)_j \right|^2.
    \]
    The bias term can be upper bounded using Theorem \ref{theorem:main-bias-result}.
    We now focus on the variance term. We have, following the same reasoning as in the proof of Proposition \ref{proposition:l2-logit-gradient-estimator}
    \[
\begin{aligned}
    \var \left[\widehat{\nabla V}(\theta)_j\right] &= 
    \frac{1}{N} \var \left(  \sum_{t=0}^{T}  \tilde Q_{t,j}^{\theta, \eps} + \tilde S_{t,j}^{\theta, \eps} \right) \\
    & \leq \frac{1}{N} \left( \sum_{t=0}^{T} \left( \var \left( \tilde Q_{t,j}^{\theta, \eps} + \tilde S_{t,j}^{\theta, \eps} \right) \right)^{1/2} \right)^2 \\
    & \leq \frac{1}{N} \left( \sum_{t=0}^{T} \left(2 \var \left( \tilde Q_{t,j}^{\theta, \eps} \right) + 2 \var \left( \tilde S_{t,j}^{\theta, \eps} \right) \right)^{1/2} \right)^2,
\end{aligned}
\]
where 
\[
\tilde Q_{t,j}^{\theta, \eps} \coloneq \eps^{-1} G_t^{\theta, \eps} \left(\Lambda_t \widehat{\nabla_\theta l^\theta_t}\right)_j,
\quad \tilde S_{t,j}^{\theta, \eps}  \coloneq \mathbf 1_{t < T} G_t^{\theta, \eps} \nabla_\theta \log \tilde p (\theta, \alpha_t^{\theta, \eps}, t,Y_t^{\theta, \eps}, l_t^\theta + \eps \Lambda_t)_j,
\]
We have
$\var \left( \tilde S_{t,j}^{\theta, \eps} \right) \leq C^2 (T + 1)^2 M_0^2$
since $\nabla_\theta \log p$ is bounded and $|G_t^{\theta, \eps}| \leq (T + 1)M_0$ by Assumption \ref{assumption:finite-state-space}.
For $\var (\tilde Q_{t,j}^{\theta, \eps})$, we have, for all $i^\prime = 1, \ldots, d$,
\[
\begin{aligned}
    \var \left( G_t^{\theta, \eps}\left(\Lambda_t\right)_{i^\prime} \left(\widehat{\nabla_\theta l_t^\theta}\right)_{i^\prime,j} \right) 
    &= \mathbb E \left[ G_t^{\theta, \eps}\left(\Lambda_t\right)_{i^\prime} \right]^2 \var \left[ \left(\widehat{\nabla_\theta l_t^\theta}\right)_{i^\prime,j} \right] \\
    & \quad + \mathbb E \left[ \left(\widehat{\nabla_\theta l_t^\theta} \right)_{i^\prime, j}\right]^2 \var \left( G_t^{\theta, \eps}\left(\Lambda_t\right)_{i^\prime} \right) \\
    & \quad + \var \left[ \left(\widehat{\nabla_\theta l_t^\theta}\right)_{i^\prime,j} \right] \var \left( G_t^{\theta, \eps}\left(\Lambda_t\right)_{i^\prime} \right). 
\end{aligned}
\]
Notice that
\[
\mathbb E \left[   G_t^{\theta, \eps}\left(\Lambda_t\right)_{i^\prime}\right]^2 \leq  (T + 1)^2 M_0^2,
\quad \var \left( G_t^{\theta, \eps}\left(\Lambda_t\right)_{i^\prime} \right) \leq (T + 1)^2 M_0^2.
\] 
Recall that $v^{\theta, \eps}_{t, i^\prime,j} = \var \left[ \left(\widehat{\nabla_\theta l_t^\theta}\right)_{i^\prime,j} \right]$, we have thus shown that
\[
\begin{aligned}
\var \left( \tilde Q_{t,j}^{\theta, \eps, t} \right) &\leq \eps^{-2} (T + 1)^2 M_0^2 \left(\sum_{i^\prime=1}^d \left(2 v^{\theta, \eps}_{t, i^\prime,j} + \mathbb E \left[ \left(\widehat{\nabla_\theta l_t^\theta} \right)_{i^\prime, j}\right]^2 \right)^{1/2} \right)^2 \\
& \leq \eps^{-2}  (T + 1)^2 M_0^2 \left( \sum_{i^\prime = 1}^{d} \sqrt{v^{\theta, \eps}_{t,i^\prime, j}}  + \frac{1}{2}\left(E_t(\theta, \eps) + \left\| \nabla_\theta l_t^\theta \right\| \right)\right)^2,
\end{aligned}
\]
where $E_t(\theta, \eps)$ is defined in \eqref{eq:bias-notation}.
Therefore, we get
\[
\begin{aligned}
   \var \left[\widehat{\nabla V}(\theta)_j\right] & \leq \frac{2 M_0^2}{N} \left( 
        \sum_{t=0}^{T} (T + 1) \left(\eps^{-1} \left( \sum_{i^\prime = 1}^{d} \sqrt{v^{\theta, \eps}_{t,i^\prime, j}} 
        + E_t(\theta, \eps) / 2 + \left\| \nabla_\theta l_t^\theta \right\| / 2 \right) + C \right)
        \right)^2 \\
        & \leq \frac{2 (T+1)^3 M_0^2}{N}  \left( 
        \sum_{t=0}^{T}  2 \eps^{-2} \left(\sum_{i^\prime = 1}^{d} \sqrt{v^{\theta, \eps}_{t,i^\prime, j}}\right)^2
        + \left(\eps^{-1}\left(E_t(\theta, \eps) + \left\| \nabla_\theta l_t^\theta \right\|\right) + C\right)^2 
        \right) \\ 
        &\leq \frac{2 (T+1)^3 M_0^2}{N}  \left( 
        \sum_{t=0}^{T}  2 \eps^{-2} d \left\| \Sigma_s^{\theta, \eps} \right\|
        +  \left(\eps^{-1}\left(E_t(\theta, \eps) + \left\| \nabla_\theta l_t^\theta \right\|\right) +  C\right)^2 
        \right) \\
        & \leq  \frac{\delta_1}{N \eps^2} \sum_{t=0}^{T}  \left\| \Sigma_t^{\theta, \eps} \right\| + \frac{\delta_{2, \eps}}{N \eps^2},
\end{aligned}
\]
where
\[
\begin{aligned}
\delta_1 = 4 (T + 1)^3 M_0^2 d, \quad \delta_{2, \eps} = 2 (T + 1)^3 M_0^2 \sum_{t = 0}^T \left(E_t(\theta, \eps) + \left\| \nabla_\theta l_t^\theta \right\| +  \eps C\right)^2.
\end{aligned}
\]
 Therefore, we have the following bound on the MSE
\[
\begin{aligned}
    \mathrm{MSE}(\theta, N, n, \eps) &\leq   \frac{\delta_1}{N \eps^2} \sum_{t=0}^{T}  \left\| \Sigma_t^{\theta, \eps} \right\| + \frac{\delta_{2, \eps}}{N \eps^2} + \eps^{2}\left( k_{1, \eps} \beta_\eps^T - k_{2, \eps} \beta^T + k_{3, \eps} + k_{4, \eps} \eps \right)^2 \\
    & \leq \frac{\delta_1 \gamma_{1, \eps}}{nN \eps^4} \sum_{t=0}^T  \left(1 + \frac{\gamma_{2} T}{\eps}\right)^t + \frac{\delta_{2, \eps}}{N \eps^2} + \eps^2 \left( k_{1} \beta_\eps^T - k_{2} \beta^T + k_3 + k_4 \eps \right)^2 \\
    & \leq \frac{\delta_1 \gamma_{1, \eps}}{nN \eps^4} \frac{\left(1 + \gamma_2 T / \eps\right)^{T+1} - 1}{\gamma_2 T / \eps}  + \frac{\delta_{2, \eps}}{N \eps^2} + \eps^2\left( k_{1}  \beta_\eps^T - k_{2} \beta^T + k_3 + k_4 \eps \right)^2. \\ 
\end{aligned}
\]
Notice that, for $\delta > 0$
\[
\frac{(1 + \delta)^{T + 1} - 1}{\delta} = \sum_{t=1}^{T+1} \delta^{t-1} \binom{T+1}{t} = \sum_{t=0}^T \delta^t \binom{T + 1}{t + 1},
\]
Hence we get 
\[
\mathrm{MSE}(\theta, N, n, \eps) \leq \mathcal K_{1, \eps} \eps^2 + \frac{\mathcal K_{2, \eps}}{ N \eps^2} + \frac{\mathcal K_{3, \eps}}{n N \eps^4}  \sum_{t = 0}^T  \mathcal{C}_t\eps^{-t},
\]
with the following constants
\[
\begin{aligned}
  \mathcal K_{1, \eps} &= \left( k_{1,\eps}  \beta_\eps^T - k_{2,\eps} \beta^T + k_{3,\eps} + k_{\eps} \eps \right)^2, \\ 
  \mathcal K_{2, \eps} &= \delta_{2, \eps} =  2 (T + 1)^3 M_0^2 \sum_{t = 0}^T \left(E_t(\theta, \eps) + \left\| \nabla_\theta l_t^\theta \right\| +  \eps C\right),\\
  \mathcal K_{3,\eps} & = \delta_1 \gamma_{1, \eps} = 6 (T + 1)^3 M_0^2 d  \sup_{0 \leq t \leq T} \mathbb E \left[ \left(\sum_{s = 0}^{t-1} \sum_{i^\prime = 1}^d  \left|\widehat{\nabla_\theta l_s^\theta}\right|_{i^\prime,j}  + \eps C\right)^2 \right] \sum_{i=1}^d \mathbb P \left( X_t^\theta = x^{(i)} \right)^{-2}, \\
  \mathcal C_t &= \binom{T + 1}{t + 1}\left( \gamma_2 T \right)^t = \binom{T + 1}{t + 1}\left(   d T \sup_{0 \leq t \leq T} \sum_{i=1}^d \mathbb P \left( X_t^\theta = x^{(i)} \right)^{-2} \right)^t,
\end{aligned}
\]
which concludes the proof.
\end{proof}

\begin{small} 

\bibliographystyle{plain}
\bibliography{references}

\end{small} 

\end{document}